\documentclass[3p,11pt]{elsarticle}
\usepackage{graphicx,subfig}
\usepackage[dvipsnames]{xcolor}
\usepackage{tikz}
\usepackage{stmaryrd}
\usepackage{commath}

\usepackage{url}
\usepackage{xcolor}
\usepackage{amsmath}
\usepackage{amssymb}
\usepackage{mathrsfs}
\usepackage[numbers]{natbib}
\usepackage[hidelinks]{hyperref}
\usepackage{booktabs}
\usepackage{caption}
\hypersetup{
	colorlinks   = true, %Colours links instead of ugly boxes
	urlcolor     = red, %Colour for external hyperlinks
	linkcolor    = red, %Colour of internal links
	citecolor   = red %Colour of citations
}
\usepackage{cleveref}
\usepackage{enumitem}

\makeatletter
\def\ps@pprintTitle{%
	\let\@oddhead\@empty
	\let\@evenhead\@empty
	\def\@oddfoot{}%
	\let\@evenfoot\@oddfoot}

\newcommand{\Bk}{\color{black}}

\newcommand{\tnorm}{\@ifstar\@tnorms\@tnorm}
\newcommand{\@tnorms}[1]{%
	\left|\mkern-1.5mu\left|\mkern-1.5mu\left|
	#1
	\right|\mkern-1.5mu\right|\mkern-1.5mu\right|
}
\newcommand{\@tnorm}[2][]{%
	\mathopen{#1|\mkern-1.5mu#1|\mkern-1.5mu#1|}
	#2
	\mathclose{#1|\mkern-1.5mu#1|\mkern-1.5mu#1|}
}
\newcommand{\jump}[1]{[ #1 ]}
\newcommand{\av}[1]{\{ #1 \}}
\newcommand{\bm}[1]{\boldsymbol{#1}}
\newcommand{\mesh}{\mathcal{E}_h}
\newcommand{\DG}{\mathrm{DG}}

\newtheorem{lemma}{Lemma}
\newdefinition{definition}{Definition}
\newtheorem{theorem}{Theorem}
\newtheorem{corollary}{Corollary}
\newtheorem{proposition}{Proposition}
\newdefinition{remark}{Remark}
\newproof{proof}{Proof}
\begin{document}
	%------------------------------------------------------------------------------
	\begin{frontmatter}	
		\title{Numerical analysis of a hybridized discontinuous Galerkin method for the Cahn--{H}illiard 
			problem}
		
		\author[KLAK]{Keegan L. A. Kirk\fnref{label1}}
		\ead{klk12@rice.edu}
		\fntext[label1]{\url{} KK acknowledges the support of the Natural Sciences and Engineering 
		Research Council of Canada (NSERC) via MSFSS 566831 and PDF 568008.}

		\author[RM]{Rami Masri\fnref{label2}}
		\ead{rami@simula.no}
		\fntext[label2]{\url{} RM acknowledges support and funding from the Research Council of Norway (RCN) via FRIPRO grant agreement 324239 (EMIx). }

		\author[KLAK]{Beatrice Riviere \fnref{label3}}
		\ead{riviere@rice.edu}
		\fntext[label3]{\url{} BR acknowledges support from the National Science Foundation via DMS 2111459 and DMS 1913291.}
		
		\address[KLAK]{Department of Computational and Applied Mathematics, Rice University, Houston, 
			USA} 
		\address[RM]{Department of Numerical Analysis and Scientific Computing, Simula Research 
			Laboratory, Oslo, Norway}

		\begin{abstract}
		The mixed form of the Cahn-Hilliard equations is discretized by the hybridizable discontinuous Galerkin method. For any 
chemical energy density, existence and uniqueness of the numerical solution is obtained.  The scheme is proved to
be unconditionally stable. Convergence of the method is obtained by deriving a priori error estimates that are valid for the
Ginzburg-Lindau chemical energy density and for convex domains. The paper also contains discrete functional tools, namely discrete 
Agmon and Gagliardo-Nirenberg inequalities, which are proved to be valid in the hybridizable discontinuous Galerkin spaces.  
		\end{abstract}
		
		\begin{keyword}
		\end{keyword}
		
	\end{frontmatter}
	
	%------------------------------------------------------------------------------
	
	\section{Introduction}
	%\subsection{Related results}
	
The Cahn--{H}illiard equation was originally proposed in \cite{cahn1958free} as a model for phase 
separation in binary alloys. Since then, it has become fundamental to the phase field theory for moving 
interface problems.  Some notable applications include tumor growth 
\cite{medina2022stabilized,agosti2017cahn} 
and
multi-phase flows \cite{fu2020CHNS,liu2020efficient,liu2020priori}.
%To name a few, this 
%	equation provides a good model for spinodal decomposition \cite{elliott1989cahn,liu2020efficient} 
%	and for avascular tumor growth \cite{medina2022stabilized,agosti2017cahn}. 
In its primal form, the Cahn--Hilliard equation is a fourth order nonlinear parabolic equation; thus, its 
numerical 
approximation 
presents a significant computational challenge. The conforming finite element 
approximation of 
fourth order 
elliptic operators is difficult, as the natural functional setting for weak solutions demands $H^2$ 
regularity. 
%In this direction, the conforming Bell's element, the Argyris element, and 
%Hsieh--Clough--Tocher 
%macroelements {\KK have been applied to} Cahn--Hilliard type equations in primal form in two 
%dimensions, but 
The construction and implementation of $H^2$ conforming elements are challenging tasks, especially in 
three dimensions. Alternatives to conforming methods for the Cahn--Hilliard problem in primal form are 
non-conforming 
approximations of $H^2$ such as the Morley element \cite{elliottfrench1989cahnmorley}, or 
$\mathcal{C}^0$ interior-penalty methods \cite{feng2007fully,wells2006discontinuous,aristotelous2015adaptive}. 
The latter approach is 
particularly 
attractive for three dimensional simulations due to ease in which a basis can be constructed in higher 
dimensions.

Due to the difficulty involved with the numerical treatment of higher order derivatives, the mixed 
form of the Cahn--Hilliard equation is often preferred as it involves instead the solution of a coupled 
second order system. To our knowledge, this 
approach to the Cahn--Hilliard problem was first considered in \cite{elliottfrench1989cahn} with 
classical $\mathcal{C}^0$ elements. However, it is well known 
that such methods violate the local mass balance satisfied at the continuous level. To this end, 
fully non-conforming discontinuous Galerkin 
(DG) methods are a suitable choice. For the use of DG methods 
in the numerical solution of the mixed Cahn--Hilliard system, we refer to 
\cite{shu2007ldgcahn,shu2017ldgcahn,yan2021ldgcahn} for the LDG method and  
\cite{kay2009discontinuous,aristotelous2013mixed,liu2019numerical} for the IPDG method.

Despite their advantages, DG methods come with a 
significant increase in the number of globally coupled degrees of freedom over conforming $\mathcal{C}^0$ 
elements. To address this additional computational burden, the 
\emph{hybridizable} discontinuous Galerkin (HDG) methods were introduced in \cite{Cockburn:2009}. 
Key to the HDG methods is the introduction of additional unknowns on the mesh skeleton which 
act as Lagrange multipliers enforcing the continuity of the normal component of the
numerical flux. As a consequence, the element unknowns can be eliminated locally through static 
condensation leading to a reduction in the total number of globally coupled degrees of freedom. The 
HDG method has seen success across a wide variety of elliptic and parabolic problems, and has recently 
been applied to the Cahn--Hilliard \cite{medina2022stabilized, chen2023LHDG} and 
Cahn--Hilliard--Navier--Stokes 
systems \cite{fu2020CHNS}. 
Closely related to the HDG method is the hybrid high order (HHO) method introduced in 
\cite{DiPietro:2015}, which has been applied to the Cahn--Hilliard problem in \cite{Chave:2016, 
chave2017HHOconv}. Similar to the HDG method, the HHO method introduces additional degrees of 
freedom 
on the mesh skeleton in order to leverage static condensation to reduce the size of the global system.
The HHO method differs from the use of local reconstruction operators and face-based 
stabilizations. For more information on the ties between the HHO and HDG methods, we refer the 
reader to 
\cite{Cockburn:2016}. 

Regarding the theoretical analysis of HDG methods for the Cahn--Hilliard system, we mention that 
optimal error bounds for the hybridizable LDG method are proven in 
\cite{chen2023LHDG}. However, 
to the best of our knowledge, the theoretical analysis of the hybridized IPDG methods in 
\cite{medina2022stabilized, fu2020CHNS} is missing from the literature.
%, though the analysis for the 
%related Allen--Cahn equation was established in \cite{fabien2022numerical} where a 
%Lipschitz continuity assumption on the first order derivative of the potential function is assumed. This 
%is 
%the main purpose of this 
%work. 
We 
remark that the additional facet unknowns in our HDG scheme precludes the use of discrete functional 
analysis tools used 
previously 
in \cite{kay2009discontinuous,aristotelous2013mixed} to analyze the IPDG method. Fortunately, 
appropriate 
analogues of these tools have been extended to the HHO setting in \cite{Chave:2016}, and using 
similar
techniques we will show that they hold also in the HDG setting.
%Recently, the analyses 
%of HDG applied to the Allen--Cahn equation was established in \cite{fabien2022numerical} where a 
%Lipschitz continuity assumption on the first order derivative of the potential function is assumed. In 
%our work, we avoid this assumption which only holds after a regularization and extension of the 
%potential function. 

%Considering dG for the CH equation written as a second order system, the symmetric interior 
%penalty dG in 2D  was analyzed in \cite{kay2009discontinuous} (with the inclusion of convection) 
%and in 3D in \cite{aristotelous2013mixed}. The non-symmetric interior penalty dG in 2D and in 3D is 
%theoretically studied in \cite{liu2019numerical}. 

% Motivated 
%by 
%	such considerations, we theoretically analyze the hybridized discontinuous Galerkin (HDG) method 
%	applied to this problem which reduces the total number of degrees of freedom (as  compared to 
%	standard interior penalty DG methods) while maintaining the aforementioned properties. 
%	
%	The literature on numerical methods and their convergence and stability properties applied to the 
%CH 
%	equation is vast. Here, we focus on the analyses of discontinuous Galerkin (DG) schemes.  The 
%	application of interior penalty dG to the fourth order equation is analyzed and  a priori error 
%	estimates are derived in \cite{feng2007fully,wells2006discontinuous,aristotelous2015adaptive}. 

%	
%	The HDG method have been recently applied to the CH equation in \cite{medina2022stabilized} 
%	where several numerical experiments demonstrated the efficiency of the method.
	
	The main contributions of this paper are (i) the unconditional unique solvability 
	(Theorem~\ref{thm:solvability}) of HDG established by using the Minty--Browder and  Brouwer fixed 
	point theorems, (ii) the unconditional stability for any $\mathcal{C}^2$ potential function 
	(Theorem~\ref{thm:stability}), (iii) the $L^\infty$ stability of the order parameter 
	(Theorem~\ref{thm:linf_stability}) established for convex domains and for the Ginzburg--Landau 
	potential without any regularization, truncation, or extension, and (iv) optimal a priori error estimates 
	in the broken $H^1$ norm (Theorem~\ref{theorem:error_estimate}) for the Ginzburg--Landau 
	potential and for convex domains. 
	It is worth noting that a major challenge in the analysis is to avoid the use of any modifications or 
	assumptions on the potential function such as the ones made in 
	\cite{liu2019numerical,barrett1999finite}. 
	Following similar strategies to 
	\cite{kay2009discontinuous, DiPietro:2017}, we successfully avoid such assumptions by proving 
	discrete Agmon and Gagliardo--Nirenberg inequalities in the HDG setting, which are useful stand-alone results that
can be applied to other problems.
%	We think these 
%	inequalities are another contribution of our work as they may be independently useful in the analysis 
%	of HDG methods applied to different problems. 
	
	\section{Model problem, notation and preliminaries}
	Let $\Omega \subset \mathbb{R}^d$, $d=2,3$ be a bounded, open, polygonal ($d=2$) 
	or 
	polyhedral ($d=3$) domain with outer unit normal $\boldsymbol{n}$. We consider the fourth 
	order 
	Cahn--Hilliard 
	equation, rewritten as a 
	second order system: find a pair $(c,\mu)$ satisfying  
	\begin{subequations} \label{eq:CH_strong}
		\begin{align}
			\partial_t c - \Delta \mu & = 0, && \text{ in } \Omega \times (0,T), \label{eq:CH_strong_a} \\
			\mu &= \Phi'(c) - \kappa \Delta c, && \text{ in } \Omega \times (0,T), \\
			\nabla c \cdot \boldsymbol{n} &=0, && \text{ on } \partial \Omega \times (0,T), \\
			\nabla \mu \cdot \boldsymbol{n} &=0, && \text{ on } \partial \Omega \times (0,T), 
			\label{eq:CH_strong_d} \\
			c &= c_0, && \text{ on } \Omega \times \cbr{0}. \label{eq:CH_strong_e}
		\end{align}
	\end{subequations}
	We assume that the scalar potential function $\Phi$ admits a concave-convex 
	decomposition; it suffices to assume
	$\Phi \in \mathcal{C}^2$. In other words, we can write
	\begin{equation}
		\Phi(c) = \Phi_+(c) +  \Phi_-(c),
	\end{equation}
	with $\Phi_{+}$ convex and $\Phi_{-}$ concave.
	% \section{Preliminaries}
	
	\subsection{Basic results on broken Sobolev and polynomial spaces}

Let $\mathcal{E}_h$ be a conforming shape-regular mesh of $\Omega$, made of simplices $E$ with boundary $\partial E$
and diameter $h_E$. The mesh size is $h = \max_{E \in \mathcal{E}_h} h_E$.  
Let $\Gamma_h^0$ (resp $\Gamma_h^b$) be the set of 
interior (resp. boundary) faces
and let $\Gamma_h = \Gamma_h^0 \cup \Gamma_h^b$.  
Further, let $\mathcal{F}_E$ denote the set of all the faces of an element $E \in \mathcal{E}_h$. 
We assume that the family $\cbr{\mathcal{E}_h}_{h > 0}$ is
quasi-uniform.

%		Moreover, in the course of our analysis we will require the further assumption that the family 
%		$\cbr{\mathcal{E}_h}_{h > 0}$ is \emph{quasi-uniform}; i.e., for each $h>0$ there exists a 
%		$C>0$ such that $h \le C h_E$  for all $E \in \mathcal{E}_h$.
		
		Let $k \ge 1$ be a fixed integer. We introduce a pair of broken polynomial spaces on 
		$\mathcal{E}_h$:
		\begin{align}
			S_h &= \cbr{ v \in L^2(\Omega) \, : \, \forall E \in \mathcal{E}_h, v|_{E} \in \mathbb{P}_k(E)}, \\
			M_h &= S_h \cap L_0^2(\Omega).
		\end{align}
		Moreover, the HDG method requires the following broken polynomial space defined on 
		$\Gamma_h$:
		\begin{align}
			\hat{S}_h &= \cbr{ v_h \in L^2(\Gamma_h) \, : \, \forall e \in \Gamma_h, v|_{e} \in 
				\mathbb{P}_k(e) 
			}. 
		\end{align}
		Here, $L_0^2(\Omega)$ is the zero mean value subspace of $L^2(\Omega)$ and 
		$\mathbb{P}_k(\mathcal{O})$ denotes the space 
		of 
		polynomials of degree less than or equal to $k$ defined on the open set $\mathcal{O}$.

		As the spaces $S_h$ and $M_h$ are non-conforming, we introduce
		the \emph{broken gradient operator} $\nabla_h v_h$ by the restriction
		$(\nabla_h v_h)|_{E} = \nabla (v_h|_{E})$.
		Moreover, the trace of a function $v_h \in S_h$ may be double-valued
		on interior facets. To each  interior facet $e \in \Gamma_h^0$, 
we associate a unique normal vector $\mathbf{n}_e$ and denote by $E_+$ and $E_-$ the neighboring elements of $e$ such 
that $\mathbf{n}_e$ points from $E_-$ to $E_+$. 
%We 
%		denote the traces of $v_h \in S_h$ on $e$
%		by $v_{h}|_{E_+} = \text{ trace of } v_h|_{+} \text{ on } e$ and $v_{h}|_{E_-} = \text{ trace of } 
%		v_h|_{E_-} 
%		\text{ on } e$. 
We introduce the \emph{jump} $\jump{\cdot}$ and  \emph{average} $\av{\cdot}$ of 
		$v_h \in S_h$ 
		across an interior facet $e \in \Gamma_h^0$ as follows: let $\jump{v_{h}} = v_{h}|_{E_+} - 
		v_{h}|_{E_-}$ 
		and $\av{v_{h}} = (v_{h}|_{E_+} + v_{h}|_{E_-})/2$. On boundary faces $e \in 
		\Gamma_h^b$, we set $\jump{v_h} = \av{v_h} =  v_h|_{E} \text{ on } 
		e$, where $E$ is the element such that
		$e\subset \partial E\cap  \partial \Omega$.
		
		We adopt the following notation for various product spaces of interest
		in this work: 
		\begin{equation*}
			\boldsymbol{S}_h = S_h \times \hat{S}_{h}, \quad \boldsymbol{M}_h = M_h \times \hat{S}_h.
		\end{equation*}
		Pairs in these product spaces will be denoted
		using boldface; for example, $\boldsymbol{s}_h = 
		(s_h,\hat{s}_h) \in \boldsymbol{S}_h$. 
		Throughout we use the notation $a \lesssim b$ to denote $a \le C b$
		where $C$ is a generic constant independent of the mesh parameters $h$  and
		$\tau$, but possibly dependent on the polynomial
		degree $k$, the spatial dimension $d$, and the domain
		$\Omega$.
		%

	%	Let $\mathcal{E}_h = \cbr{E_k}$ be a family of conforming regular (in the sense of Ciarlet 
	%	\textcolor{red}{CITE}) partition of $\overline{\Omega }
	%	\subset \mathbb{R}^d$ into simplices or quadrilaterals ($d=2$) or hexahedra ($d=3$). 
	
	%	= 
	%	\max_{E \in \mathcal{E}_h} h_E$ where $h_E = \text{diam}(E)$. 
Given an 
	integer $s \ge 1$, we define the broken Sobolev space:
	\begin{equation}
		H^s(\mathcal{E}_h) = \cbr{ v \in L^2(\Omega) \, : \, \forall E \in \mathcal{E}_h, \quad v|_{E} \in 
			H^s(E)}.
	\end{equation}
	%	
	%	Next, for any fixed integer $k \ge 1$ we define the broken polynomial spaces
	%	%
	%	\begin{align}
	%		S_h &= \cbr{ v \in L^2(\Omega) \, : \, \forall E \in \mathcal{E}_h, v|_{E} \in \mathbb{P}_k(E)}, \\
	%		M_h &= S_h \cap L_0^2(\Omega),
	%	\end{align}
	%	%
	%	where $L_0^2(\Omega) = L^2(\Omega) / \mathbb{R}$ and $\mathbb{P}_k(E)$ denotes the space 
	%of 
	%	polynomials of degree less than or equal to $k$ in the case of simplices and 
	%	\textcolor{red}{quads/hexes...}
	%
	Define $j_{0}, j_{1}: \boldsymbol{S}_h \times \boldsymbol{S}_h \rightarrow \mathbb{R}$ as follows 
	\begin{align} 
		j_{0} (\bm{u}_h, \bm{v}_h) & =  \sum_{E \in \mathcal{E}_h}  	h_E\int_{\partial E}  (u_h - 
		\hat{u}_h)(v_h - \hat{v}_h) \dif s, \label{eq:def_j0h} \\ 
		j_{1}(\bm{u}_h, \bm{v}_h) & =  \sum_{E \in \mathcal{E}_h}  \frac{1}{h_E} \int_{\partial E} (u_h 
		-\hat{u}_h)(v_h - \hat{v}_h) \dif s. \label{eq:def_j1h}  
	\end{align}  
	We equip $\boldsymbol{S}_h$ with the following (semi-)inner-products:
	\begin{align} \label{eq:product_space_inner_product_l2}
		(\boldsymbol{u}_h,\boldsymbol{v}_h)_{0,h} &= \int_{\Omega} u_h v_h \dif x + j_{0}(\bm{u}_h, 
		\bm{v}_h) , \\
		(\boldsymbol{u}_h, \boldsymbol{v}_h)_{1,h} &= \sum_{E \in \mathcal{E}_h} \int_{E} \nabla u_h 
		\cdot \nabla v_h \dif x + j_{1}(\bm{u}_h, \bm{v}_h) , \label{eq:product_space_inner_product_h1}
	\end{align}	
	as well as their induced (semi-)norms:
	\begin{align} \label{eq:product_space_norm_l2}
		\norm{\boldsymbol{v}_h}_{0,h} &= \del[3]{\norm{
				v_h}_{L^2(\Omega)}^2 + 
			\sum_{E 
				\in \mathcal{E}_h} h_E \norm{ v_h - \hat{v}_h}_{L^2(\partial E)}^2 }^{1/2}, \quad \forall 
		\boldsymbol{v}_h \in \boldsymbol{S}_h, \\
		\norm{\boldsymbol{v}_h}_{1,h} &= \del[3]{\sum_{E \in \mathcal{E}_h} \norm{\nabla 
				v_h}_{L^2(E)}^2 + 
			\sum_{E 
				\in \mathcal{E}_h} \frac{1}{h_E} \norm{ v_h - \hat{v}_h}_{L^2(\partial E)}^2 }^{1/2}, \quad 
		\forall 
		\boldsymbol{v}_h \in \boldsymbol{S}_h. \label{eq:product_space_norm_h1}
	\end{align}
	The space $\boldsymbol{M}_h$ equipped with $(\cdot,\cdot)_{1,h}$ is an inner-product space.  We 
	note that
	$\norm{\cdot}_{1,h}$  is not a norm on $\boldsymbol{S}_h$. 
	%, but it is a norm on $M_h$, since $\norm{\alpha}_{1,h} = 0$ for any
	%constant $\alpha$.  $\norm{\cdot}_{0,h}$ is a norm on $\boldsymbol{S}_h$. \\
	From the definition of $j_{0}(\cdot, \cdot)$ and the fact that $h_E \leq h$ , we have that  
	\begin{equation}
		j_{0}(\bm{v},\bm{v}) \leq h^2 \sum_{E 
			\in \mathcal{E}_h} \frac{1}{h_E} \norm{ v - \hat{v}}_{L^2(\partial E)}^2 \leq h^2 \|\bm{v} 
		\|_{1,h}^2 , \quad 
		\forall 
		\boldsymbol{v} = (v,\hat{v}) \in H^1(\mathcal{E}_h)\times L^2(\Gamma_h). 
	\end{equation}
	Therefore we have for any $\boldsymbol{v} = (v,\hat{v}) \in H^1(\mathcal{E}_h)\times L^2(\Gamma_h)$
	\begin{align}
		\norm{\boldsymbol{v}}_{0,h}^2 = \norm{v}_{L^2(\Omega)}^2 + 
		j_0(\boldsymbol{v},\boldsymbol{v})
		&\leq \norm{v}_{L^2(\Omega)}^2 + h^2 \Vert \boldsymbol{v} \Vert_{1,h}^2, \\ 
		|j_0(\bm{u},\bm{v})| \leq j_0(\bm{u},\bm{u})^{1/2} j_0(\bm{v}, \bm{v})^{1/2} & \leq  h^2 \|\bm{u}\|_{1,h} \|\bm{v}\|_{1,h},\\
		|j_{0} (\bm{u}, \bm{v}) | \leq j_{0}  (\bm{u}, \bm{u})^{1/2} j_{0}  (\bm{v}, 
		\bm{v})^{1/2} &\leq  h \norm{\bm{u}}_{0,h}\norm{\bm{v}}_{1,h} .   \label{eq:bound_j0_1h} 
	\end{align}
	Moreover, on the broken Sobolev space $H^2(\mathcal{E}_h) \times  L^2 (\Gamma_h)$ we 
	introduce another norm (with $\mathbf{n}_E$ denoting the unit outward normal vector to $\partial E$):
	\begin{equation}
		\norm{\boldsymbol{v}}_{1,h,\star} = \del[3]{\norm{\boldsymbol{v}}_{1,h}^2 + \sum_{E \in 
				\mathcal{E}_h} 
			h_E \norm{\nabla v \cdot \mathbf{n}_E}_{L^2(\partial E)}^2}^{1/2}.
	\end{equation}
	It is equivalent to the norm $\Vert \cdot\Vert_{1,h}$ on $\boldsymbol{S}_h$ using trace and 
	inverse inequalities:
	\begin{equation}\label{eq:equivnorms}
		\Vert \boldsymbol{v}_h \Vert_{1,h} \leq \norm{\boldsymbol{v}_h}_{1,h,\star} \lesssim \Vert 
		\boldsymbol{v}_h \Vert_{1,h}, 
		\quad\forall \boldsymbol{v}_h\in\boldsymbol{S}_h.
	\end{equation}
	We recall the standard DG semi-norm for $H^1(\mathcal{E}_h)$:
	\begin{equation}
		\norm{v}_{\DG} = \bigg(\sum_{E \in \mathcal{E}_h} \norm{\nabla v}_{L^2(\Omega)}^2 + 
		\sum_{e \in \Gamma_h^0 } \frac{1}{h_e} 
		\norm{\jump{v}}_{L^2(e)}^2\bigg)^{1/2}.\label{eq:def_dGnorm2}
	\end{equation}
	The following bound holds: 
	\begin{equation}
		\norm{v}_{\DG} \lesssim \norm{\boldsymbol{v}}_{1,h},  \quad \forall \boldsymbol{v} \in 
		H^1(\mathcal{E}_h)\times L^2(\Gamma_h).  
		\label{eq:1h_norm_bounded_DG}
	\end{equation}
	Indeed, the triangle inequality yields for any $\hat{v}\in L^2(\Gamma_h^0)$
	\begin{equation}
		\norm{\jump{v}}_{L^2(e)}^2 = 	\norm[0]{v|_{E_+} - v|_{E_-} }_{L^2(e)}^2 \le 		
		\norm[0]{v|_{E_+} - \hat{v} }_{L^2(e)}^2 + \norm[0]{v|_{E_-} - \hat{v} 
		}_{L^2(e)}^2.
	\end{equation}
	We sum  over all interior faces and use that $h_e \approx h_E$ which results from the shape and 
	contact regularity of the mesh (see Lemma 1.41 and Lemma 1.42 in \cite{Pietro:book}).  %We also 
	%note 
	%that $\text{card}\cbr{\mathcal{F}_{E}}$ is uniformly bounded with respect to $h$ for all $E \in 
	%\mathcal{F}_E$. 
	Thus, we  obtain 
	\begin{equation} \nonumber
		\begin{split}
			\sum_{e \in \Gamma_h^0} \frac{1}{h_e} \norm{\jump{v}}_{L^2(e)}^2 &\le \sum_{e \in 
				\Gamma_h^0} \frac{1}{h_e}
			\del{	
				\norm[0]{v|_{E_+} - \hat{v} }_{L^2(e)}^2 + \norm[0]{v|_{E_-} - \hat{v} 
				}_{L^2(e)}^2} \\ &\le
			\sum_{E \in \mathcal{E}_h} \sum_{e \in \mathcal{F}_{E} \cap
				\Gamma_h^0 } \frac{1}{h_e}
			\del{	
				\norm[0]{v|_{E_+} - \hat{v} }_{L^2(e)}^2 + \norm[0]{v|_{E_-} - \hat{v} 
				}_{L^2(e)}^2}\\
			&\lesssim
			\sum_{E \in \mathcal{E}_h}  \frac{1}{h_E}
			\del{	
				\norm[0]{v - \hat{v} }_{L^2(\partial E_+)}^2 + \norm[0]{v - \hat{v} 
				}_{L^2(\partial E_-)}^2} \\
			&\lesssim \sum_{E \in \mathcal{E}_h} \frac{1}{h_E} \norm{v - \hat{v}}_{L^2(\partial E)}^2,
		\end{split}
	\end{equation}
	which is sufficient to conclude. 
	%\	
  Throughout the paper, we use the notation $\overline{v}$ to denote the mean value operator of any function $v$:
\[
\overline{v} = \frac{1}{\vert \Omega\vert} \int_\Omega v, \quad \forall v \in L^2(\Omega).
\]
With this notation, we recall Poincar\'{e}'s inequality valid on broken Sobolev spaces.
	\begin{lemma}[Poincar\'{e} inequality in $H^1(\mathcal{E}_h)$] \label{lem:broken_poincare2}
		Let $p^\star$ be the exponent of the Sobolev embedding of $H^1(\Omega)$ into 
		$L^p(\Omega)$ 
		defined by
		\begin{equation}
			\frac{1}{p^\star} = \frac{1}{2} - \frac{1}{d}.
		\end{equation}
		Then, for each $p \le p^\star$, there is a constant $C_P>0$ independent of $h$ such that
			%for any function $\boldsymbol{v} = (v,\hat{v})\in H^1(\mathcal{E}_h)\times 
			%L^2(\Gamma_h)$, 
			%
			\begin{equation}
				\norm{v}_{L^p(\Omega)} \le C_P \del{\norm{\boldsymbol{v}}_{1,h}^2 + |\overline{v}|^2}^{1/2},
\quad \forall \boldsymbol{v} = (v,\hat{v})\in H^1(\mathcal{E}_h)\times L^2(\Gamma_h).
				\label{eq:poincare}
			\end{equation}
			%		Then, for each $p \le p^\star$, there is a constant $C_P>0$ independent of $h$ such 
			%that
			%		for any function $\boldsymbol{v} = (v,\hat{v})\in (H^1(\mathcal{E}_h)\cap 
			%L_0^2(\Omega))\times 
			%		L^2(\Gamma_h)$ 
			%		%
			%		\begin{equation}
			%			\norm{v}_{L^p(\Omega)} \le C_P \norm{\boldsymbol{v}}_{1,h}. 
			%			\label{eq:poincare_zero_avg2}
			%		\end{equation}
			%
	\end{lemma}
	\begin{proof}
			From \cite{Lasis1,Lasis2}, we have for $v\in H^1(\mathcal{E}_h)$
			\[
			\Vert v \Vert_{L^p(\Omega)}^2 \lesssim \Vert v \Vert_{\mathrm{DG}}^2 + |\Psi(v)|^2,
			\]
			with $\Psi \in H^1(\mathcal{E}_h)'$ any bounded linear functional satisfying $\Psi(1) = 1$. 
			Choosing $\Psi : v \mapsto \overline{v}$, the result then follows with \eqref{eq:1h_norm_bounded_DG}.	
		\qed
	\end{proof} 
As a consequence, with \eqref{eq:poincare}, we have
	\begin{equation}
		\norm{\boldsymbol{v}}_{0,h} 
		\leq ( C_P^2 + h^2)^{1/2} \Vert \boldsymbol{v} \Vert_{1,h} \lesssim 
		\norm{\boldsymbol{v}}_{1,h}, \quad \forall \boldsymbol{v} \in (H^1(\mathcal{E}_h)\cap L_0^2(\Omega)) \times L^2(\Gamma_h).
\label{eq:poincare_zeroavg_vector}
	\end{equation}
	From \eqref{eq:poincare_zeroavg_vector}, it is evident that $\|\cdot\|_{1,h}$ defines a norm on 
	$\bm{M}_h$. 
	We now recall important inequalities which hold in the discrete spaces, thanks to the quasi-uniformity assumption on the mesh.
	%\begin{proposition}[Lebesgue embeddings for broken polynomial spaces] 
		%\label{prop:lebesgue_embed}
	%	Assume that the family of meshes $\cbr{\mathcal{E}_h}_{h>0}$ is quasi-uniform and 
%Assume that $1 \le p \le q \le \infty$. Then, it holds that
		%
		\begin{align}
			\norm{v_h}_{L^q(\Omega)} &\lesssim h^{\frac{d}{q} - \frac{d}{p}} \norm{v_h}_{L^p(\Omega)}, 
			\quad \forall v_h \in S_h, \quad 1 \le p \le q \le \infty,  \label{eq:lebesgue_embed}\\
	%\end{proposition}
	%
	%\begin{proposition}[Discrete inverse inequality] \label{prop:inverse_ineq}  For all $E \in 
		%\mathcal{E}_h$, 
		 \label{eq:inverse_ineq}
			\norm{\nabla v_h}_{L^2(E)} &\lesssim h_E^{-1} \norm{v_h}_{L^2(E)}, \quad \forall v_h \in S_h, \quad \forall E\in\mathcal{E}_h, \\
		%If the mesh is assumed to be quasi-uniform, we have that 
			\|\bm{v}_h\|_{1,h} &\lesssim h^{-1} \|\bm{v}_h\|_{0,h}, \quad \forall \bm{v}_h \in \bm{S}_h. 
			\label{eq:inverse_estimate_1h_to_0h}
\end{align}
	%\end{proposition} 
	The proofs of \eqref{eq:lebesgue_embed} and \eqref{eq:inverse_ineq} can be found 
	in \cite{Brenner:book}.  To obtain \eqref{eq:inverse_estimate_1h_to_0h}, we use 
	\eqref{eq:inverse_ineq} and the definition of $j_1$. 
	
	We consider the symmetric discretization of the Laplace operator $-\Delta u$: 
	$a_{\mathcal{D}} : 
	(H^2(\mathcal{E}_h), L^{2}(\Gamma_h)) \times 
	(H^2(\mathcal{E}_h), L^{2}(\Gamma_h)) \to \mathbb{R}$: 
	\begin{multline} 
		a_{\mathcal{D}}(\boldsymbol{u},\boldsymbol{v}) = \sum_{E \in \mathcal{E}_h} \int_{E} \nabla 
		u 
		\cdot 
		\nabla 
		v 
		\dif x
		- \sum_{E \in \mathcal{E}_h} \int_{\partial E} \left( \nabla u \cdot \mathbf{n}_E \del{v - \hat{v}} +
		\del{u - \hat{u}} \nabla v  \cdot \mathbf{n}_E \right)\dif s \\
		+ \sum_{E \in \mathcal{E}_h} \frac{\sigma}{h_E} \int_{\partial E}
		\del{u - \hat{u}}\del{v - \hat{v}} \dif s.
	\end{multline}

	The parameter $\sigma>0$ is a user-specified penalty parameter. We recall the following basic results concerning the bilinear form $a_{\mathcal{D}}$  \cite{Fabien:2020, Rhebergen:2017}:
	\begin{lemma}[Coercivity and continuity]
		Provided the penalty parameter $\sigma > 0$ is chosen sufficiently large, the bilinear form is 
		coercive on $\boldsymbol{S}_h$: there exists a constant $C_\mathrm{coer} > 0$ such that
		\begin{equation} \label{eq:aD_coercive}
			a_{\mathcal{D}}(\boldsymbol{v}_h,\boldsymbol{v}_h) \ge C_\mathrm{coer} 
			\norm{\boldsymbol{v}_h}_{1,h}^2, 
			\quad \forall \boldsymbol{v}_h \in \boldsymbol{S}_h.
		\end{equation}
		Moreover, the bilinear form $a_{\mathcal{D}}$ is continuous on $\boldsymbol{S}_h \times 
		\boldsymbol{S}_h$: there exists a constant $C_\mathrm{cont} > 0$ such that
		\begin{equation} \label{eq:aD_continuous}
			|a_{\mathcal{D}}(\boldsymbol{u}_h, \boldsymbol{v}_h)| \le C_\mathrm{cont} 
			\norm{\boldsymbol{u}_h}_{1,h} 
			\norm{\boldsymbol{v}_h}_{1,h}, \quad \forall \boldsymbol{u}_h, \boldsymbol{v}_h \in 
			\boldsymbol{S}_h,
		\end{equation}
		 and furthermore, there exists a constant $C_b^\star > 0$ such that
			\begin{align}\label{eq:aD_extended_continuous}
				|a_{\mathcal{D}}(\boldsymbol{u} , \boldsymbol{v}_h)| &\le C_b^\star
				\norm{\boldsymbol{u}}_{1,h, \star} 
				\norm{\boldsymbol{v}_h}_{1,h}, \quad \forall \boldsymbol{u} \in H^2(\mathcal{E}_h) \times 
				L^2(\Gamma_h), \forall \boldsymbol{v}_h \in 
				\boldsymbol{S}_h, \\
				\label{eq:aD_extended_continuous2}
				|a_{\mathcal{D}}(\boldsymbol{u} , \boldsymbol{v})| &\le C_b^\star
				\norm{\boldsymbol{u}}_{1,h, \star} 
				\norm{\boldsymbol{v}}_{1,h,\star}, \quad \forall \boldsymbol{u}, \boldsymbol{v} \in 
				H^2(\mathcal{E}_h) \times 
				L^2  (\Gamma_h).
		\end{align} 
	\end{lemma}
	Henceforth, we will always assume $\sigma > 0$ is sufficiently large to ensure the coercivity of 
	the 
	bilinear form $a_{\mathcal{D}}$. 
	%	Therefore, we may introduce an equivalent (semi-)norm on 
	%	$\boldsymbol{S}_h$:(\rami{i don't think this is used anywhere now, maybe remove it ? } ) 
	%	%
	%	\begin{equation} \label{eq:alpha_norm}
	%		\norm{\boldsymbol{s}_h}_{\alpha} = \sqrt{ 
	%a_{\mathcal{D}}(\boldsymbol{s}_h,\boldsymbol{s}_h)}, 
	%		\quad \forall \boldsymbol{s}_h \in \boldsymbol{S}_h.
	%	\end{equation}
	\section{Discrete functional analysis tools} The main goal of this section is to show discrete 
	counter-parts to the Agmon (\Cref{lem:disc_agmon}) and Gagliardo--Nirenberg 
	(\Cref{lemma:discrete_GN}) inequalities in the HDG setting 
	following ideas in 
	\cite{Chave:2016,Kay:2009}. Such inequalities are important to establish stability estimates (see 
	\Cref{thm:linf_stability}) which allow us to show convergence of the HDG scheme without any 
	regularization of the potential function $\Phi$. We start with recalling and defining continuous and 
	discrete Green and Laplace operators.
	\subsection{Continuous Green operator}
We recall the continuous Green operator $G:V'\rightarrow V$, where $V = H^1(\Omega)\cap L_0^2(\Omega)$. For any $f\in V'$, 
$G(f)$ is the unique function in $V$ such that
	\begin{equation} 
		\int_{\Omega} \nabla G(f) \cdot \nabla v \dif x = \langle f, v \rangle.
%, \quad \norm{G(f)}_{H^1(\Omega)} \lesssim \norm{f}_{V'}.
	\end{equation}

 The embedding $V \subset L_0^2(\Omega)$ is dense
	%	\footnote{ 
	%		Given $v \in 
	%		H$ 
	%		and a sequence $(v_n)_{n \in \mathbb{N}} \subset H^1(\Omega)$ converging to $v$ in 
	%		$L^2(\Omega)$, it 
	%		holds that
	%		%
	%		\begin{equation}
	%			0 = |(v,1)_{L^2(\Omega)}| = \lim_{n \to \infty} \big| \int_{\Omega} v_n \dif x \big|,
	%		\end{equation}
	%		%
	%		and thus, the sequence of functions $(v_n - \frac{1}{|\Omega|}\int_{\Omega} v_n \dif x)_{n \in 
	%			\mathbb{N}} 
	%		\subset 
	%		V$ 
	%		is 
	%		such that:
	%		%
	%		\begin{equation}	
	%			\norm[2]{ v - v_n + \frac{1}{|\Omega|}\int_{\Omega} v_n \dif x}_{L^2(\Omega)} \le \norm{v - 
	%				v_n}_{L^2(\Omega)} 
	%			+  \big| \int_{\Omega} v_n \dif x \big| \to 0 \text{ as } n \to 
	%			\infty.
	%		\end{equation}
	%		%
	%		Consequently, 
	%		identifying $H \cong H'$ via the Riesz isomorphism, we
	%		have a Gelfand triple:
	%		%
	%		\begin{equation}
	%			V \subset H \subset V', \quad \langle f,v \rangle_{V' \times V} = (f,v)_H, \quad \forall f \in H.
	%	\end{equation}} 
	and therefore, if $f \in L_0^2(\Omega)$ and $(\cdot,\cdot)_\Omega$ denote the $L^2$ inner-product, we find:
	\begin{equation} \label{eq:regularity_shift}
		\norm{G(f)}_{H^1(\Omega)} \lesssim \norm{f}_{V'} = \sup_{ \substack{ v \in H^1(\Omega) \cap 
				L_0^2(\Omega) \\ 
				\norm{v}_{H^1(\Omega)} = 1}}| (f, v)_{\Omega} |.  
	\end{equation}
In addition, if $\Omega$ is convex, we have
\begin{equation}\label{eq:ellipticity}
\Vert G(f)\Vert_{H^2(\Omega)} \lesssim  \Vert f \Vert_{L^2(\Omega)}.
\end{equation}

	\subsection{Discrete Green operator}
	We introduce a discrete analogue of the Green operator, $G$,  on the space $\boldsymbol{S}_h$. 
	Consider $\boldsymbol{G}_h: \boldsymbol{S}_h \to \boldsymbol{M}_h$ satisfying: 
	\begin{equation} \label{eq:def_disc_green}
		a_{\mathcal{D}}(\boldsymbol{G}_h \boldsymbol{w}_h, \boldsymbol{v}_h) = 
		(\boldsymbol{w}_h, 
		\boldsymbol{v}_h)_{0,h}, \quad \forall \boldsymbol{v}_h \in \boldsymbol{M}_h, \quad \forall \boldsymbol{w}_h \in \boldsymbol{S}_h.
	\end{equation}
	Note that the right-hand side of \eqref{eq:def_disc_green} defines, for fixed $\boldsymbol{w}_h 
	\in \boldsymbol{S}_h$, a bounded linear functional on $\boldsymbol{M}_h$ by the 
	Cauchy-Schwarz's inequality and equivalence of norms on finite dimensional spaces. This fact, combined with the fact 
	that 
	$a_{\mathcal{D}}(\cdot,\cdot)$ 
	is coercive on $\boldsymbol{M}_h$, shows that the operator $\boldsymbol{G}_h$ is 
	well defined by 
	the Lax-Milgram theorem.
	\subsection{Discrete Laplace operator} \label{ss:disc_lapl}
	We introduce a discrete Laplace 
	operator $\boldsymbol{\Delta}_h: \boldsymbol{S}_h \to \boldsymbol{M}_h$ as the unique solution to
	\begin{equation} \label{eq:discrete_laplacian}
		-	(\boldsymbol{\Delta}_h \boldsymbol{w}_h, \boldsymbol{v}_h)_{0,h} = 
		a_{\mathcal{D}}(\boldsymbol{w}_h,\boldsymbol{v}_h), \quad 
		\forall \boldsymbol{v}_h \in \boldsymbol{M}_h.
	\end{equation}
	That $\boldsymbol{\Delta}_h$ is well-defined follows from the Riesz representation theorem, as the 
	right-hand side defines a bounded linear functional on $\boldsymbol{M}_h$ while 
	$(\cdot,\cdot)_{0,h}$ 
	defines an inner-product on $\boldsymbol{S}_h$. We now show that $-\boldsymbol{G}_h$ is the 
	inverse of the discrete 
	Laplacian 
	$\boldsymbol{\Delta}_h$ restricted to $\boldsymbol{M}_h$. To this end, the 
	definitions of $\boldsymbol{G}_h$ 
	\eqref{eq:def_disc_green} and $\boldsymbol{\Delta}_h$ \eqref{eq:discrete_laplacian} yield
	\begin{equation}
		a_{\mathcal{D}}(\boldsymbol{G}_h \boldsymbol{\Delta}_h \boldsymbol{w}_h, \boldsymbol{v}_h) = 
		(\boldsymbol{\Delta}_h \boldsymbol{w}_h,\boldsymbol{v}_h)_{0,h} = 
		-a_{\mathcal{D}}(\boldsymbol{w}_h, \boldsymbol{v}_h), \quad \forall \boldsymbol{w}_h, 
		\boldsymbol{v}_h \in 
		\boldsymbol{M}_h,
	\end{equation}
	so that
	\begin{equation}
		a_{\mathcal{D}}(\boldsymbol{G}_h \boldsymbol{\Delta}_h \boldsymbol{w}_h + \boldsymbol{w}_h, 
		\boldsymbol{v}_h) = 0, \quad \forall \boldsymbol{v}_h \in \boldsymbol{M}_h.
	\end{equation}
	Choosing $\boldsymbol{v}_h = \boldsymbol{G}_h \boldsymbol{\Delta}_h \boldsymbol{w}_h + 
	\boldsymbol{w}_h \in \boldsymbol{M}_h$, using the coercivity of the bilinear form $a_{\mathcal{D}}$ 
	\eqref{eq:aD_coercive}, and noting that $\norm{\cdot}_{1,h}$ is a norm on $\boldsymbol{M}_h$, we 
	find 
	that
	\begin{equation}
		\boldsymbol{w}_h = - \boldsymbol{G}_h \boldsymbol{\Delta}_h \boldsymbol{w}_h, \quad \forall 
		\boldsymbol{w}_h \in \boldsymbol{M}_h. \label{eq:G_h_inverse_laplace}
	\end{equation}
	\subsection{Properties of $\bm{\Delta}_h$ and $\bm{G}_h$}
	To set notation, for $\bm{w}_h \in \bm{S}_h$, we write
	\begin{equation}
		\begin{split}
			\boldsymbol{G}_h \boldsymbol{w}_h  = (G_h \boldsymbol{w}_h, \hat{G}_h \boldsymbol{w}_h) 
			\in M_h \times \hat{S}_h,    \,\,\,
			\boldsymbol{\Delta}_h \boldsymbol{w}_h = (\Delta_h \boldsymbol{w}_h, \hat{\Delta}_h 
			\boldsymbol{w}_h ) \in M_h \times \hat{S}_h. 
		\end{split}
	\end{equation}
	In other words, $\Delta_h \boldsymbol{w}_h$ (or $G_h \boldsymbol{w}_h$) and $\hat{\Delta}_h 
	\boldsymbol{w}_h$ (or $\hat{G}_h \boldsymbol{w}_h$) refer to, respectively, the element and face 
	degrees of freedom obtained from $\boldsymbol{\Delta}_h \boldsymbol{w}_h$ (or 
	$\boldsymbol{G}_h 
	\boldsymbol{w}_h$).
	
	We now proceed to show properties for $\bm{\Delta}_h$ and $\bm{G}_h$. To this end, we use the 
	local $L^2$ projections. 
	Let $E \in \mathcal{E}_h$, $e \in \Gamma_h$ and denote by $\pi_h$ and 
	$\hat{\pi}_h$ the 
	orthogonal 
	$L^2$-projections satisfying
	\begin{align}
		\int_E \del{ v - \pi_h v} w_h \dif x = 0, \quad \forall w_h \in \mathbb{P}_k(E),  \,\,\, 
		%\label{eq:l2_projection_element_projection}\\
		\int_e \del{ v - \hat{\pi}_h v} \hat{w}_h \dif s& = 0, \quad \forall \hat{w}_h \in \mathbb{P}_k(e). 
		\label{eq:l2_projections}
	\end{align}
	\begin{lemma}[Properties of $\pi_h$ and $\hat{\pi}_h$] \label{prop:proj_estimates} Let $E \in 
		\mesh$ and $1\leq p \leq \infty$. We have that 
		\begin{align}
			\norm{\pi_h v}_{L^p(E)} &\lesssim \norm{v}_{L^p(E)}, \quad \forall v \in 
			L^p(E),  \label{eq:lp_stab_proj}\\
			\norm{ \nabla \pi_h v}_{L^p(E)} &\lesssim |v|_{W^{1,p}(E)}, \quad \forall v \in 
			W^{1,p}(E).  \label{eq:w1p_stab_proj}
		\end{align}
		Moreover, the following approximation results hold:  let $s \in 
		\mathbb{N}$ such that $1 \le s \le k+1$. Then, it holds for all $v \in 
		W^{s,p}(E)$ that 
		\begin{align}
			|v - \pi_h v|_{W^{m,p}(E)} & \lesssim h_E^{s-m} |v|_{W^{s,p}(E)}, && \forall \, 0 \le m \le s, 
			\label{eq:l2_proj_elem_est} \\
			|v - \pi_h v|_{W^{m,p}(e)} & \lesssim h_E^{s-m - \frac{1}{p}} |v|_{W^{s,p}(E)},  &&\forall \,
			0 \le m \le s-1,
			\label{eq:l2_proj_elem_face_est}
		\end{align}
		where $e \in \mathcal{F}_E$. 
		Moreover, if $v \in H^s(E)$, then it holds that
		\begin{equation}
			\norm{v - \hat{\pi}_h v}_{L^2(e)} \lesssim h_E^{s - \frac{1}{2}} |v|_{H^s(E)}.
			\label{eq:l2_proj_face_est} 
		\end{equation}
	\end{lemma}
	\begin{proof}
		Proofs of \eqref{eq:lp_stab_proj}-\eqref{eq:l2_proj_elem_face_est} can be found in 
		\cite{DiPietro:2017}. The estimate 
		\eqref{eq:l2_proj_face_est} follows from \eqref{eq:l2_proj_elem_face_est} by observing that the 
		best approximation property of the $L^2$-projection $\hat{\pi}_h$ ensures that
		\begin{equation}
			\norm{v - \hat{\pi}_h v}_{L^2(e)} \le \norm{v - \pi_h v}_{L^2(e)}, \quad \forall v \in L^2(E), \quad
\forall e \in \partial E.
		\end{equation}
		\qed
	\end{proof}
	With the above properties, we show the following inverse  estimates for the discrete Laplacian. 
	\begin{lemma}
		The following estimates hold:
		\begin{align}
			\norm{\boldsymbol{\Delta}_h \boldsymbol{w}_h}_{0,h} &\lesssim h^{-1}  
			\label{eq:lapl_l2_bnd}
			\norm{\boldsymbol{w}_h}_{1,h}, \quad\forall \boldsymbol{w}_h \in \boldsymbol{S}_h, \\
			\norm{\Delta_h \boldsymbol{w}_h}_{V'} &\lesssim\norm{\boldsymbol{w}_h}_{1,h}, \quad\forall 
			\boldsymbol{w}_h \in \boldsymbol{S}_h. \label{eq:lapl_h-1_bnd}
		\end{align}
	\end{lemma}
	
	\begin{proof}
		We begin by showing \eqref{eq:lapl_l2_bnd}. Choosing $\boldsymbol{v}_h = 
		\boldsymbol{\Delta}_h 
		\boldsymbol{w}_h$ in \eqref{eq:discrete_laplacian} and using the continuity of the bilinear 
		form 
		$a_{\mathcal{D}}$ \eqref{eq:aD_continuous} and \eqref{eq:inverse_estimate_1h_to_0h}, we find
		\begin{equation}
			\norm{\boldsymbol{\Delta}_h \boldsymbol{w}_h}_{0,h}^2  \lesssim 
			\norm{\boldsymbol{w}_h}_{1,h} 	
			\norm{\boldsymbol{\Delta}_h \boldsymbol{w}_h}_{1,h} \lesssim h^{-1} 
			\norm{\boldsymbol{w}_h}_{1,h} 		\norm{\boldsymbol{\Delta}_h \boldsymbol{w}_h}_{0,h} .
		\end{equation}
		% %
		% By a standard inverse equality \Cref{prop:inverse_ineq}, it holds that
		% %
		% \begin{equation}
		% 	\norm{ \nabla (\Delta_h \boldsymbol{w}_h)}_{L^2(E)} \lesssim h_E^{-1} 	\norm{ \Delta_h 
		% 		\boldsymbol{w}_h}_{L^2(E)},
		% \end{equation}
		% %
		% and thus the assumed quasi-uniformity of the mesh implies 
		% %
		% \begin{equation}
		% 	\norm{ \nabla_h (\Delta_h \boldsymbol{w}_h)}_{L^2(\Omega)} \lesssim h^{-1} 	\norm{ 
		% 	\Delta_h 
		% 		\boldsymbol{w}_h}_{L^2(\Omega)},
		% \end{equation}
		% %
		% and also for all $\boldsymbol{v}_h \in \boldsymbol{S}_h$ 
		% that
		% %
		% \begin{equation}
		% 	h_E^{-1}	\norm{v_h - \hat{v}_h}_{L^2(\partial E)}^2 = h_E^{-2} h_E	\norm{v_h - 
		% 		\hat{v}_h}_{L^2(\partial E)}^2 \lesssim h^{-2} h_E \norm{v_h - 
		% 		\hat{v}_h}_{L^2(\partial E)}^2.
		% \end{equation}
		% %
		%Consequently, we find $\norm{\boldsymbol{\Delta}_h \boldsymbol{w}_h}_{1,h} \lesssim h^{-1} 
		%\norm{\boldsymbol{\Delta}_h \boldsymbol{w}_h}_{0,h}$.
		This shows \eqref{eq:lapl_l2_bnd}. Next, we show \eqref{eq:lapl_h-1_bnd}. As $\Delta_h 
		\boldsymbol{w}_h \in M_h \subset 
		L_0^2(\Omega)$, we have
		\begin{equation} \label{eq:V'_bnd_1}
			\norm{\Delta_h \boldsymbol{w}_h}_{V'} = \sup_{ \substack{ \varphi \in 
					H^1(\Omega)\cap L_0^2(\Omega) \\ \norm{\varphi}_{H^1(\Omega)} = 1 }} |(\Delta_h 
			\boldsymbol{w}_h,\varphi)_{\Omega}| = \sup_{ \substack{ \varphi \in 
					H^1(\Omega)\cap L_0^2(\Omega) \\ \norm{\varphi}_{H^1(\Omega)} = 1 }} |(\Delta_h 
			\boldsymbol{w}_h,\pi_h \varphi)_{\Omega}|.
		\end{equation}
		% 
		%with $\pi_h: L^2(\Omega) \to S_h$ the standard orthogonal $L^2$-projection onto $S_h$.
		By the definition of the discrete Laplace operator \eqref{eq:discrete_laplacian} and the 
		inner-product 
		\eqref{eq:product_space_inner_product_l2},
		%
		%		\begin{equation} \label{eq:V'_bnd_2}
		%			(\Delta_h 
		%			\boldsymbol{w}_h,\pi_h \varphi)_{L^2(\Omega)} = - 
		%			a_{\mathcal{D}}(\boldsymbol{w}_h,(\pi_h 
		%			\varphi, 
		%			\hat{\pi}_h \varphi)) - \sum_{E \in \mathcal{E}_h} h_E \int_{\partial E} \del[0]{ \Delta_h 
		%				\boldsymbol{w}_h - 
		%				\hat{\Delta}_h \boldsymbol{w}_h }\del{ \pi_h \varphi - \hat{\pi}_h 
		%				\varphi } \dif s,
		%		\end{equation}
		\begin{equation} \label{eq:V'_bnd_2}
			(\Delta_h 
			\boldsymbol{w}_h,\pi_h \varphi)_{\Omega} = - 
			a_{\mathcal{D}}(\boldsymbol{w}_h,(\pi_h 
			\varphi, 
			\hat{\pi}_h \varphi)) - j_{0}(\boldsymbol{\Delta}_h \boldsymbol{w}_h, (\pi_h 
			\varphi, 
			\hat{\pi}_h \varphi)). 
		\end{equation}
		%
		%where $\hat{\pi}_h: L^2(\partial \mathcal{E}_h) \to \hat{S}_h$ is the standard orthogonal 
		%$L^2$-projection o
		%nto $\hat{S}_h$. 
		Observe that  \Cref{prop:proj_estimates} implies that 
		%		\footnote{Observe that by the triangle 
		%			inequality,%
		%			\begin{equation}
		%				\begin{split}
		%					\norm[0]{(\pi_h \varphi, \hat{\pi}_h \varphi)}_{1,h}^2 &= \norm{\nabla_h \pi_h 
		%						\varphi}_{L^2(\Omega)}^2 + \sum_{E \in \mathcal{E}_h}\frac{1}{h_E} \norm{(\pi_h 
		%						\varphi - 
		%						\varphi) - (\hat{\pi}_h \varphi - v)}_{L^2(\partial E)}^2
		%					\\ &= \norm{\nabla_h \pi_h 
		%						\varphi}_{L^2(\Omega)}^2 + \sum_{E \in \mathcal{E}_h} \sum_{e \in 
		%						\mathcal{F}_E}\frac{1}{h_E} 
		%					\norm{(\pi_h \varphi - 
		%						\varphi) - (\hat{\pi}_h \varphi - v)}_{L^2(e)}^2 \\
		%					& \le \norm{\nabla_h \pi_h 
		%						\varphi}_{L^2(\Omega)}^2 + \sum_{E \in \mathcal{E}_h} \sum_{e \in 
		%						\mathcal{F}_E}\frac{1}{h_E} 
		%					\del{ \norm{\pi_h \varphi - 
		%							\varphi}_{L^2(e)}^2 + \norm{\hat{\pi}_h \varphi - v}_{L^2(e)}^2}.
		%				\end{split}
		%			\end{equation}
		%			%
		%			By \cref{eq:w1p_stab_proj}, $\norm{\nabla_h \pi_h 
		%				\varphi}_{L^2(\Omega)} \lesssim \norm{\varphi}_{H^1(\Omega)}$, and by 
		%			\cref{eq:l2_proj_elem_face_est} and \cref{eq:l2_proj_face_est} and the fact that 
		%			$\text{card}\cbr{\mathcal{F}_{E}} = d+1$, we have
		%			\begin{equation}
		%				\sum_{e \in \mathcal{F}_E}\frac{1}{h_E} 
		%				\del{ \norm{\pi_h \varphi - 
		%						\varphi}_{L^2(e)}^2 + \norm{\hat{\pi}_h \varphi - v}_{L^2(e)}^2} \lesssim (d+1) 
		%				\norm{\varphi}_{H^1(E)}^2.
		%		\end{equation}	} 
		%
		\begin{align}% \label{eq:approx_prop_disc_lap_1}
			%	&\norm{\pi_h \varphi - \hat{\pi}_h \varphi}_{L^2(\partial E)} \lesssim h_E^{1/2} 
			% \norm{\varphi}_{H^1(E)}, \\
			\norm[0]{(\pi_h \varphi, \hat{\pi}_h \varphi)}_{1,h} \lesssim \norm{\varphi}_{H^1(\Omega)}.  
			\label{eq:approx_prop_disc_lap_2}
		\end{align}
		Returning to \eqref{eq:V'_bnd_1}, using \eqref{eq:V'_bnd_2}, the
		continuity of the bilinear form $a_{\mathcal{D}}$ \eqref{eq:aD_continuous},
		\eqref{eq:bound_j0_1h},  \eqref{eq:approx_prop_disc_lap_2} and \eqref{eq:lapl_l2_bnd}, we obtain
		%
		%		\begin{equation}
		%			\begin{split}
		%				\big| (\Delta_h 
		%				\boldsymbol{w}_h,\pi_h \varphi)_{L^2(\Omega)} \big| & \lesssim \norm{ 
		%					\boldsymbol{w}_h}_{1,h} \norm{ (\pi_h \varphi, \hat{\pi}_h \varphi)}_{1,h}  + \norm{ 
		%					\boldsymbol{\Delta}_h \boldsymbol{w}_h}_{0,h}  \sum_{E \in \mathcal{E}_h} 
		%h_E^{1/2} 
		%					\norm{\pi_h 
		%					\varphi - 
		%					\hat{\pi}_h \varphi}_{L^2(\partial E)} \\
		%				&	\lesssim \norm{\boldsymbol{w}_h}_{1,h} \norm{\varphi}_{H^1(\Omega)} + h^{-1} 
		%				\norm{\boldsymbol{w}_h}_{1,h} h \norm{\varphi}_{H^1(\Omega)} \\
		%				& \lesssim \norm{\boldsymbol{w}_h}_{1,h} \norm{\varphi}_{H^1(\Omega)} .
		%			\end{split}
		%		\end{equation}
		\begin{equation}
			\begin{split}
				\big| (\Delta_h 
				\boldsymbol{w}_h,\pi_h \varphi)_{\Omega} \big| & \lesssim (\,\norm{ 
					\boldsymbol{w}_h}_{1,h} + h \norm{ 
					\boldsymbol{\Delta}_h \boldsymbol{w}_h}_{0,h} ) \norm{ (\pi_h \varphi, \hat{\pi}_h 
					\varphi)}_{1,h}  \lesssim \norm{\boldsymbol{w}_h}_{1,h} \norm{\varphi}_{H^1(\Omega)} .
			\end{split}
		\end{equation}
		Therefore, for any $\varphi \in H^1(\Omega) \cap L_0^2(\Omega)$ satisfying 
		$\norm{\varphi}_{H^1(\Omega)} = 1$, it holds that
		\begin{equation}
			\big| (\Delta_h 
			\boldsymbol{w}_h, \varphi)_{\Omega} \big| \lesssim  \norm{\boldsymbol{w}_h}_{1,h}.
		\end{equation}
		The result follows.
		\qed
	\end{proof}
	We proceed to show approximation properties of the discrete Green operator.
	\begin{lemma}   \label{lem:green_estimates}
		Assume that $\Omega$ is convex.  Then, for any $\bm{w}_h \in \bm{M}_h$,  the 
		following estimates hold:
		\begin{align}
			\norm{\boldsymbol{G}_h \boldsymbol{w}_h - \boldsymbol{\pi}_h(G w_h)}_{1,h} 
			&\lesssim 
			h 
			\norm{\boldsymbol{w}_h}_{0,h}, \label{eq:green_op_error_1} \\
			\norm{G_h \boldsymbol{w}_h - \pi_h(G w_h)}_{L^2(\Omega)} & \lesssim h^2 
			\norm{\boldsymbol{w}_h}_{0,h}, \label{eq:green_op_error_2}
		\end{align}
		where we define $\boldsymbol{\pi}_h(G w_h) = (\pi_h G w_h, \hat{\pi}_h G w_h) \in 
		\boldsymbol{M}_h$.
	\end{lemma}
	\begin{proof}
		By consistency and definition of $G(w_h)$, we obtain that 
		$$(w_h, v_h)_\Omega = - (\Delta G(w_h), v_h)_\Omega = 
		a_{\mathcal{D}} (\bm{G}w_h , \bm{v}_h), \quad  \forall \bm{v}_h \in \bm{S}_h, $$ where 
		$\bm{G}w_h = (Gw_h, Gw_h|_{\Gamma_h})$. With \eqref{eq:def_disc_green}, the following 
		orthogonality 
		relation easily follows: 
		\begin{equation}
			a_{\mathcal{D}}(\bm{G} w_h - \bm{G}_h \bm{w}_h, \bm{v}_h) = - j_0(\bm{w}_h, \bm{v}_h), 
			\quad 
			\forall \bm{v}_h \in \bm{M_h}. \label{eq:galerkin_orth_green}
		\end{equation} 
		Thus, $\bm{G}_h \bm{w}_h$ is a modified HDG elliptic projection of $\bm{G} w_h$.
		Hence, the rest of the proof naturally modifies standard convergence proofs for HDG applied to 
		elliptic problems. The details are 
		provided in \ref{appendix:proof_G_prop}.\qed
	\end{proof}
	\subsection{Discrete Agmon inequality} 
	Recall Agmon's inequality (see e.g. \cite[Lemma 4.10]{Constantin:book}):
	\begin{equation} \label{eq:cont_agmon}
		\norm{v}_{L^{\infty}(\Omega)} \lesssim \norm{v}_{H^1(\Omega)}^{1/2} 
		\norm{v}_{H^2(\Omega)}^{1/2}, \quad \forall v \in H^2(\Omega).
	\end{equation}
	We now establish the following discrete counter-part in the HDG setting following closely the 
	arguments in 
	\cite{Chave:2016, Kay:2009}.
	\begin{lemma}[Discrete Agmon Inequality] \label{lem:disc_agmon}
		Assume that $\Omega$ is convex. 
% and that the family of 
%		meshes $\cbr{\mathcal{E}_h}$ is quasi-uniform. 
Then, the following inequality holds:
		\begin{equation}
			\norm{w_h}_{L^{\infty}(\Omega)} \lesssim 
			\norm{\boldsymbol{w}_h}_{1,h}^{1/2} \norm{\boldsymbol{\Delta}_h 
				\boldsymbol{w}_h}_{0,h}^{1/2}, 
			\quad \forall \boldsymbol{w}_h \in \boldsymbol{M}_h.
		\end{equation}
	\end{lemma}
	
	\begin{proof}
		Considering \eqref{eq:G_h_inverse_laplace} component-by-component, we observe that $w_h = 
		-G_h 
		\boldsymbol{\Delta}_h \boldsymbol{w}_h$. 
		This fact, combined with the triangle inequality, yields  (recall that $\boldsymbol{\Delta}_h 
		\boldsymbol{w}_h = (\Delta_h \boldsymbol{w}_h, \hat{\Delta}_h \boldsymbol{w}_h)$)
		\begin{equation}
			\begin{split}
				\norm{w_h}_{L^{\infty}(\Omega)} &= \norm{G_h \boldsymbol{\Delta}_h \boldsymbol{w}_h 
				}_{L^{\infty}(\Omega)} \\
				& \le \norm{\pi_h G(\Delta_h \boldsymbol{w}_h)}_{L^{\infty}(\Omega)} + \norm{G_h 
					\boldsymbol{\Delta}_h 
					\boldsymbol{w}_h -\pi_h G(\Delta_h \boldsymbol{w}_h)}_{L^{\infty}(\Omega)} = T_1 + T_2.
			\end{split} 
		\end{equation}
		To bound $T_1$, we note that the $L^{\infty}$-stability of the $L^2$-projection operator 
		\eqref{eq:lp_stab_proj} ensures that
		\begin{equation}
			\norm{\pi_h G(\Delta_h \boldsymbol{w}_h)}_{L^{\infty}(\Omega)}  \lesssim \norm{ G(\Delta_h 
				\boldsymbol{w}_h)}_{L^{\infty}(\Omega)},
		\end{equation}
		and thus the continuous Agmon inequality \eqref{eq:cont_agmon} yields (since $G(\Delta_h 
		\boldsymbol{w}_h)\in H^2(\Omega)$)
		\begin{equation}
			\norm{ G(\Delta_h \boldsymbol{w}_h)}_{L^{\infty}(\Omega)} 
			\lesssim \norm{ G(\Delta_h \boldsymbol{w}_h)}_{H^1(\Omega)}^{1/2} 
			\norm{G(\Delta_h \boldsymbol{w}_h)}_{H^2(\Omega)}^{1/2}.
		\end{equation}
		By \eqref{eq:regularity_shift}, \eqref{eq:ellipticity}, and 
		\eqref{eq:lapl_h-1_bnd}, we find
		\begin{equation}
			\begin{split}
				\norm{ G(\Delta_h \boldsymbol{w}_h)}_{L^{\infty}(\Omega)} &\lesssim \norm{  \Delta_h 
					\boldsymbol{w}_h}_{V'}^{1/2} \norm{\Delta_h 
					\boldsymbol{w}_h}_{L^2(\Omega)}^{1/2} 
				\lesssim \norm{  
					\boldsymbol{w}_h}_{1,h}^{1/2} \norm{\boldsymbol{\Delta}_h 
					\boldsymbol{w}_h}_{0,h}^{1/2}.
			\end{split}
		\end{equation}
		Next, we bound the term $T_2$. By the quasi-uniformity assumption and 
		\eqref{eq:lebesgue_embed}, we find that 
		\begin{equation}
			\begin{split}
				\norm{G_h \boldsymbol{\Delta}_h \boldsymbol{w}_h -\pi_h G(\Delta_h 
					\boldsymbol{w}_h)}_{L^{\infty}(\Omega)} 
				&\lesssim h^{-d/2} 	\norm{G_h \boldsymbol{\Delta}_h \boldsymbol{w}_h -\pi_h G(\Delta_h 
					\boldsymbol{w}_h)}_{L^{2}(\Omega)},
			\end{split}
		\end{equation}
		and thus by \Cref{lem:green_estimates}, \eqref{eq:ellipticity},  and \eqref{eq:lapl_l2_bnd},
		\begin{equation}
			\begin{split}
				\norm{G_h \boldsymbol{\Delta}_h \boldsymbol{w}_h -\pi_h G(\Delta_h 
					\boldsymbol{w}_h)}_{L^{\infty}(\Omega)} 
				& \lesssim 
				h^{(4-d)/2} 	\norm{\boldsymbol{\Delta}_h \boldsymbol{w}_h}_{0,h} \\
				& \lesssim h^{(3-d)/2} 	\del{h^{1/2}\norm{\boldsymbol{\Delta}_h 
						\boldsymbol{w}_h}_{0,h}^{1/2} }\norm{\boldsymbol{\Delta}_h 
					\boldsymbol{w}_h}_{0,h}^{1/2} \\ &\lesssim h^{(3-d)/2} \norm{
					\boldsymbol{w}_h}_{1,h}^{1/2} \norm{\boldsymbol{\Delta}_h 
					\boldsymbol{w}_h}_{0,h}^{1/2}.
			\end{split}
		\end{equation}
		We conclude by noting that $d \le 3$ and thus $h^{(3-d)/2} \lesssim 1$.
		\qed
	\end{proof}
	%We denote by $\overline{v}$ the global average of any function $v$ in $L^2(\Omega)$:
	%\begin{equation}
	%	\overline{v} = \frac{1}{|\Omega|} \int_{\Omega} v, \quad \forall v\in L^2(\Omega).
	%\end{equation}
	
	\begin{corollary}
		\label{cor:disc_agmon_Sh}
		Assume that $\Omega$ is convex. % and that the family of 
		%meshes $\cbr{\mathcal{E}_h}$ is quasi-uniform.  
Then, the following inequality holds:
		\begin{equation}
			\norm{w_h - \overline{w_h}}_{L^{\infty}(\Omega)} \lesssim 
			\norm{\boldsymbol{w}_h}_{1,h}^{1/2} \norm{\boldsymbol{\Delta}_h 
				\boldsymbol{w}_h}_{0,h}^{1/2}, 
			\quad \forall \boldsymbol{w}_h=(w_h,\hat{w}_h) \in \boldsymbol{S}_h.
		\end{equation}
	\end{corollary}
	\begin{proof}
		Let $\boldsymbol{w}_h \in \boldsymbol{S}_h$ and set $\overline{\boldsymbol{w}}_h = 
		(\overline{w}_h, \overline{w}_h|_{\Gamma_h}) \in 
		\boldsymbol{S}_h$. Then, $\boldsymbol{w}_h - \overline{\boldsymbol{w}}_h \in 
		\boldsymbol{M}_h$, 
		and thus by \Cref{lem:disc_agmon},
		\begin{equation}
			\norm{w_h - \overline{w}_h}_{L^{\infty}(\Omega)} \lesssim 
			\norm{\boldsymbol{w}_h - \overline{\boldsymbol{w}}_h}_{1,h}^{1/2} 
			\norm{\boldsymbol{\Delta}_h \del{
					\boldsymbol{w}_h - \boldsymbol{\overline{w}}_h}}_{0,h}^{1/2}.
		\end{equation}
		Observe that $\norm{\overline{\boldsymbol{w}}_h}_{1,h} = 0$, and
		\begin{equation} \label{eq:disc_lapl_const_zero}
			(\boldsymbol{\Delta}_h \overline{\boldsymbol{w}}_h, \boldsymbol{v}_h)_{0,h} = 
			-a_{\mathcal{D}}(\overline{\boldsymbol{w}}_h, \boldsymbol{v}_h) = 0, \quad \forall 
			\boldsymbol{v}_h 
			\in 
			\boldsymbol{M}_h,
		\end{equation}
		and thus testing \eqref{eq:disc_lapl_const_zero} with $\boldsymbol{v}_h = \boldsymbol{\Delta}_h 
		\overline{\boldsymbol{w}}_h$,
		\begin{equation}
			\norm{\boldsymbol{\Delta}_h \overline{\boldsymbol{w}}_h}_{0,h}^2 = 0.
		\end{equation}
		The result follows.
		\qed
	\end{proof}
	\subsection{Discrete Gagliardo--Nirenberg inequality }
	We recall the Gagliardo--Nirenberg inequality for bounded domains (\cite{Gagliardo1959}, 
	Theorem 1.2 in \cite{LiZhang2022}).
	for $2 \le p \le 
	p^\star$ with $p^\star$ defined as in \Cref{lem:broken_poincare2}, it holds that 
	\begin{equation} \label{eq:cont_gagl_niren}
		\norm{\nabla v}_{L^{p}(\Omega)} \lesssim \norm{v}_{L^2(\Omega)}^{1-\alpha} 
		\Vert v \Vert_{H^2(\Omega)}^{\alpha}  \lesssim \Vert v\Vert_{H^1(\Omega)}^{1-\alpha}  
		\Vert v\Vert_{H^2(\Omega)}^{\alpha}, \quad \forall v \in H^2(\Omega),
	\end{equation}
	where
	\begin{equation} \label{eq:alpha}
		\alpha = \frac{1}{2} + \frac{d}{2} \del{\frac{1}{2} - \frac{1}{p}}.
	\end{equation}

	We show the following discrete counter-part.	 
	
	%	\Rd For any function $v_h\in S_h$, we denote the broken gradient $\nabla_h v_h \in 
	%L^2(\Omega)$ 
	%	by $(\nabla_h v_h)|_E = (\nabla v_h)|_E$ for all $E\in\mathcal{E}_h$. \Bk
	
	\begin{lemma}[Discrete Gagliardo--Nirenberg inequality] \label{lemma:discrete_GN}
		Suppose that $2 \le p \le p^{\star}$, with $p^\star$ defined in \Cref{lem:broken_poincare2} and let 
		$\alpha$ be
		defined by \eqref{eq:alpha}. Then, it 
		holds 
		that
		\begin{equation} \label{eq:disc_gagl_niren}
			\norm{\nabla_h w_h}_{L^p(\Omega)} \lesssim \norm{\boldsymbol{w}_h}_{1,h}^{1-\alpha} 
			\norm{\boldsymbol{\Delta}_h \boldsymbol{w}_h}_{0,h}^{\alpha}, \quad \forall \boldsymbol{w}_h 
			= (w_h, \hat{w}_h) 
			\in \boldsymbol{M}_h.
		\end{equation}
		As a consequence, we have
		\begin{equation} \label{eq:disc_gagl_niren2}
			\norm{\nabla_h w_h}_{L^p(\Omega)} \lesssim \norm{\boldsymbol{w}_h}_{1,h}^{1-\alpha} 
			\norm{\boldsymbol{\Delta}_h \boldsymbol{w}_h}_{0,h}^{\alpha}, \quad \forall \boldsymbol{w}_h  
			= (w_h,\hat{w}_h)
			\in \boldsymbol{S}_h.
		\end{equation}
	\end{lemma}
	\begin{proof}
		Fix $\bm{w}_h \in \bm{M}_h$. To simplify notation, denote by $\bm{z}_h = (z_h, \hat{z}_h) = 
		\bm{\Delta}_h \bm{w}_h$. Recall from \eqref{eq:G_h_inverse_laplace} that $w_h = -G_h 
		\boldsymbol{z}_h$.  By the triangle inequality, we have that 
		\begin{equation}\label{eq:int11}
			\norm{\nabla_h w_h}_{L^p(\Omega)} 
			\leq \norm{\nabla_h \pi_h G z_h }_{L^p(\Omega)} + 
			\norm{\nabla_h \del{G_h \boldsymbol{z}_h - \pi_h G z_h }}_{L^p(\Omega)}. 
		\end{equation}
		To bound the first term in the right-hand side of \eqref{eq:int11}, we use the $W^{1,p}$-stability 
		of the $L^2$-projection \eqref{eq:w1p_stab_proj}, 
		the continuous 
		Gagliardo--Nirenberg inequality \eqref{eq:cont_gagl_niren}, 
		\eqref{eq:regularity_shift}, \eqref{eq:ellipticity}, and \eqref{eq:lapl_h-1_bnd}:
		\begin{multline}
			\|\nabla_h \pi_h G z_h\|_{L^p(\Omega)} \lesssim \Vert \nabla G z_h\Vert_{L^{p}(\Omega)} 
			\lesssim \Vert G z_h \Vert^{1-\alpha}_{H^1(\Omega)} \Vert G z_h\Vert^\alpha_{H^2(\Omega)} \\ 
			\lesssim \|z_h\|_{V'}^{1-\alpha} \|z_h\|^{\alpha}_{L^2(\Omega)}  \lesssim  
			\norm{\boldsymbol{w}_h}_{1,h}^{1-\alpha} 
			\norm{\boldsymbol{\Delta}_h \boldsymbol{w}_h}_{0,h}^{\alpha}.
		\end{multline}
		To bound the second term in the right-hand side of \eqref{eq:int11},  we use 
		\eqref{eq:lebesgue_embed}, the 
		definition of the norm $\norm{\cdot}_{1,h}$, and \Cref{lem:green_estimates} since $\bm{z}_h \in 
		\bm{M}_h$:
		\begin{multline}
			\norm{\nabla_h \del{G_h \bm{z}_h  - \pi_h G z_h }}_{L^p(\Omega)} \lesssim h^{d(\frac{1}{p} - 
				\frac{1}{2})} 	
			\norm{\nabla_h 
				\del{G_h \bm{z}_h  - \pi_h G z_h }}_{L^2(\Omega)} \\   \lesssim h^{d(\frac{1}{p} - 
				\frac{1}{2})} 	\norm{\boldsymbol{G}_h 
				\bm{z}_h  - \bm{\pi}_h G z_h  }_{1,h} 
			\lesssim h^{d(\frac{1}{p} - \frac{1}{2}) + 1} \norm{\bm{z}_h}_{0,h}  = h^{d(\frac{1}{p} - 
				\frac{1}{2}) + 1} \norm{\boldsymbol{\Delta}_h  
				\boldsymbol{w}_h}_{0,h}.
		\end{multline}
		Finally, by \eqref{eq:lapl_l2_bnd}, we have
		\begin{multline}
			\norm{\nabla_h \del{G_h \bm{z}_h  - \pi_h G z_h }}_{L^p(\Omega)} \\
			\lesssim h^{d(\frac{1}{p} - \frac{1}{2}) + \alpha} \del{ h^{1-\alpha} \norm{ 
					\boldsymbol{\Delta}_h \boldsymbol{w}_h}_{0,h}^{1-\alpha} }  \norm{ 
				\boldsymbol{\Delta}_h 
				\boldsymbol{w}_h}_{0,h}^{\alpha} 
			\lesssim h^{d(\frac{1}{p} - \frac{1}{2}) + \alpha} \norm{ 
				\boldsymbol{w}_h}_{1,h}^{1-\alpha}   
			\norm{ \boldsymbol{\Delta}_h 
				\boldsymbol{w}_h}_{0,h}^{\alpha}.
		\end{multline}
		We conclude by noting that $h^{d(\frac{1}{p} - \frac{1}{2}) + \alpha} \lesssim 1$ since for $d \le 3$ 
		and $2 \le p \le p^\star$ we have  
		\begin{equation}
			d\del{\frac{1}{p} - \frac{1}{2}} + \alpha = d\del{\frac{1}{p} - \frac{1}{2}} + \frac{1}{2} + \frac{d}{2} 
			\del{\frac{1}{2} - \frac{1}{p}} =  \frac{1}{2} - \frac{d}{2} \del{\frac{1}{2} - \frac{1}{p}} \ge 0.
		\end{equation}
		As in the proof of Corollary~\ref{cor:disc_agmon_Sh}, we fix 
		$\boldsymbol{w}_h\in\boldsymbol{S}_h$ and
		consider $\overline{\boldsymbol{w}}_h = (\overline{w}_h,\overline{w}_h|_{\Gamma_h})$. Then 
		$\boldsymbol{w}_h-\overline{\boldsymbol{w}}_h
		\in \boldsymbol{M}_h$ and we can apply \eqref{eq:disc_gagl_niren}.  We then conclude by noting 
		that $\Vert \overline{\boldsymbol{w}}_h\Vert_{1,h}=0$
		and $\boldsymbol{\Delta}_h \overline{\boldsymbol{w}}_h = \mathbf{0}$.
		\qed
	\end{proof}
	
	%---------------------
	\section{The numerical method}
	Let $0 = t_0 < t_1 < \dots < t_{N} = T$ be a partition of the time interval $(0,T)$ into $N$ 
	subintervals of equal length $\tau$.  We discretize 
	\eqref{eq:CH_strong} in time using the first order 
	implicit Euler method, and we employ the Eyre convex-concave splitting scheme \cite{eyre1998} for the nonlinear
 chemical energy density. 
	We treat the convex part implicitly and the concave part explicitly. The 
	fully 
	discrete numerical scheme reads: for any $1 \le n \le N$, given $c_h^{n-1} \in S_h$, find 
	$(\boldsymbol{c}_h^n, \boldsymbol{\mu}_h^n) \in \boldsymbol{S}_h \times \boldsymbol{S}_h$ 
	satisfying 
	\begin{subequations} \label{eq:CH_disc}
		\begin{align}
			(\delta_{\tau} c_h^n, \chi_h)_{\Omega} + a_{\mathcal{D}}(\boldsymbol{\mu}_h^n, 
			\boldsymbol{\chi}_h) &= 0, && \forall \boldsymbol{\chi}_h \in \boldsymbol{S}_h,
			\label{eq:CH_disc_a}\\
			(\Phi_+'(c_h^n) + \Phi_-'(c_h^{n-1}), \phi_h)_{\Omega} + 
			\kappa 
			a_{\mathcal{D}}(\boldsymbol{c}_h^n, \boldsymbol{\phi}_h) &= (\mu_h^n, 
			\phi_h)_{\Omega}, && 
			\forall \boldsymbol{\phi}_h \in \boldsymbol{S}_h. \label{eq:CH_disc_b}
		\end{align}
	\end{subequations}
	Here, $\delta_\tau$ denotes the backward difference operator:
	\begin{equation}
		\delta_\tau c_h^n = \frac{c_h^n - c_h^{n-1}}{\tau}.
	\end{equation}
	We define $\boldsymbol{c}_h^0 \in \boldsymbol{S}_h$ as the solution to the variational 
	problem: given $c_0 \in H^2(\Omega)\cap L_0^2(\Omega)$, set $\boldsymbol{c}_0 = (c_0, 
	c_0|_{\Gamma_h})$  and
	find $\boldsymbol{c}_h^0 \in 
	\boldsymbol{S}_h$ such that 
	\begin{equation} 
		a_{\mathcal{D}}( \boldsymbol{c}_h^0 - \boldsymbol{c}_0, \boldsymbol{v}_h) =0, \quad \forall 
		\boldsymbol{v}_h \in 
		\boldsymbol{S}_h,\,\,\, \mathrm{with  \,\,  constraint} \,\, (c_h^0 - c_0,1)_{\Omega} = 0.  
		\label{eq:init_proj_ch0}
	\end{equation}
	To see that this variational problem is well-posed, first observe that any solution is unique. Indeed, if 
	$\boldsymbol{w}_1, \boldsymbol{w}_2 \in \boldsymbol{S}_h$ solve \eqref{eq:init_proj_ch0}, then 
	$\boldsymbol{w}_1 - \boldsymbol{w}_2 \in \boldsymbol{M}_h$ and
	\begin{equation}\label{eq:proj_wellposed_intermediate}
		a_{\mathcal{D}}(\boldsymbol{w}_1 - \boldsymbol{w}_2, \boldsymbol{v}_h) = 0, \quad \forall 
		\boldsymbol{v}_h \in \boldsymbol{M}_h.
	\end{equation}
	 By testing \eqref{eq:proj_wellposed_intermediate} with $\boldsymbol{v}_h = \boldsymbol{w}_1 - 
	 \boldsymbol{w}_2$ and using the coercivity of $a_{\mathcal{D}}(\cdot,\cdot)$ 
	 \eqref{eq:aD_coercive}, we conclude that $\boldsymbol{w}_1 = \boldsymbol{w}_2$. To show 
	 existence, it is enough to realize that if $\boldsymbol{w}_h \in \boldsymbol{M}_h$ is the unique 
	 solution of the auxiliary problem
	 \begin{equation*}
	 	a_{\mathcal{D}}(\boldsymbol{w}_h,\boldsymbol{v}_h) = 
	 	a_{\mathcal{D}}(\boldsymbol{c}_0,\boldsymbol{v}_h), \quad \forall \boldsymbol{v}_h \in 
	 	\boldsymbol{M}_h,
	 \end{equation*}
	 the existence of which is guaranteed by the Lax--Milgram theorem, then $\boldsymbol{w}_h + 
	 (\overline{c_0},\overline{c_0})$ solves \eqref{eq:init_proj_ch0}, since for all $\boldsymbol{v}_h \in 
	 \boldsymbol{S}_h$,
	 \begin{equation}
	 	a_{\mathcal{D}}(\boldsymbol{w}_h + (\overline{c_0},\overline{c_0}),\boldsymbol{v}_h) = 
	 	a_{\mathcal{D}}(\boldsymbol{w}_h ,\boldsymbol{v}_h - (\overline{v_h}, \overline{v_h})) = 
	 	a_{\mathcal{D}}(c_0 ,\boldsymbol{v}_h - (\overline{v_h}, \overline{v_h})) =  a_{\mathcal{D}}(c_0 
	 	,\boldsymbol{v}_h).
	 \end{equation}
	 In the above, we used that for any $\bm{\theta} \in H^2(\mathcal{E}_h) \times L^2(\Gamma_h)$ and any constant $\alpha\in\mathbb{R}$, we have that 
	 \begin{equation} 
		a_\mathcal{D} (\bm{\theta} , (\alpha,\alpha)) =a_\mathcal{D} ((\alpha,\alpha),\bm{\theta} ) 
		=  0. \label{eq:ad_0_constant}
	\end{equation} 
  %Here, $\overline{\boldsymbol{v}_h} = (\overline{v_h},\overline{v_h}). $ 
	%
	%	\rami{I think this works as $\bm{c}_h^0$ is uniquely defined: To see that, suppose there is a 
	%		$\bm{w}_h^0 \neq \bm{c}_h^0$ solving \cref{eq:init_proj_ch0} then from the constraint 
	%		$\bm{w}_h^0 - \bm{c}_h^0 \in \bm{M}_h$ and satisfy $a_\mathcal{D} (\bm{w}_h^0 - 
	%\bm{c}_h^0, 
	%		\bm{v}_h) = 0$. From coercivity of $a_\mathcal{D}$ and the constraint, we obtain that 
	%		$\bm{w}_h^0 
	%		- \bm{c}_h^0 = \bm{0}. $} 

	\begin{proposition}[Global mass conservation] \label{prop:mass_conservation}
		The HDG scheme \eqref{eq:CH_disc} satisfies the discrete global mass conservation
		\begin{equation}
			(c_h^n,1)_\Omega = (c_h^0,1)_\Omega = (c_0,1)_\Omega = (c(t_n),1)_\Omega, \quad \forall 1 
			\le n \le N.
		\end{equation}
	\end{proposition}
	\begin{proof}
		Choose $\boldsymbol{\chi}_h = (1,1)$ in \eqref{eq:CH_disc_a} and use \eqref{eq:ad_0_constant} to 
		conclude that   
		% \begin{align}
		% 	a_{\mathcal{D}}(\boldsymbol{\mu}_h^n,(1,1)) 	&= \sum_{E \in \mathcal{E}_h} \int_{E} \nabla 
		% 	\mu_h \cdot \nabla 1 
		% 	\dif x
		% 	- \sum_{E \in \mathcal{E}_h} \int_{\partial E} \del{ \nabla \mu_h \cdot n_E \del{1 - 1} \dif s + 
		%\nabla 
		% 	1 \cdot n_E \del{\mu_h - \hat{\mu}_h} \dif s } \nonumber  \\
		% 	&\quad+ \sum_{E \in \mathcal{E}_h} \int_{\partial E} \frac{\sigma}{h_E}
		% 	(\mu_h - \hat{\mu}_h)(1-1) \dif s \nonumber \\ & =  0. 
		% \end{align}
		$(\delta_\tau c_h^n , 1  ) = 0$ for any $n$, which yields the first equality. The second equality 
		holds from \eqref{eq:init_proj_ch0}. The third equality is a property of problem 
		\eqref{eq:CH_strong}, which follows from integrating 
		\eqref{eq:CH_strong_a} over $\Omega$, using the divergence theorem, and boundary condition 
		\eqref{eq:CH_strong_d}.  
		\qed
	\end{proof}	
	The remaining part of this section focuses on  showing that a unique solution exists for 
	\eqref{eq:CH_disc}.
	\begin{theorem}[Unconditional unique solvability]  \label{thm:solvability} 
		Given $c_h^{n-1} \in S_h$,  there exists a unique solution $(\boldsymbol{c}_h^n,  
		\boldsymbol{\mu}_h^n) \in \boldsymbol{S}_h \times \boldsymbol{S}_h$ to \eqref{eq:CH_disc} for 
		any $\tau, h > 0$. 
	\end{theorem}
	\begin{proof}
		The proof is divided into two steps. We first show that problem \eqref{eq:CH_disc} is equivalent 
		to problem \eqref{eq:CH_disc_zero_mean} posed in $\bm{M}_h$, see Lemma 
		\ref{lemma:equivalence}. Then, we employ Brouwer's fixed point theorem and the Minty-Browder 
		theorem to show existence of solutions to \eqref{eq:CH_disc_zero_mean} in Lemma 
		\ref{lem:uniq_solv_p1} and Lemma \ref{lemma:uniq_solv_p2} respectively. The result then 
		immediately follows.\qed	
\end{proof}
	\subsection{An equivalent problem} 
	Following the strategy in \cite{liu2020priori},  we consider an equivalent problem to 
	\eqref{eq:CH_disc}: for any $1 \le n \le N$, find $(\boldsymbol{y}_h^n, 
	\boldsymbol{w}_h^n) \in \boldsymbol{M}_h \times \boldsymbol{S}_h$ \begin{subequations} 
		\label{eq:CH_disc_zero_mean}
		\begin{align}
			(\delta_\tau y_h^n, \mathring{\chi}_h)_{\Omega} + 
			a_{\mathcal{D}}(\boldsymbol{w}_h^n,
			\mathring{\boldsymbol{\chi}}_h) &= 0, && \forall \mathring{\boldsymbol{\chi}}_h \in 
			\boldsymbol{M}_h,
			\label{eq:CH_disc_zero_mean_a} \\
			(\Phi_+'(y_h^n + \overline{c_0}) + \Phi_-'(y_h^{n-1} + 
			\overline{c_0}), 
			\mathring{\phi}_h)_{\Omega} + \kappa 
			a_{\mathcal{D}}(\boldsymbol{y}_h^n,  \mathring{\boldsymbol{\phi}}_h) &= (w_h^n, 
			\mathring{\phi}_h)_{\Omega}, 
			&& 
			\forall \mathring{\boldsymbol{\phi}}_h \in \boldsymbol{M}_h,
			\label{eq:CH_disc_zero_mean_b}
		\end{align}
		where $y_h^{n-1} = c_h^{n-1} - \overline{c_0}$. 
	\end{subequations}

	\begin{lemma}\label{lemma:equivalence}
		The unique solvability of the problem defined by \eqref{eq:CH_disc} is equivalent to the unique 
		solvability of the problem defined by
		\eqref{eq:CH_disc_zero_mean}.
	\end{lemma}
	\begin{proof}
		Assume that $(\bm{c}_h^n, \bm{\mu}_h^n) \in \bm{S}_h \times \bm{S}_h$ is the unique solution 
		to \eqref{eq:CH_disc}. We now exhibit a  solution to \eqref{eq:CH_disc_zero_mean}. In particular, 
		we show that $\bm{y}_h^n = (c_h^n - \overline{c_0}, \hat{c}_h^n - \overline{c_0})$ and $\bm{w}_h^n = 
		(\mu_h^n - \overline{\mu_h^n}, \hat{\mu}_h^n - \overline{\mu_h^n})$  solve 
		\eqref{eq:CH_disc_zero_mean}.  Indeed, $\delta_\tau y_h^n = \delta_\tau c_h^n$ and  
		$a_\mathcal{D}(\bm{w}_h^n, \mathring{\bm{\chi}_h})=a_\mathcal{D}(\bm{\mu}_h^n, 
		\mathring{\bm{\chi}_h})$ from \eqref{eq:ad_0_constant}. Hence, it follows that 
		\eqref{eq:CH_disc_zero_mean_a} holds. Noticing that  $(\mu_h^n - \overline{\mu_h^n}, 
		\mathring{\phi}_h)_{\Omega} = (\mu_h^n   , \mathring{\phi}_h)_{\Omega}$, it also follows that 
		\eqref{eq:CH_disc_zero_mean_b} holds. 
		
		Conversely, assume that $(\boldsymbol{y}_h^n, \boldsymbol{w}_h^n) \in \boldsymbol{M}_h 
		\times \boldsymbol{M}_h$ solves \eqref{eq:CH_disc_zero_mean}. Let $f(y_h^n) = \Phi_+'(y_h^n 
		+ \overline{c_0}) + \Phi_-'(y_h^{n-1} + \overline{c}_0)$. We show that $\bm{c}_h^n = (y_h^n + 
		\overline{c_0}, \hat{y}_h^n + \overline{c_0})$ and $\bm{\mu}_h^n  = (w_h^n + \overline{f(y_h^n)},   \hat{ 
			w}_h^n + \overline{f(y_h^n)})$ solve \eqref{eq:CH_disc}. To this end, let $\bm{\chi}_h \in 
		\bm{S}_h$ and take $\mathring{\bm{\chi}}_h = (\chi_h - \overline{\chi_h}, \hat{\chi}_h - \overline{\chi_h}) 
		\in \bm{M}_h$ in \eqref{eq:CH_disc_zero_mean_a}. Since $\delta_\tau y_h^n \in M_h$, we have 
		that $(\delta_\tau y_h^n, \chi_h - \overline{\chi_h}) = (\delta_\tau y_h^n, \chi_h) = (\delta_\tau 
		c_h^n, \chi_h)$. In addition, from \eqref{eq:ad_0_constant}, $a_\mathcal{D}(\bm{w}_h^n, 
		\mathring{\bm{\chi}}_h) =  a_\mathcal{D}(\bm{w}_h^n, \bm{\chi}_h) = 
		a_\mathcal{D}(\bm{\mu}_h^n, \bm{\chi}_h)$.  Thus, \eqref{eq:CH_disc_a} holds. To show that 
		\eqref{eq:CH_disc_b} holds, let $\bm{\phi}_h \in \bm{S}_h$ and take $\mathring{\bm{\phi}}_h = 
		(\phi_h - \overline{\phi_h}, \hat{\phi}_h - \overline{\phi_h}) \in \bm{M}_h$ in 
		\eqref{eq:CH_disc_zero_mean_b}. Observe that from \eqref{eq:ad_0_constant},  $a_\mathcal{D} 
		(\bm{y}_h^n ,\mathring{\bm{\phi}}_h) = a_\mathcal{D} (\bm{c}_h^n ,\mathring{\bm{\phi}}_h) =   
		a_\mathcal{D} (\bm{c}_h^n ,\bm{\phi}_h) $. 
		Further, since $w_h^n \in M_h$, we obtain 
		\begin{align*} 
			(f(y_h^n) , \mathring{\phi}_h)_{\Omega} - (w_h^n, \mathring{\phi}_h)_{\Omega} &= (f(y_h^n), 
			\phi_h - \overline{\phi_h})_{\Omega} - (\mu_h^n - \overline{f(y_h^n)} , \phi_h - 
			\overline{\phi_h})_{\Omega} \\&  =   (f(y_h^n) - \mu_h^n, \phi_h)_{\Omega} + (w_h^n , 
			\overline{\phi_h})_{\Omega} = (f(y_h^n) - \mu_h^n, \phi_h)_{\Omega}. 
		\end{align*} 
		This shows that \eqref{eq:CH_disc_zero_mean} is satisfied. Since we have shown how to obtain a 
		solution to \eqref{eq:CH_disc} from \eqref{eq:CH_disc_zero_mean} and vice-versa, the unique 
		solvability of each problem follows from the unique solvability of the other problem via a standard 
		proof by contradiction.  \qed
	\end{proof}
	\subsection{Existence and uniqueness of the discrete solution to 
		\eqref{eq:CH_disc_zero_mean}} 
	To prove the existence and uniqueness of the discrete solution, we will rely on the following Lemma 
	\ref{lemma:brouwer} and Lemma \ref{lemma:minty_browder} 
	\cite{ciarlet2013linear}.
	\begin{lemma}[Corollary to Brouwer's fixed point theorem]\label{lemma:brouwer}
		Let $(V,(\cdot,\cdot)_V)$ be a finite-dimensional Hilbert space and let $f:V \to V$ be a 
		continuous 
		mapping with the following property: there exists an $M>0$ such that
		\begin{equation*}
			(f(v),v)_V \ge 0 \text{ for all } v \in V \text{ such that } \norm{v}_V = M.
		\end{equation*}
		Then, there exists $v_0 \in V$ such that
		\begin{equation}
			\norm{v_0}_V \le M \text{ and } f(v_0) = 0.
		\end{equation}
	\end{lemma}
	
	\begin{lemma}[Minty--Browder theorem] \label{lemma:minty_browder}
		Let $(V,\norm{\cdot}_V)$ be a separable and reflexive Banach space and let $A: V \to V'$ be 
		a 
		coercive and hemicontinuous monotone operator. Then, given any $f \in V'$, there exists a $u 
		\in 
		V$ such that
		\begin{equation*}
			A(u) = f.
		\end{equation*}
		Moreover, if $A$ is strictly monotone, then the solution $u$ is unique.
	\end{lemma}
	We now employ Lemma \ref{lemma:brouwer} to show that given $w_h^n \in M_h$, there exists a 
	unique $\bm{y}_h^n \in \boldsymbol{M}_h$ solving \eqref{eq:CH_disc_zero_mean_b}.
	\begin{lemma} \label{lem:uniq_solv_p1}
		Given $w_h^{n}, y_h^{n-1} \in M_h$, there exists a unique $\boldsymbol{y}_h^n \in 
		\boldsymbol{M}_h$ such that
		\begin{equation}	\label{eq:brouwer_eq}
			(\Phi_+'(y_h^n + \overline{c_0}) + 
			\Phi_-'(y_h^{n-1} + \overline{c_0}), \mathring{\phi}_h)_{\Omega} + 
			\kappa 
			a_{\mathcal{D}}(\boldsymbol{y}_h^n, \mathring{\boldsymbol{\phi}}_h) = (w_h^n, 
			\mathring{\phi}_h)_{\Omega}, \quad \forall \mathring{\boldsymbol{\phi}}_h\in 
			\boldsymbol{M}_h.
		\end{equation}
	\end{lemma}
	
	\begin{proof}
		For fixed $\boldsymbol{v}_h \in \boldsymbol{M}_h$, we define a mapping 
		$\boldsymbol{\mathcal{F}}: \boldsymbol{M}_h \to \boldsymbol{M}_h $ as follows: for all 
		$\mathring{\boldsymbol{\phi}_h} \in \boldsymbol{M}_h$, 
		\begin{equation} \label{eq:operator_F}
			(\boldsymbol{\mathcal{F}}(\boldsymbol{v}_h), \mathring{\boldsymbol{\phi}}_h)_{1,h} 
			= 	
			(\Phi_+'(v_h + \overline{c_0}) + 
			\Phi_-'(y_h^{n-1} + \overline{c_0}), \mathring{\phi}_h)_{\Omega} + 
			\kappa 
			a_{\mathcal{D}}(\boldsymbol{v}_h, \mathring{\boldsymbol{\phi}}_h) - (w_h^n, 
			\mathring{\phi}_h)_{\Omega}.
		\end{equation}
		%
%		\footnote{Indeed, for a given pair $\boldsymbol{v}_h \in \boldsymbol{M_h}$ the 
%			mapping 
%			%
%			\begin{equation}
%				\mathring{\boldsymbol{\phi}}_h \mapsto (\boldsymbol{\Phi}_+'(v_h + \overline{c}_0) + 
%				\boldsymbol{\Phi}_-'(y_h^{n-1} + \overline{c}_0), \mathring{\phi}_h)_{\mathcal{E}_h} + 
%				\kappa 
%				a_{\mathcal{D}}(\boldsymbol{v}_h, \mathring{\boldsymbol{\phi}}_h) - (w_h^n, 
%				\mathring{\phi}_h)_{\mathcal{E}_h},
%			\end{equation}
%			%
%			defines a bounded linear functional in $\boldsymbol{M}_h'$. To see this, note that by the 
%			Cauchy--Schwarz's inequality, the broken Poincar\'{e}'s inequality 
%\Cref{lem:broken_poincare2}, 
%			and 
%			the 
%			continuity of the bilinear form $a_{\mathcal{D}}$ \eqref{eq:aD_continuous}, we have
%			%
%			\begin{equation}
%				\begin{split}
%					\big|&(\boldsymbol{\Phi}_+'(v_h + \overline{c}_0) + 
%					\boldsymbol{\Phi}_-'(y_h^{n-1} + \overline{c}_0), \mathring{\phi}_h)_{\mathcal{E}_h} + 
%					\kappa 
%					a_{\mathcal{D}}(\boldsymbol{v}_h, \mathring{\boldsymbol{\phi}}_h) - (w_h^n, 
%					\mathring{\phi}_h)_{\mathcal{E}_h} \big|
%					\\
%					&\le C_P\del{ \norm{(\boldsymbol{\Phi}_+'(v_h + \overline{c}_0)}_{L^2(\Omega)} +
%						\norm[0]{\boldsymbol{\Phi}_-'(y_h^{n-1} + \overline{c}_0)}_{L^2(\Omega)} + 
%						\frac{C_\mathrm{cont}\kappa}{C_P} 
%						\norm{\boldsymbol{v}_h}_{1,h} + \norm[0]{w_h^n}_{L^2(\Omega)}} 
%					\norm[1]{\mathring{\boldsymbol{\phi}}_h}_{1,h} .
%				\end{split}
%			\end{equation}
%			%
%		}
		
		The mapping $\boldsymbol{\mathcal{F}}$ is well-defined thanks to the Riesz representation theorem.
		Choosing $\mathring{\boldsymbol{\phi}}_h= \boldsymbol{v}_h$ in \eqref{eq:operator_F} and 
		using the coercivity of the bilinear form $a_{\mathcal{D}}$ \eqref{eq:aD_coercive}, we find
		\begin{equation}
			\del{\boldsymbol{\mathcal{F}}(\boldsymbol{v}_h), \boldsymbol{v}_h}_{1,h}  \ge 
			(\Phi_+'(v_h + \overline{c_0}) + 
			\Phi_-'(y_h^{n-1} + \overline{c_0}), v_h)_{\Omega} + C_\mathrm{coer}\kappa 
			\norm{\boldsymbol{v}_h}_{1,h}^2  - (w^n_h, v_h)_{\Omega}.
		\end{equation}
		Expanding $\Phi_+'(v_h + \overline{c_0})$ around $\overline{c_0}$, we find 
		there exists $\xi_h$ between $\overline{c_0}$ and $v_h + \overline{c_0}$ such that
		\begin{equation}
			\Phi_+'(v_h + \overline{c_0})  = 	\Phi_+'(\overline{c_0})  + 
			v_h \Phi_+''(\xi_h),
		\end{equation}
		and thus by the convexity of $\Phi_+$ we have
		\begin{equation} \label{eq:convex_potential_positivity}
			(\Phi_+'(v_h  + \overline{c_0}) , v_h)_{\Omega}  = 	
			(\Phi_+'(\overline{c_0})  
			+ 
			\Phi_+''(\xi_h)v_h, v_h)_{\Omega} \ge 
			(\Phi_+'(\overline{c_0})  , v_h)_{\Omega} = 0.
		\end{equation}
		Next, the Cauchy-Schwarz's inequality, discrete Poincar\'{e}'s inequality 
		\eqref{eq:poincare}, and Young's inequality yield
		\begin{equation}
			( \Phi_-'(y_h^{n-1} + \overline{c_0}) - w_h^n, v_h)_{\Omega} \ge 
			-\frac{\epsilon}{2}  
			\norm[1]{\Phi_-'(y_h^{n-1} + 
				\overline{c_0}) - w_h^n}_{L^2(\Omega)}^2 - \frac{C_P^2}{2\epsilon} 
			\norm{\boldsymbol{v}_h}_{1,h}^2, \label{eq:bounding_ph_-}
		\end{equation}
		so that, upon choosing $\epsilon = C_P^2/(C_\mathrm{coer}\kappa)$, we find
		\begin{equation}
			\del{\boldsymbol{\mathcal{F}}(\boldsymbol{v}_h), \boldsymbol{v}_h}_{1,h}  \ge   
			\frac{C_\mathrm{coer}\kappa}{2} 
			\norm{\boldsymbol{v}_h}_{1,h}^2  - \frac{C_P^2}{2C_\mathrm{coer}\kappa}  
			\norm[1]{\Phi_-'(y_h^{n-1} + 
				\overline{c_0}) - w_h^n}_{L^2(\Omega)}^2.
		\end{equation}
		Define the sphere $\Xi$ in $\boldsymbol{M}_h$ as follows:
		\begin{equation}
			\Xi := \cbr{ \boldsymbol{v}_h \in \boldsymbol{M}_h \; : \;  \norm{\boldsymbol{v}_h}_{1,h}^2 = 
				\frac{C_P^2}{C_\mathrm{coer}^2 \kappa^2}  
				\norm[1]{\Phi_-'(y_h^{n-1} + 
					\overline{c_0}) - w_h^n}_{L^2(\Omega)}^2 }.
		\end{equation}
		It holds that $\del{\boldsymbol{\mathcal{F}}(\boldsymbol{v}_h), \boldsymbol{v}_h}_{1,h} \ge 
		0$ 
		for any 
		$\boldsymbol{v}_h \in \Xi$. Consequently, Brouwer's fixed point theorem guarantees the existence of a  function 
		$\boldsymbol{y}_h^n
		\in \boldsymbol{M}_h$ such that $\boldsymbol{\mathcal{F}}(\boldsymbol{y}_h^n) = 0$ and
		\begin{equation} \label{eq:apriori_est_yh}
			\norm{\boldsymbol{y}_h^n}_{1,h}^2 \le 	\frac{C_P^2}{C_\mathrm{coer}^2 \kappa^2}  
			\norm[1]{\Phi_-'(y_h^{n-1} + 
				\overline{c_0}) - w_h^n}_{L^2(\Omega)}^2.
		\end{equation}
		Equivalently,
		$\boldsymbol{y}_h^n
		\in \boldsymbol{M}_h$ satisfies \eqref{eq:brouwer_eq}.
		
		Lastly, we show that the solution $\boldsymbol{y}_h^n
		\in \boldsymbol{M}_h$ solving \eqref{eq:brouwer_eq} is unique. To this end, assume 
		$\boldsymbol{y}_1 \in 
		\boldsymbol{M}_h$ and 
		$\boldsymbol{y}_2 \in \boldsymbol{M}_h$ are two solutions of \eqref{eq:brouwer_eq}. In both 
		cases, we 
		test
		\eqref{eq:brouwer_eq} with $\mathring{\boldsymbol{\phi}}_h= \boldsymbol{y}_1 - 
		\boldsymbol{y}_2$, 
		subtract 
		the 
		two resulting 
		equations, and use 
		the coercivity of the bilinear form $a_{\mathcal{D}}$ to find
		\begin{equation}
			C_\mathrm{coer} \kappa \norm{\boldsymbol{y}_1 - \boldsymbol{y}_2}_{1,h}^2 \le - \del{ 
				\Phi_+'(y_1 
				+ \overline{c_0}) - \Phi_+'(y_2 + \overline{c_0}), y_1 - y_2}_{\Omega} .
		\end{equation}
		As $\Phi_{+} \in \mathcal{C}^2$ is convex, $\Phi_{+}'$ is nondecreasing, and thus
		\begin{equation}
			\del{ \Phi_+'(y_1 + 
				\overline{c_0}) - \Phi_+'(y_2 + \overline{c_0}), y_1 - y_2}_{\Omega} \ge 0. 
			\label{eq:convexityphi}
		\end{equation}
		Hence $\norm{\boldsymbol{y}_1 - \boldsymbol{y}_2}_{1,h}^2 \le 0$, so that 
		$\boldsymbol{y}_1 
		= 
		\boldsymbol{y}_2$  as required.
		\qed
	\end{proof}
	We are now ready to show that \eqref{eq:CH_disc_zero_mean} is well posed by employing Lemma 
	\ref{lemma:minty_browder} and Lemma \ref{lem:uniq_solv_p1}. 
	\begin{lemma}\label{lemma:uniq_solv_p2}
	The scheme \eqref{eq:CH_disc_zero_mean} is 
		uniquely solvable for any fixed $\tau$, $h$, and $\kappa$. 
	\end{lemma}
	
	\begin{proof}
		For any $\boldsymbol{w}_h = (w_h,\hat{w}_h) \in \boldsymbol{M}_h$, let 
		$\boldsymbol{y}_h^n(w_h)\in \boldsymbol{M}_h$ be the unique solution to 
		\eqref{eq:brouwer_eq}.
		There exists a well-defined operator $\boldsymbol{\mathcal{G}} : \boldsymbol{M}_h \to 
		\boldsymbol{M}_h'$ satisfying
%		\footnote{We note that the mapping
%			%
%			\begin{equation}
%				\mathring{\boldsymbol{\chi}}_h \mapsto (y_h^n - 
%				y_h^{n-1},\mathring{\chi}_h)_{\Omega} + \tau 
%				a_{\mathcal{D}}(\boldsymbol{w}_h,\mathring{\boldsymbol{\chi}}_h)
%			\end{equation}
%			%
%			defines a bounded linear functional on $\boldsymbol{M}_h$. Indeed, by the 
%			Cauchy--Schwarz's
%			inequality, the discrete Poincar\'{e}'s inequality, and the continuity of the bilinear form 
%			$a_{\mathcal{D}}$,
%			we have
%			%
%			\begin{equation}
%				\big| (y_h^n - 
%				y_h^{n-1},\mathring{\chi}_h)_{\Omega} + \tau a_{\mathcal{D}}(\boldsymbol{w}_h, 
%				\mathring{\boldsymbol{\chi}}_h) \big| \le \del{ C_P\norm[0]{y_h^n - 
%						y_{h}^{n-1}}_{L^2(\Omega)} + C_\mathrm{cont}\tau 
%					\norm{\boldsymbol{w}_h}_{1,h}}  \norm{\mathring{\boldsymbol{\chi}}_h}_{1,h}.
%		\end{equation} }
		%
		\begin{equation} \label{eq:operator_G}
			\big \langle \boldsymbol{\mathcal{G}}(\boldsymbol{w}_h), \mathring{\boldsymbol{\chi}}_h
			\big 
			\rangle 
			= 
			(y_h^n - 
			y_h^{n-1},\mathring{\chi}_h)_{\Omega} + \tau a_{\mathcal{D}}(\boldsymbol{w}_h, 
			\mathring{\boldsymbol{\chi}}_h), \quad\forall \boldsymbol{\chi}_h\in\boldsymbol{M}_h.
		\end{equation}
		We begin by showing $\boldsymbol{\mathcal{G}}$ is coercive. Take 
		$\mathring{\boldsymbol{\chi}}_h = \boldsymbol{w}_h$ in \eqref{eq:operator_G} and note that 
		by 
		the coercivity of the 
		bilinear 
		form $a_{\mathcal{D}}(\cdot,\cdot)$ \eqref{eq:aD_coercive}, we have that 
		\begin{equation} 
			\big \langle \boldsymbol{\mathcal{G}}(\boldsymbol{w}_h), \boldsymbol{w}_h \big 
			\rangle 
			\ge
			(y_h^n - 
			y_h^{n-1}, w_h)_{\Omega} + C_\mathrm{coer} \tau \norm{\boldsymbol{w}_h}_{1,h}^2.
		\end{equation}
		The Cauchy--Schwarz's, Poincar\'{e}'s \eqref{eq:poincare}, and Young's
		inequalities 
		then yield for any $\epsilon_1 > 0$: 
		\begin{equation} \label{eq:coer_1}
			\big \langle \boldsymbol{\mathcal{G}}(\boldsymbol{w}_h), \boldsymbol{w}_h \big 
			\rangle 
			\ge
			(y_h^n , w_h)_{\Omega} - \frac{\epsilon_1}{2} \norm[1]{
				y_h^{n-1}}_{L^2(\Omega)}^2 + \del[2]{C_\mathrm{coer} \tau-  \frac{C_P^2}{2\epsilon_1} } 
			\norm{\boldsymbol{w}_h}_{1,h}^2.
		\end{equation}
		We now aim to find a lower bound on $(y_h^n, w_h)_{\Omega} $. To this end, we test 
		\eqref{eq:CH_disc_zero_mean_b} 
		with 
		$\mathring{\boldsymbol{\phi}}_h = \boldsymbol{y}_h^n$ and once again use the coercivity of 
		the 
		bilinear 
		form 
		$a_{\mathcal{D}}$ \eqref{eq:aD_coercive}:
		\begin{equation}
			(w_h, y_h^n)_{\Omega} \ge (\Phi_+'(y_h^n + \overline{c_0}) + 
			\Phi_-'(y_h^{n-1} + \overline{c_0}), y_h^n)_{\Omega} + 
			C_\mathrm{coer} \kappa \norm{\boldsymbol{y}_h^n}_{1,h}^2. 
		\end{equation}
		Proceeding as in \Cref{lem:uniq_solv_p1}, see  the derivation of 
		\eqref{eq:convex_potential_positivity} \eqref{eq:bounding_ph_-}, we find that
		\begin{equation}
			(\Phi_+'(y_h^n + \overline{c_0}), y_h^n)_{\Omega} \ge 0,
		\end{equation}
		and for any $\epsilon_2 >0$ that 
		\begin{equation} \label{eq:concave_potential_bound}
			( \Phi_-'(y_h^{n-1} + \overline{c_0}), y_h^n)_{\Omega} \ge 
			-\frac{\epsilon_2}{2}  
			\norm[1]{\Phi_-'(y_h^{n-1} + 
				\overline{c_0}) }_{L^2(\Omega)}^2 - \frac{C_P^2}{2\epsilon_2} 
			\norm{\boldsymbol{y}_h^n}_{1,h}^2. 
		\end{equation}
		Thus, we obtain 
		\begin{equation} \label{eq:coer_2}
			(w_h, y_h^n)_{\Omega} \ge 
			\del[2]{ C_\mathrm{coer} \kappa - \frac{C_P^2}{2\epsilon_2}} 
			\norm{\boldsymbol{y}_h^n}_{1,h}^2 	
			-\frac{\epsilon_2}{2}  
			\norm[1]{\Phi_-'(y_h^{n-1} + 
				\overline{c_0}) }_{L^2(\Omega)}^2 .
		\end{equation}
		Combining \eqref{eq:coer_1} and \eqref{eq:coer_2}, we have
		\begin{multline} 
			\big \langle \boldsymbol{\mathcal{G}}(\boldsymbol{w}_h), \boldsymbol{w}_h \big 
			\rangle 
			\ge  \del[2]{C_\mathrm{coer}\tau -  \frac{C_P^2}{2\epsilon_1} } 
			\norm{\boldsymbol{w}_h}_{1,h}^2 + \del[2]{ C_\mathrm{coer} \kappa - 
				\frac{C_P^2}{2\epsilon_2}} 
			\norm{\boldsymbol{y}_h^n}_{1,h}^2 	\\  - \frac{\epsilon_1}{2} \norm[1]{
				y_h^{n-1}}_{L^2(\Omega)}^2 -\frac{\epsilon_2}{2}  
			\norm[1]{\Phi_-'(y_h^{n-1} + 
				\overline{c_0}) }_{L^2(\Omega)}^2.
		\end{multline}
		Choosing $\epsilon_1 
		=C_P^2/(C_\mathrm{coer}\tau),  \epsilon_2 = C_P^2/(2 C_\mathrm{coer} \kappa )$  we find 
		\begin{equation} 
			\big \langle \boldsymbol{\mathcal{G}}(\boldsymbol{w}_h), \boldsymbol{w}_h \big 
			\rangle 
			\ge  \frac{C_\mathrm{coer} \tau }{2}
			\norm{\boldsymbol{w}_h}_{1,h}^2  
			- \frac{C_P^2}{4 C_\mathrm{coer} \kappa} \norm[1]{\Phi_-'(y_h^{n-1} + 
				\overline{c_0}) }_{L^2(\Omega)}^2 - \frac{C_P^2}{C_\mathrm{coer}\tau} \Vert 
			y_h^{n-1}\Vert_{L^2(\Omega)}^2, 
		\end{equation}
		and thus
		\begin{equation}
			\lim_{ \norm[0]{\boldsymbol{w}_h}_{1,h} \to +\infty} 	\frac{\big \langle 
				\boldsymbol{\mathcal{G}}(\boldsymbol{w}_h), \boldsymbol{w}_h \big 
				\rangle }{	\norm{\boldsymbol{w}_h}_{1,h} } = + \infty.
		\end{equation}

		Next, we show that the operator $\boldsymbol{\mathcal{G}}$ is bounded. 
		For any test function $\mathring{\boldsymbol{\chi}}_h \in \boldsymbol{M}_h$, the 
		Cauchy--Schwarz's inequality, Poincar\'{e} inequality's \eqref{eq:poincare}, and 
		continuity 
		of 
		the bilinear form 
		$a_{\mathcal{D}}$ \eqref{eq:aD_continuous} yield
		\begin{equation} \label{eq:G_bnd_1}
			\begin{split}
				|	\big \langle \boldsymbol{\mathcal{G}}(\boldsymbol{w}_h), 
				\mathring{\boldsymbol{\chi}}_h
				\big 
				\rangle |
				& \le \norm{y_h^n}_{L^2(\Omega)}\norm{\mathring{\chi}_h}_{L^2(\Omega)} + 
				\norm[0]{y_h^{n-1}}_{L^2(\Omega)}\norm{\mathring{\chi}_h}_{L^2(\Omega)} + \tau 
				a_{\mathcal{D}}(\boldsymbol{w}_h, \mathring{\boldsymbol{\chi}}_h) \\
				& \le \del{ C_\mathrm{cont} \tau \norm{\boldsymbol{w}_h}_{1,h} + 
					C_P^2  \norm{\boldsymbol{y}_h^n}_{1,h} + C_P 
					\norm[0]{y_h^{n-1}}_{L^2(\Omega)}} 
				\norm{\mathring{\boldsymbol{\chi}}_h}_{1,h}.
			\end{split}
		\end{equation}
		As $\boldsymbol{y}_h^n(w_h) \in \boldsymbol{M}_h$ is the unique solution of 
		\eqref{eq:brouwer_eq}, 
		we take
		$\mathring{\boldsymbol{\phi}}_h = \boldsymbol{y}_h^n$ in \eqref{eq:brouwer_eq} and use
		the coercivity of $a_{\mathcal{D}}$ to find 
		\begin{equation}
			C_\mathrm{coer} \kappa \norm{\boldsymbol{y}_h^n}_{1,h} \le
			-(\Phi_+'(y_h^n + \overline{c_0}) + 
			\Phi_-'(y_h^{n-1} + \overline{c_0}), y_h^n)_{\Omega}  + (w_h^n, 
			y_h^n)_{\Omega}.
		\end{equation}
		Recall from \eqref{eq:convex_potential_positivity} and \eqref{eq:concave_potential_bound} that 
		there 
		exists some $\epsilon_3 >0$ such that
		\begin{equation}
			\begin{split}
				C_\mathrm{coer} \kappa \norm{\boldsymbol{y}_h^n}_{1,h}^2 &\le
				\frac{\epsilon_3}{2}  
				\norm[1]{\Phi_-'(y_h^{n-1} + 
					\overline{c_0}) }_{L^2(\Omega)}^2 + \frac{C_P^2}{2\epsilon_3} 
				\norm{\boldsymbol{y}_h^n}_{1,h}^2 + (w_h^n, 
				y_h^n)_{\Omega} \\
				& \le 	\frac{\epsilon_3}{2}  
				\norm[1]{\Phi_-'(y_h^{n-1} + 
					\overline{c_0}) }_{L^2(\Omega)}^2 + \frac{C_P^2}{\epsilon_3} 
				\norm{\boldsymbol{y}_h^n}_{1,h}^2 	 + 
				\frac{C_P^2\epsilon_3}{2}\norm{\boldsymbol{w}_h}_{1,h}^2,
			\end{split}
		\end{equation}
		where we have used the Cauchy--Schwarz's, broken Poincar\'{e}'s, and Young's inequalities to 
		obtain 
		the 
		second inequality. Choosing $\epsilon_3 = 
		(2C_P^2)/(C_\mathrm{coer} \kappa)$ and rearranging yields
		%n
		\begin{equation}
			\norm{\boldsymbol{y}_h^n}_{1,h}^2 \le 	\frac{2 C_P^2}{C_\mathrm{coer}^2 \kappa^2}  
			\norm[1]{\Phi_-'(y_h^{n-1} + 
				\overline{c_0}) }_{L^2(\Omega)}^2 + 
			\frac{2 C_P^4}{C_\mathrm{coer}^2\kappa^2}\norm{\boldsymbol{w}_h}_{1,h}^2,
		\end{equation}
		and using the fact that $\sqrt{a+b} \le \sqrt{a} + \sqrt{b}$ for $a,b \ge 0$, we find
		\begin{equation} \label{eq:G_bnd_2}
			\norm{\boldsymbol{y}_h^n}_{1,h} \le 	\frac{\sqrt{2} C_P}{C_\mathrm{coer} \kappa}  
			\norm[1]{\Phi_-'(y_h^{n-1} + 
				\overline{c_0}) }_{L^2(\Omega)} + 
			\frac{\sqrt{2} C_P^2}{C_\mathrm{coer}\kappa}\norm{\boldsymbol{w}_h}_{1,h}.
		\end{equation}
		Combining \eqref{eq:G_bnd_1} and \eqref{eq:G_bnd_2}, 
		\begin{equation} 
			\begin{split}
				\big|	\big \langle \boldsymbol{\mathcal{G}}(\boldsymbol{w}_h), 
				\mathring{\boldsymbol{\chi}}_h
				\big 
				\rangle \big|
				& \le \del[3]{ \del[2]{C_\mathrm{cont} \tau + \frac{\sqrt{2}C_P^4}{C_\mathrm{coer}\kappa} 
					}\norm{\boldsymbol{w}_h}_{1,h} 
					+ 
					\frac{\sqrt{2}C_P^3}{C_\mathrm{coer} \kappa}  
					\norm[1]{\Phi_-'(y_h^{n-1} + 
						\overline{c}_0) }_{L^2(\Omega)}  + C_P \norm[0]{y_h^{n-1}}_{L^2(\Omega)}} 
				\norm{\mathring{\boldsymbol{\chi}}_h}_{1,h}.
			\end{split} 
		\end{equation}
		Consequently, it holds that
		\begin{equation}
			\begin{split}
				\norm{\boldsymbol{\mathcal{G}}(\boldsymbol{w}_h)}_{  \boldsymbol{M}_h'} & = 
				\sup_{ \substack{\mathring{\boldsymbol{\chi}}_h \in \boldsymbol{M}_h \\ 
						\norm[0]{\mathring{\boldsymbol{\chi}}_h}_{1,h} 
						= 
						1}} \big|	\big 
				\langle 
				\boldsymbol{\mathcal{G}}(\boldsymbol{w}_h), \mathring{\boldsymbol{\chi}}_h \big 
				\rangle \big| \\
				& \le  \del[2]{C_\mathrm{cont} \tau + \frac{\sqrt{2} C_P^4}{C_\mathrm{coer}\kappa} 
				}\norm{\boldsymbol{w}_h}_{1,h} 
				+ 
				\frac{\sqrt{2}C_P^3}{C_\mathrm{coer} \kappa}  
				\norm[1]{\Phi_-'(y_h^{n-1} + 
					\overline{c}_0) }_{L^2(\Omega)}  + C_P \norm[0]{y_h^{n-1}}_{L^2(\Omega)}.
			\end{split}
		\end{equation}
		In other words, the operator $\boldsymbol{\mathcal{G}}$ maps bounded sets in 
		$\boldsymbol{M}_h$ 
		to bounded 
		sets in 
		$\boldsymbol{M}_h'$.
		
		Next, we show hemicontinuity of the operator $\boldsymbol{\mathcal{G}}$. It suffices to 
		show for 
		any $\boldsymbol{w}_h, \boldsymbol{v}_h, \mathring{\boldsymbol{\chi}}_h \in 
		\boldsymbol{M}_h$ 
		the 
		(sequential) continuity of the mapping $g:\mathbb{R}\rightarrow \mathbb{R}$ defined by 
		\begin{equation}
			g(t) =  	\big \langle \boldsymbol{\mathcal{G}}(\boldsymbol{w}_h+ 
			t\boldsymbol{v}_h), 
			\mathring{\boldsymbol{\chi}}_h \big 
			\rangle.
		\end{equation}
		To this end, take an arbitrary $t^\star \in \mathbb{R}$ and a sequence $(t_i)_{i \in 
			\mathbb{N}} 
		\subset 
		\mathbb{R}$ such 
		that $t_i \to t^\star$ as $i \to \infty$. From the definition of $\boldsymbol{\mathcal{G}}$ 
		\eqref{eq:operator_G} and the 
		bilinearity and continuity of the form $a_{\mathcal{D}}$ \eqref{eq:aD_continuous}, we find that
		\begin{equation} \label{eq:hemi_cont_1}
			\begin{split}
				\big|\big \langle \boldsymbol{\mathcal{G}}(\boldsymbol{w}_h + t^\star 
				\boldsymbol{v}_h), &
				\mathring{\boldsymbol{\chi}}_h \big 
				\rangle - \big \langle \boldsymbol{\mathcal{G}}(\boldsymbol{w}_h + t_i 
				\boldsymbol{v}_h), \mathring{\boldsymbol{\chi}}_h  \big 
				\rangle \big|
				\\
				&=  \big| 	(y_h^n(w_h + t^\star v_h) - 
				y_h^n(w_h + t_i v_h),\mathring{\chi}_h)_{\Omega}  + \tau 
				(t^\star - t_i) a_{\mathcal{D}} ( \boldsymbol{v}_h, \mathring{\boldsymbol{\chi}}_h ) 
				\big|. 
				%	\\
				%	& \le C_\mathrm{cont}\tau \big|t^\star - t_i| \norm{(v_h,\hat{v}_h)}_{1,h} 
				%	\norm{(\mathring{\chi}_h,\hat{\chi}_h)}_{1,h}  \to 0, \text{ as } t_i \to \infty.
			\end{split}
		\end{equation}
		Given $w_h + t^\star v_h \in M_h$ and 
		$w_h + t_i v_h \in M_h$, we have respectively the 
		solutions
		$\boldsymbol{y}_h^n(w_h + t^\star v_h)$ and 
		$\boldsymbol{y}_h^n(w_h + t_i 
		v_h)$ to \eqref{eq:brouwer_eq}. We can choose 
		$\mathring{\boldsymbol{\phi}}_h =  \boldsymbol{y}_h^n(w_h + t^\star 
		v_h) - 
		\boldsymbol{y}_h^n(w_h + t_i v_h)$ in \eqref{eq:brouwer_eq}, 
		subtract 
		the 
		two resulting equations, and use the bilinearity and coercivity of $a_{\mathcal{D}}$ 
		\eqref{eq:aD_coercive}
		to find:
		%
		%%%\textcolor{ForestGreen}{ Fixing typos above:	\begin{equation}
		%%%\begin{split}
		%%%&C_\mathrm{coer} \kappa \norm{\boldsymbol{y}_h^n(w_h + t^\star v_h) - 
		%%%\boldsymbol{y}_h^n(w_h + 
		%%%t_k v_h)}_{1,h}^2 \\& \le 
		%%%-(\Phi_+'(y_h^n(w_h + t^\star v_h) + 
		%%%\overline{c_0})  - 
		%%%\Phi_+'(y_h^n(w_h + 
		%%%t_k v_h) + \overline{c_0}),  y_h^n(w_h + t^\star 
		%%%v_h) \Rd - \Bk  
		%%%y_h^n(w_h + t_k 
		%%%v_h))_{\mathcal{E}_h}
		%%%\\&  \quad + (t^\star - t_k) \del{ y_h^n(w_h + t^\star 
		%%%v_h) \Rd - \Bk  
		%%%y_h^n(w_h + t_k 
		%%%v_h), 
		%%%v_h }_{\mathcal{E}_h}  \\
		%%%& \le C_P|t^\star-t_k| \norm{v_h}_{L^2(\Omega)}  
		%%%\norm{\boldsymbol{y}_h^n(w_h + 
		%%%t^\star v_h) - 
		%%%\boldsymbol{y}_h^n(w_h + 
		%%%t_k 
		%%%v_h) }_{1,h},
		%%%\end{split}
		%%%\end{equation} }
		\begin{align*}
			&C_\mathrm{coer} \kappa \norm{\boldsymbol{y}_h^n(w_h + t^\star v_h) - 
				\boldsymbol{y}_h^n(w_h + t_i v_h)}_{1,h}^2 \\& 
			\leq -(\Phi_+'(y_h^n(w_h + t^\star v_h) + 
			\overline{c_0}) - \Phi_{+}'(y_h^n(w_h + t_i v_h) + 
			\overline{c_0}) , \mathring{\phi_h} )_{\Omega} + (t^\star - 
			t_i)(v_h,\mathring{\phi_h})_{\Omega}. 
		\end{align*}
		For the first term, we use \eqref{eq:convexityphi} to have
		\begin{align*}
			(\Phi_{+}^{'}(y_h^n(w_h + t^\star v_h) + \overline{c_0}) 
			- \Phi_{+}'(y_h^n(w_h + t_i v_h) + \overline{c_0}), y_h^n(w_h + t^\star v_h)-y_h^n(w_h + t_i 
			v_h))_\Omega \geq 0. 
		\end{align*}
		Thus, with Cauchy--Schwarz's inequality and the discrete Poincar\'{e}'s inequality 
		\eqref{eq:poincare}, we obtain: 
		\begin{align*}
			C_\mathrm{coer} \kappa \norm{\boldsymbol{y}_h^n(w_h + t^\star v_h) - 
				\boldsymbol{y}_h^n(w_h + t_i v_h)}_{1,h} 
			\leq C_P^2 |t^\star - t_i|\|\boldsymbol{v}_h\|_{1,h}. 
		\end{align*}
		Using this bound in \eqref{eq:hemi_cont_1}, we obtain 
		\begin{align*}
			\big|\big \langle \boldsymbol{\mathcal{G}}(\boldsymbol{w}_h + t^\star 
			\boldsymbol{v}_h), &
			\mathring{\boldsymbol{\chi}}_h  \big 
			\rangle - \big \langle \boldsymbol{\mathcal{G}}(\boldsymbol{w}_h + t_i 
			\boldsymbol{v}_h), 
			\mathring{\boldsymbol{\chi}}_h \big 
			\rangle \big|
			\\
			&\le 	\del[2]{ \frac{C_P^4 }{C_\mathrm{coer}\kappa} + C_\mathrm{cont} \tau }|t^\star 
			-t_i| 
			\norm{\boldsymbol{v}_h}_{1,h} 
			\norm{\mathring{\boldsymbol{\chi}}_h}_{1,h}.
		\end{align*}
		Since the right-hand side above tends to zero as $i$ tends to infinity, this shows that $\boldsymbol{\mathcal{G}}$ is hemicontinuous.
		Finally, we show the strict monotonicity of $\boldsymbol{\mathcal{G}}$. For any 
		$\boldsymbol{w}_h = (w_h,\hat{w}_h), \boldsymbol{s}_h = (s_h,\hat{s}_h) \in \boldsymbol{M}_h$, 
		it holds that 
		\begin{equation}
			\begin{split}
				\big\langle &\boldsymbol{\mathcal{G}}(\boldsymbol{w}_h) - 
				\boldsymbol{\mathcal{G}}(\boldsymbol{s}_h),  
				\boldsymbol{w}_h - \boldsymbol{s}_h \big 
				\rangle \\ & =
				(y_h^n(w_h) - y_h^{n}(s_h), w_h - s_h)_{\Omega} + \tau 
				a_{\mathcal{D}}(\boldsymbol{w}_h - \boldsymbol{s}_h, \boldsymbol{w}_h - 
				\boldsymbol{s}_h) \\
				& \ge 	(y_h^n(w_h) - y_h^{n}(s_h),w_h - s_h)_{\Omega} + C_\mathrm{coer}\tau 
				\norm{\boldsymbol{w}_h - \boldsymbol{s}_h}_{1,h}^2,
			\end{split}
		\end{equation}
		where we have used the coercivity of the bilinear form $a_{\mathcal{D}}$ 
		\eqref{eq:aD_coercive}.
		To each given $w_h, s_h \in M_h$, we have  
		corresponding 
		unique solutions $\boldsymbol{y}_h^n(w_h)$ and 
		$\boldsymbol{y}_h^n(s_h)$ to 
		\eqref{eq:brouwer_eq}. 
		Therefore, by testing \eqref{eq:brouwer_eq} with $\mathring{\boldsymbol{\phi}}_h= 
		\boldsymbol{y}_h^n(w_h) - \boldsymbol{y}_h^n(s_h)\in 
		\boldsymbol{M}_h$ for 
		each solution and subtracting the two resulting equations yields
		\begin{equation}
			\begin{split}
				(y_h^n(w_h) - 
				y_h^{n}(s_h),w_h - s_h)_{\Omega} &= \del{ 
					\Phi_{+}'(y_h^n(w_h) + 
					\overline{c_0}) 
					-  
					\Phi_{+}'(y_h^n(s_h) + \overline{c_0}), y_h^n(w_h) - 
					y_h^{n}(s_h) } \\
				&+ \kappa a_{\mathcal{D}}( \boldsymbol{y}_h^n(w_h) - 
				\boldsymbol{y}_h^n(s_h), \boldsymbol{y}_h^n(w_h) - 
				\boldsymbol{y}_h^n(s_h)) \\
				& \ge C_\mathrm{coer} \kappa \norm{\boldsymbol{y}_h(w_h) - 
					\boldsymbol{y}_h(s_h)}_{1,h}^2 \ge 0,
			\end{split}
		\end{equation}
		where we have used the coercivity of $a_{\mathcal{D}}$ \eqref{eq:aD_coercive} and 
		\eqref{eq:convexityphi}. 
		%the fact 
		%		that 
		%		$\Phi_+ 
		%		\in C^2$ is convex and hence $\Phi_+^-$ is nondecreasing. 
		Consequently,
		\begin{equation} \label{eq:monotone}
			\begin{split}
				\big\langle &\boldsymbol{\mathcal{G}}(\boldsymbol{w}_h) - 
				\boldsymbol{\mathcal{G}}(\boldsymbol{s}_h),  
				\boldsymbol{w}_h - \boldsymbol{s}_h \big 
				\rangle  \ge C_\mathrm{coer}\tau \norm{\boldsymbol{w}_h - \boldsymbol{s}_h}_{1,h}^2 \ge 0,
			\end{split}
		\end{equation}
		and thus $\boldsymbol{\mathcal{G}}_h$ is monotone. Moreover, as 
		$\norm[0]{\cdot}_{1,h}$ 
		defines 
		a norm 
		on $\boldsymbol{M}_h$, it immediately follows that the inequality in \eqref{eq:monotone} is 
		strict if 
		$\boldsymbol{w}_h \ne \boldsymbol{s}_h$; hence, $\boldsymbol{\mathcal{G}}_h$ is strictly 
		monotone.
		
		We have now verified all of the assumptions of the Minty--Browder theorem, and we conclude 
		that there 
		exists a unique solution $\boldsymbol{w}_h^n$ satisfying
		\begin{equation}
			\big \langle \boldsymbol{\mathcal{G}}(\boldsymbol{w}_h^n), 
			\mathring{\boldsymbol{\chi}}_h
			\big 
			\rangle = 0, \quad \forall \mathring{\boldsymbol{\chi}}_h \in \boldsymbol{M}_h,
		\end{equation}
		implying that $(\boldsymbol{y}_h^n(w_h^n),  \boldsymbol{w}_h^n) \in 
		\boldsymbol{M}_h \times \boldsymbol{M}_h$ is the unique solution to 	
		\eqref{eq:CH_disc_zero_mean}.
		\qed
	\end{proof}

	\section{Stability analysis for $\mathcal{C}^2$ chemical energy density}
	In this section, we  show stability estimates for the discrete solutions. 
	We will make use of the following operator.  
	We define an operator $\boldsymbol{\mathcal{J}}_h : M_h \to \boldsymbol{M}_h$ such that
	\begin{equation}
		a_{\mathcal{D}}(\boldsymbol{\mathcal{J}}_h(w_h), \boldsymbol{v}_h) = 
		(w_h,v_h)_{\Omega}, \quad \forall \boldsymbol{v}_h \in \boldsymbol{M}_h,
	\end{equation}
	which is well-defined by the Lax--Milgram theorem. 

	\begin{lemma} \label{lem:alternate_green_op_properties}
		For all $w_h \in M_h$ it holds that
		\begin{equation}
			a_{\mathcal{D}}(\boldsymbol{\mathcal{J}}_h(w_h), \boldsymbol{v}_h) = 
			(w_h,v_h)_{\Omega}, \quad \forall \boldsymbol{v}_h \in \boldsymbol{S}_h.
		\end{equation}
		Further, it holds that
		\begin{equation} \label{eq:alternate_green_op_bnd}
			|(w_h,v)_{\Omega}| \lesssim \norm[0]{\boldsymbol{\mathcal{J}}_h(w_h)}_{1,h} 
			\norm[0]{v}_{\DG}, \quad \forall  v \in H^1(\mathcal{E}_h), \quad \forall w_h \in M_h.
		\end{equation}
	\end{lemma}
	\begin{proof}
		Fix $w_h\in M_h$ and $\boldsymbol{v}_h \in \boldsymbol{S}_h$.  Set 
		$\overline{\boldsymbol{v}_h} = 
		(\overline{v_h}, \overline{v_h}|_{\Gamma_h}) \in 
		\boldsymbol{S}_h$. With \eqref{eq:ad_0_constant},  we have 
		\begin{equation}
				a_{\mathcal{D}}(\boldsymbol{\mathcal{J}}_h (w_h), \boldsymbol{v}_h) = 
				a_{\mathcal{D}}(\boldsymbol{\mathcal{J}}_h (w_h), \boldsymbol{v}_h 
				-\overline{\boldsymbol{v}_h}) = (w_h, v_h - \overline{v_h})_{\Omega} =  (w_h,v_h)_{\Omega}.
		\end{equation}
		%
		%since $a_{\mathcal{D}}(\boldsymbol{\mathcal{J}}_h (w_h), 
		%\overline{\boldsymbol{v}}_h) = 0$\footnote{See the proof of \Cref{lem:disc_neg_norm} 
		%\eqref{eq:disc_dual_CS}.} and $w_h \in 
		%M_h$. 
		The proof of \eqref{eq:alternate_green_op_bnd} is  in \ref{sec:proof_alternate_green_op_bnd}. \qed
	\end{proof}
	\begin{theorem}[Unconditional energy stability] \label{thm:stability}	The following bound holds  
		for all $1\leq m\leq N$: 
		\begin{multline} \label{eq:uniform_energy_bnd}
			( \Phi(c_h^m), 1)_{\Omega} + \frac{C_\mathrm{coer} 
				\kappa}{2}\norm[0]{\boldsymbol{c}_h^m}_{1,h}^2 + C_\mathrm{coer} \tau 
			\sum_{n=1}^m \del{
				\norm[0]{\boldsymbol{\mu}_h^n}_{1,h}^2 +  \frac{\kappa}{2} \tau \norm[1]{\delta_\tau 
					\boldsymbol{c}_h^n}_{1,h}^2+ \frac{C_\mathrm{coer}^2}{C_\mathrm{cont}^2} 
				\|\bm{\mathcal{J}}_h(\delta_\tau 
				c_h^n)\|_{1,h}^2  } 
			\\ \le ( \Phi(c_h^0), 1)_{\Omega} + \frac{C_\mathrm{cont} \kappa}{2} 
			\norm[0]{\boldsymbol{c}_h^0}_{1,h}^2.
		\end{multline}
		% \begin{equation} \label{eq:uniform_energy_bnd}
		% 	\max_{1 \le n \le N} \del{ \norm{\boldsymbol{c}_h^n}_{1,h}^2 +  \del{ 
		% 			\boldsymbol{\Phi}(c_h^n),1}_{\mathcal{E}_h} } + \tau \sum_{n=1}^N \del{
		% 	\norm{\boldsymbol{\mu}_h^n}_{1,h}^2 +
		% 	\norm[1]{\delta_\tau \boldsymbol{c}_h^n}_{1,h}^2 + \norm{\delta_\tau 
		% 	\boldsymbol{c}_h^n}_{-1,h}^2} \le C
		% \end{equation}
	\end{theorem}
	
	\begin{proof}
		Choosing $\boldsymbol{\chi}_h = \tau \boldsymbol{\mu}^n_h$ in \eqref{eq:CH_disc_a}, 
		$\boldsymbol{\phi}_h = \boldsymbol{c}_h^n - \boldsymbol{c}_h^{n-1}$ in 
		\eqref{eq:CH_disc_b}, and subtracting the two resulting equations, we find
		\begin{equation} \label{eq:energy_stab_1}
			\tau a_{\mathcal{D}}(\boldsymbol{\mu}_h^n, 
			\boldsymbol{\mu}_h^n) +
			(\Phi_+'(c_h^n) + \Phi_-'(c_h^{n-1}), 
			c_h^n - c_h^{n-1})_{\Omega} +	\kappa 
			a_{\mathcal{D}}(\boldsymbol{c}_h^n, \boldsymbol{c}_h^n - \boldsymbol{c}_h^{n-1}) = 0.
		\end{equation}
		With a Taylor expansion of $\Phi_{+}$ and $\Phi_{-}$, there exist $\xi_h$ and 
		$\zeta_h$ between $c_h^{n-1}$ and $c_h^n$ such that
		%		\footnote{Taylor expanding 
		%			$\Phi_+(c_h^{n-1})$ around $c_h^{n}$, we find there exists an $\xi_h$ between 
		%			$c_h^n$ and $c_h^{n-1}$ such that
		%			%
		%			\begin{equation}
		%				\Phi_+(c_h^{n-1}) = \Phi_+(c_h^n) + 
		%				\Phi_+'(c_h^n)\del[1]{c_h^{n-1} - c_h^n} + \frac{1}{2} 
		%				\Phi_+''(\xi_h) 
		%				\del[1]{c_h^n - c_h^{n-1}}^2.
		%			\end{equation}
		%			%
		%			Similarly, Taylor expanding $\Phi_{-}(c_h^n)$ around $c_h^{n-1}$, there exists a 
		%			$\zeta_h$ between $c_h^n$ and $c_h^{n-1}$ such that
		%			%
		%			\begin{equation}
		%				\Phi_-(c_h^{n}) = \Phi_-(c_h^{n-1}) + 
		%				\Phi_+'(c_h^{n-1})\del[1]{c_h^{n} - c_h^{n-1}} + \frac{1}{2} 
		%				\Phi_+''(\zeta_h) 
		%				\del[1]{c_h^n - c_h^{n-1}}^2.
		%			\end{equation}
		%		}
		%
		\begin{align}
			\Phi_{+}'(c_h^n) \del[1]{c_h^n - c_h^{n-1}} &= \Phi_{+}(c_h^n)  -  
			\Phi_{+}(c_h^{n-1}) + \frac{1}{2} \Phi_{+}''(\xi_h)(c_h^n - 
			c_h^{n-1})^2, \label{eq:taylor_exp_1}\\
			\Phi_{-}'(c_h^{n-1}) \del[1]{c_h^n - c_h^{n-1}} &= 
			\Phi_{-}(c_h^n)  -  \Phi_{-}(c_h^{n-1}) - \frac{1}{2} \Phi_{-}''(\zeta_h)(c_h^n - 
			c_h^{n-1})^2. \label{eq:taylor_exp_2}
		\end{align}
		Adding \eqref{eq:taylor_exp_1} and \eqref{eq:taylor_exp_2} and integrating over $\Omega$,
		\begin{equation} \label{eq:taylor_exp_3}
			\begin{split}
				(\Phi_+'(c_h^n) &+ \Phi_-'(c_h^{n-1}), 
				c_h^n - c_h^{n-1})_{\Omega} \\ &= \tau(\delta_\tau \Phi(c_h^n), 
				1)_{\Omega} + \frac{1}{2}( \Phi_+''(\xi_h),(c_h^n - 
				c_h^{n-1})^2)_{\Omega} - \frac{1}{2}( \Phi_-''(\zeta_h),(c_h^n - 
				c_h^{n-1})^2)_{\Omega}.
			\end{split}
		\end{equation}
		As $\Phi_+$ and $-\Phi_-$ are convex, the second and third terms on 
		the right-hand side of \eqref{eq:taylor_exp_3} are non-negative, and hence
		\begin{equation} \label{eq:taylor_exp_4}
			(\Phi_+'(c_h^n) + \Phi_-'(c_h^{n-1}), 
			c_h^n - c_h^{n-1})_{\Omega} \ge \tau(\delta_\tau \Phi(c_h^n), 
			1)_{\Omega}.
		\end{equation}
		Further, by the symmetry, bilinearity, and coercivity of $a_{\mathcal{D}}$, 
		\begin{equation} \label{eq:bilinear_backward_euler_formula}
			\begin{split}
				a_{\mathcal{D}}(\boldsymbol{c}_h^n, \boldsymbol{c}_h^n - \boldsymbol{c}_h^{n-1}) &= 
				\frac{1}{2}a_{\mathcal{D}}(\boldsymbol{c}_h^n, \boldsymbol{c}_h^n) + 
				\frac{1}{2} a_{\mathcal{D}}(\boldsymbol{c}_h^n - \boldsymbol{c}_h^{n-1}, 
				\boldsymbol{c}_h^{n} - \boldsymbol{c}_h^{n-1}) - \frac{1}{2} 
				a_{\mathcal{D}}(\boldsymbol{c}_h^{n-1}, \boldsymbol{c}_h^{n-1}) \\
				& \ge \frac{1}{2}a_{\mathcal{D}}(\boldsymbol{c}_h^n, \boldsymbol{c}_h^n) - \frac{1}{2} 
				a_{\mathcal{D}}(\boldsymbol{c}_h^{n-1}, \boldsymbol{c}_h^{n-1}) + 
				\frac{C_\mathrm{coer}}{2} 
				\norm[1]{\boldsymbol{c}_h^n - \boldsymbol{c}_h^{n-1}}_{1,h}^2.
			\end{split}
		\end{equation}
		Returning to \eqref{eq:energy_stab_1} and using \eqref{eq:taylor_exp_4}, 
		\eqref{eq:bilinear_backward_euler_formula}, and the coercivity of the 
		bilinear 
		form $a_{\mathcal{D}}$ \eqref{eq:aD_coercive}, we find
		\begin{equation}
			\tau(\delta_\tau \Phi(c_h^n), 
			1)_{\Omega} + \frac{\kappa}{2} \del{a_{\mathcal{D}}(\boldsymbol{c}_h^n, 
				\boldsymbol{c}_h^n) - 
				a_{\mathcal{D}}(\boldsymbol{c}_h^{n-1}, \boldsymbol{c}_h^{n-1})} + C_\mathrm{coer} \tau 
			\norm[0]{\boldsymbol{\mu}_h^n}_{1,h}^2  + \frac{ \kappa C_\mathrm{coer}}{2} 
			\norm[1]{\boldsymbol{c}_h^n - \boldsymbol{c}_h^{n-1}}_{1,h}^2  \le 0.
		\end{equation}
		Summing from $n=1$ to $n=m$ with $m\le N$, we have
		\begin{multline}
			( \Phi(c_h^m), 1)_{\Omega} + \frac{\kappa}{2} a_{\mathcal{D}}(\boldsymbol{c}_h^m, 
			\boldsymbol{c}_h^m) + C_\mathrm{coer} \sum_{n=1}^m \del{ \tau
				\norm[0]{\boldsymbol{\mu}_h^n}_{1,h}^2 + \frac{\kappa}{2} 
				\norm[1]{\boldsymbol{c}_h^n - \boldsymbol{c}_h^{n-1}}_{1,h}^2} \\ \leq 	(\Phi(c_h^0), 
			1)_{\Omega} + \frac{\kappa}{2} a_{\mathcal{D}}(\boldsymbol{c}_h^0, 
			\boldsymbol{c}_h^0).
		\end{multline}
		By the coercivity \eqref{eq:aD_coercive} and continuity \eqref{eq:aD_continuous} of the bilinear 
		form $a_{\mathcal{D}}$, 
		\begin{equation} \label{eq:stability_semi_final}
			( \Phi(c_h^m), 1)_{\Omega} 
			+ \frac{C_\mathrm{coer} \kappa}{2}\norm[0]{\boldsymbol{c}_h^m}_{1,h}^2 + C_\mathrm{coer} 
			\tau 
			\sum_{n=1}^m \del{
				\norm[0]{\boldsymbol{\mu}_h^n}_{1,h}^2 +  \frac{\kappa}{2} \tau \norm[1]{\delta_\tau 
					\boldsymbol{c}_h^n}_{1,h}^2 } 
			\le (\Phi(c_h^0), 1)_{\Omega} + \frac{C_\mathrm{cont} \kappa}{2} 
			\norm[0]{\boldsymbol{c}_h^0}_{1,h}^2,
		\end{equation}
		which is the desired bound, save for one term.
		%%, $	( \boldsymbol{\Phi}(c_h^m), 
		%%1)_{\mathcal{E}_h} + |\Omega| |\inf_c \boldsymbol{\Phi} | \ge 0$ for $0 \le n \le m$. Hence,
		%%%
		%%\begin{equation} \label{eq:energy_ineq}
		%%( \boldsymbol{\Phi}(c_h^m), 
		%%1)_{\mathcal{E}_h} + |\Omega| |\inf_c \boldsymbol{\Phi} |+ 
		%%\norm[0]{\boldsymbol{c}_h^m}_{1,h}^2 
		%%+ \tau \sum_{n=1}^m \del{
		%%\norm[0]{\boldsymbol{\mu}_h^n}_{1,h}^2 + \rami{\tau} \norm[1]{\delta_\tau 
		%%\boldsymbol{c}_h^n}_{1,h}^2} \le  
		%%C \del[2]{	( \boldsymbol{\Phi}(c_h^0), 
		%%1)_{\mathcal{E}_h} + \frac{\kappa}{2} \norm[0]{\boldsymbol{c}_h^0}_{1,h}^2 +  |\Omega| 
		%%|\inf_c 
		%%\boldsymbol{\Phi} |}.
		%%\end{equation}
		%
		%%\rami{can we ommit the above equation? We can just say that if $\Phi$ is bounded below by a 
		%%positive constant we have the bound on all the other term }
		
		To bound the remaining term, take $\bm{\chi}_h = \bm{\mathcal{J}}_h(\delta_\tau \bm{c}_h^n) $ 
		in 
		\eqref{eq:CH_disc_a}, and use the definition 
		of $\bm{\mathcal{J}}_h$, \eqref{eq:aD_coercive}, and \eqref{eq:aD_continuous}. We have, on the 
		one hand, 
		\begin{align}
			(\delta_\tau c_h^n, \bm{\mathcal{J}}_h (\delta_\tau c_h^n))_\Omega = - 
			a_\mathcal{D}(\bm{\mu}_h^n, 
			\bm{\mathcal{J}}_h (\delta_\tau c_h^n)) \leq C_\mathrm{cont} \|\bm{\mu}_h^n\|_{1,h} 
			\|\bm{\mathcal{J}}_h 
			(\delta_\tau c_h^n)\|_{1,h},
		\end{align}
		and on the other hand,
		\begin{align}
			C_\mathrm{coer} \|\bm{\mathcal{J}}_h (\delta_\tau c_h^n)\|_{1,h}^2 \leq
			a_\mathcal{D}(\bm{\mathcal{J}}_h(\delta_\tau c_h^n), \bm{\mathcal{J}}_h (\delta_\tau c_h^n))
			= (\delta_\tau c_h^n, \bm{\mathcal{J}}_h (\delta_\tau c_h^n))_\Omega.
		\end{align}
		Thus, we obtain 
		\begin{align}
			\frac{C_\mathrm{coer}^3}{C_\mathrm{cont}^2} \tau \sum_{n=1}^{m} 
			\|\bm{\mathcal{J}}_h(\delta_\tau c_h^n)\|_{1,h}^2 
			\leq C_\mathrm{coer} \tau \sum_{n=1}^m  \|\bm{\mu}_h^n\|_{1,h}^2.
			% \leq  ( \boldsymbol{\Phi}(c_h^0), 
			%				1)_{\mathcal{E}_h} + \frac{C_\mathrm{cont} \kappa}{2} 
			%\norm[0]{\boldsymbol{c}_h^0}_{1,h}^2. 
		\end{align}
		Thus, with \eqref{eq:stability_semi_final}, we obtain the remaining bound.
		This completes the proof. \qed
	\end{proof}
	\section{Analysis for the Ginzburg--Landau potential}

	For the remainder of the article, it will be assumed that the chemical energy density $\Phi$ is the 
	Ginzburg--Landau potential:
	\begin{equation} \label{eq:ginzburg_landau} 
		\Phi(c) = \frac{1}{4}(1 + c)^2(1-c)^2, \quad \Phi_+(c) = \frac{1}{4}(1+c^4), \quad \Phi_-(c) = - 
		\frac{1}{2}c^2.
	\end{equation}
	
	\subsection{Uniform a priori bounds}
	
	\begin{theorem} \label{thm:linf_stability} Assume that $\nabla c_0 \cdot \boldsymbol{n} = 0$ on 
		the boundary
		$\partial\Omega$. There exists a constant $C$ independent of $h$ and 
		$\tau$ but depending linearly on the final time $T$ such that the following holds. 
		\begin{equation} \label{eq:uniform_apriori_bnds}
			\max_{1 \le n \le N} \del{\norm{\mu_h^n}_{L^2(\Omega)}+  \norm{\boldsymbol{\Delta}_h 
					\boldsymbol{c}_h^n}_{0,h}
			} \le C. 
		\end{equation}
		In addition, if $\Omega$ is convex, there is a constant $C$ independent of $h$ and $\tau$ 
		such that
		\begin{equation} \label{eq:uniform_apriori_bnds2}
			\max_{1 \le n \le N}  \norm{c_h^n}_{L^{\infty}(\Omega)}  \le C.
		\end{equation}
		%\Rd We also have (do we need this later?)
		%\[
		%\kappa \sum_{n=1}^N \tau 
		%\norm{\delta_\tau c_h^n}_{L^2(\Omega)}^2 \le C
		%\]
		%\Bk
	\end{theorem}
	\begin{proof}
		To begin, we define $\boldsymbol{c}_h^{-1} = \boldsymbol{c}_h^0$ and $\boldsymbol{\mu}_h^0 
		\in \boldsymbol{S}_h$ as the solution to the following variational problem 
		\begin{equation} \label{eq:mu0_def}
			(\bm{\mu}_h^0, \boldsymbol{\phi}_h)_{0,h} = (\Phi'(c_h^0), \phi_h)_{\Omega} + \kappa 
			a_{\mathcal{D}}(\boldsymbol{c}_h^0,\boldsymbol{\phi}_h), \quad \forall \boldsymbol{\phi}_h \in 
			\boldsymbol{S}_h.
		\end{equation}
		Recall that $\boldsymbol{c}_h^0 \in \boldsymbol{S}_h$ is defined by \eqref{eq:init_proj_ch0}.
		
		%as the solution to the variational 
		%problem: given $c_0 \in H^2(\Omega)\cap L_0^2(\Omega)$, find $\boldsymbol{c}_h^0 \in 
		%\boldsymbol{S}_h$ such that
		%%
		%\begin{equation} \label{eq:def_c0}
		%a_{\mathcal{D}}( \boldsymbol{c}_h^0 - \boldsymbol{c}_0, \boldsymbol{v}_h) =0, \quad \forall 
		%\boldsymbol{v}_h \in 
		%\boldsymbol{S}_h,
		%\end{equation}
		%
		%with the constraint $(c_h^0 - c_0,1)_{\Omega} = 0$. 
		Therefore by testing \eqref{eq:init_proj_ch0} 
		with $\boldsymbol{v}_h = \boldsymbol{c}_h^0$, we 
		obtain from the coercivity and continuity of the bilinear form $a_{\mathcal{D}}$ that
		\begin{equation} \label{eq:c0_bnd}
			\norm[0]{\boldsymbol{c}_h^0}_{1,h} \lesssim \norm{ c_0}_{H^1(\Omega)}.
		\end{equation}
		Choosing $\boldsymbol{\phi}_h = \boldsymbol{\mu}_h^0$ in \eqref{eq:mu0_def} and using 
		\eqref{eq:init_proj_ch0}, we find 
		\begin{equation}
			\norm[0]{\mu_h^0}_{L^2(\Omega)}^2 \leq \Vert \boldsymbol{\mu}_h^0\Vert_{0,h}^2  = (\Phi'(c_h^0), \mu_h^0)_{\Omega} - \kappa 
			(\Delta c_0, \mu_h^0)_{\Omega},
		\end{equation}
		and thus by the Cauchy--Schwarz's inequality,
		\begin{equation} \label{eq:mu0_bnd}
			\norm[0]{\mu_h^0}_{L^2(\Omega)} \le \norm[0]{ \Phi'(c_h^0)}_{L^2(\Omega)} +  \kappa 
			\norm{c_0}_{H^2(\Omega)} \lesssim 1,
		\end{equation}
		where we have used that $c_0$ is a given quantity and that by the definition of the chemical energy
		density $\Phi$ \eqref{eq:ginzburg_landau} and the discrete Poincar\'{e} inequality 
		\Cref{lem:broken_poincare2},
		\begin{equation}
			\norm[0]{\Phi'(c_h^0)}_{L^2(\Omega)} = 	\norm[0]{ (c_h^0)^3 - c_h^0 }_{L^2(\Omega)} \le 
			\norm[0]{c_h^0}_{L^6(\Omega)}^3 + \norm[0]{c_h^0}_{L^2(\Omega)}.
			%\lesssim \norm[0]{\boldsymbol{c}_h^0}_{1,h}^3 + \norm[0]{\boldsymbol{c}_h^0}_{1,h},
		\end{equation}
		We note that since $\overline{c_0} = \overline{c_h^0}$:
		\[
		\norm[0]{c_h^0}_{L^2(\Omega)} \leq \Vert c_h^0 -\overline{c_h^0}\Vert_{L^2(\Omega)} 
		+ \Vert \overline{c_0}\Vert_{L^2(\Omega)} \leq C_P \Vert \boldsymbol{c}_h^0\Vert_{1,h} + \Vert 
		\overline{c_0}\Vert_{L^2(\Omega)}
		\leq C,
		\]
		thanks to \eqref{eq:c0_bnd}.
		Further, by \Cref{lem:broken_poincare2}, we have
		\[
		\norm[0]{c_h^0}_{L^6(\Omega)} \leq C_P ( |\Omega|^{-1/2}\Vert c_h^0\Vert_{L^2(\Omega)} 
		+ \Vert 
		\boldsymbol{c}_h^0\Vert_{1,h}) \leq C.
		\]
		Next, we show that
		\begin{equation}
			\max_{1 \le n \le N} \norm{\mu_h^n}_{L^2(\Omega)}^2 + \kappa \tau \sum_{n=1}^N 
			\norm{\delta_\tau c_h^n}_{L^2(\Omega)}^2 \le C.
		\end{equation}
		For $1 \le n \le N$, we subtract \eqref{eq:CH_disc_b} with $n$ replaced by $n-1$ from 
		\eqref{eq:CH_disc_b} and test with $\boldsymbol{\phi}_h = 
		\boldsymbol{\mu}_h^n$ to find 
		\begin{equation} \label{eq:apriori_inter_5}
			(\mu_h^n - \mu_h^{n-1}, \mu_h^n)_{\Omega} = \del{(c_h^n)^3 - 
				(c_h^{n-1})^3  - c_h^{n-1} + c_h^{n-2}, \mu_h^n}_{\Omega} 
			+ 
			\kappa 
			a_{\mathcal{D}}(\boldsymbol{c}_h^n - \boldsymbol{c}_h^{n-1}, \boldsymbol{\mu}_h^n). 
		\end{equation}
		Testing \eqref{eq:CH_disc_a} with $\boldsymbol{\chi}_h^n = \kappa \tau \delta_\tau 
		\boldsymbol{c}_h^n$, using the symmetry of $a_{\mathcal{D}}$, and combining with 
		\eqref{eq:apriori_inter_5}, we find
		\begin{equation} \label{eq:apriori_inter_6}
			\kappa \tau \norm{\delta_{\tau} c_h^n}_{L^2(\Omega)}^2 +	(\mu_h^n - \mu_h^{n-1}, 
			\mu_h^n)_{\Omega} = \del{(c_h^n)^3 - 
				(c_h^{n-1})^3  - c_h^{n-1} + c_h^{n-2}, \mu_h^n}_{\Omega}.
		\end{equation}
		Using the identity $2x(x-y) = x^2 + (x-y)^2 - y^2$, we find
		\begin{equation}
			(\mu_h^n - \mu_h^{n-1}, 
			\mu_h^n)_{\Omega} = \frac{1}{2}\norm{\mu_h^n}_{L^2(\Omega)}^2 +\frac{1}{2} 
			\norm[0]{\mu_h^n - \mu_h^{n-1}}_{L^2(\Omega)}^2 - 
			\frac{1}{2}\norm[0]{\mu_h^{n-1}}_{L^2(\Omega)}^2. \label{eq:identity_l2}
		\end{equation}
		Moreover, it holds that
		\begin{equation}
			(c_h^n)^3 - 
			(c_h^{n-1})^3  - c_h^{n-1} + c_h^{n-2} = \tau \delta_\tau c_h^n \del{ (c_h^n)^2 + 
				c_h^nc_h^{n-1} + 
				(c_h^{n-1})^2} - \tau \delta_\tau c_h^{n-1}. \label{eq:expanding_potential}
		\end{equation}
		Using \eqref{eq:alternate_green_op_bnd},  \eqref{eq:1h_norm_bounded_DG}, and Young's 
		inequality, we have that 
		\begin{multline}
			(\delta_\tau c_h^{n-1}, \mu_h^n)_{\Omega}  \leq C  \|\bm{\mathcal{J}}_h (\delta_\tau 
			c_h^{n-1})\|_{1,h}\| \mu_h^n  \|_{\DG} \leq C   \|\bm{\mathcal{J}}_h (\delta_\tau 
			c_h^{n-1})\|_{1,h}\|\bm{\mu}_h^n\|_{1,h} \\ \leq \norm[1]{\bm{\mathcal{J}}_h(\delta_\tau 
				c_h^{n-1})}_{1,h}^2 
			+ C \norm{\boldsymbol{\mu}_h^n}_{1,h}^2.
		\end{multline}
		Thus, by H\"{o}lder's, Cauchy--Schwarz's, and Young's inequalities,  we find
		\begin{align} 
				\kappa \tau \norm{\delta_{\tau} c_h^n}_{L^2(\Omega)}^2 +
				& \frac{1}{2}\norm{\mu_h^n}_{L^2(\Omega)}^2 +\frac{1}{2} 
				\norm[0]{\mu_h^n - \mu_h^{n-1}}_{L^2(\Omega)}^2 - 
				\frac{1}{2}\norm[0]{\mu_h^{n-1}}_{L^2(\Omega)}^2 \nonumber\\  
&\leq \tau  \norm[1]{ (c_h^n)^2 + 
					c_h^nc_h^{n-1} + 
					(c_h^{n-1})^2}_{L^3(\Omega)} \norm{\mu_h^n}_{L^6(\Omega)} \norm{\delta_\tau 
					c_h^n}_{L^2(\Omega)} \nonumber \\
& +\tau \norm[1]{\bm{\mathcal{J}}_h(\delta_\tau 
					c_h^{n-1})}_{1,h}^2 
				+ C \tau \norm{\boldsymbol{\mu}_h^n}_{1,h}^2.  \label{eq:apriori_inter_7}%\\
				%& \quad +  \frac{h^2\tau}{2} \norm[1]{\delta_\tau 
				%	\boldsymbol{c}_h^{n-1}}_{0,h}^2 + \frac{\tau}{2} 
				%	\norm{\boldsymbol{\mu}_h^n}_{1,h}^2.
		\end{align}
		By the triangle inequality,
		\[
		\norm[1]{ (c_h^n)^2 + c_h^nc_h^{n-1} + (c_h^{n-1})^2}_{L^3(\Omega)} \leq (1 + 2^{-1/3}) 
		(\Vert 
		c_h^n 
		\Vert_{L^6(\Omega)}^2
		+\Vert c_h^{n-1}\Vert_{L^6(\Omega)}^2).
		\]
		But for each $n\geq 1$ and for $2\leq p\leq p^\star$, we have from \eqref{eq:uniform_energy_bnd} 
		and the fact that $\overline{c_h^n} = 
		\overline{c_0}$,
		\[
		\Vert c_h^n \Vert_{L^p(\Omega)} \leq \Vert c_h^n -\overline{c_h^n}\Vert_{L^p(\Omega)} 
		+ \Vert \overline{c_0}\Vert_{L^p(\Omega)}
		\leq C_P \Vert \boldsymbol{c}_h^n \Vert_{1,h} + \Vert \overline{c_0}\Vert_{L^p(\Omega)}
		\leq C\left( \Vert \boldsymbol{c}_h^0\Vert_{1,h} + (\Phi(c_h^0),1)_\Omega^{1/2}\right) + \Vert 
		\overline{c_0}\Vert_{L^p(\Omega)}.
		\]
		It is easy to show (with for instance \Cref{lem:broken_poincare2} and \eqref{eq:c0_bnd}) that
		\[
		(\Phi(c_h^0),1)_\Omega^{1/2}\leq C.
		\]
			We conclude that for $2\leq p\leq p^\star$,
			\begin{equation}\label{eq:Lpboundch}
				\Vert c_h^n \Vert_{L^p(\Omega)}\leq C.
		\end{equation}
		This is true for all $n\geq 0$, thus we have
		\[
		\norm[1]{ (c_h^n)^2 + c_h^nc_h^{n-1} + (c_h^{n-1})^2}_{L^3(\Omega)} \leq C.
		\]
		With Young's inequality, \eqref{eq:apriori_inter_7} becomes
		\begin{multline} 
			\frac{\kappa \tau}{2}  \norm{\delta_{\tau} c_h^n}_{L^2(\Omega)}^2 +
			\frac{1}{2}\norm{\mu_h^n}_{L^2(\Omega)}^2 +\frac{1}{2} 
			\norm[0]{\mu_h^n - \mu_h^{n-1}}_{L^2(\Omega)}^2 - 
			\frac{1}{2}\norm[0]{\mu_h^{n-1}}_{L^2(\Omega)}^2 \\  \leq  
			\tau \frac{C}{\kappa} \norm{\mu_h^n}_{L^6(\Omega)}^2   + 
			\tau \norm[1]{\bm{\mathcal{J}}_h(\delta_\tau c_h^{n-1})}_{1,h}^2 
			+ C \tau \norm{\boldsymbol{\mu}_h^n}_{1,h}^2. % \\
			\label{eq:apriori_inter_8}
			%& \quad +  \frac{h^2\tau}{2} \norm[1]{\delta_\tau 
			%	\boldsymbol{c}_h^{n-1}}_{0,h}^2 + \frac{\tau}{2} 
			%	\norm{\boldsymbol{\mu}_h^n}_{1,h}^2.
		\end{multline}
		%
		%Using %the inverse inequality 
		%%%\cref{eq:negative_norm_inverse_ineq} and 
		%the triangle inequality, the fact that $\overline{c_h^n} = \overline{c_0}, \,\, \forall 1\leq n\leq N$, 
		Using \eqref{lem:broken_poincare2}, we find
		\begin{equation} \label{eq:apriori_inter_10}
			\begin{split}
				&\kappa \tau \norm{\delta_{\tau} c_h^n}_{L^2(\Omega)}^2 +
				\norm{\mu_h^n}_{L^2(\Omega)}^2  - 
				\norm[0]{\mu_h^{n-1}}_{L^2(\Omega)}^2 \\ & \leq C \tau  
				\norm{\mu_h^n}_{L^2(\Omega)}^2   
				+  C \tau 
				\norm{\boldsymbol{\mu}_h^n}_{1,h}^2 + 	2 \tau \norm[1]{\bm{\mathcal{J}}_h(\delta_\tau 
					c_h^{n-1})}_{1,h}^2. 
			\end{split}
		\end{equation}
		Summing \eqref{eq:apriori_inter_10} from $n=1$ to $n=m$ (recall that we have 
		defined $\boldsymbol{c}_h^{-1}$ such that
		$\delta_\tau \boldsymbol{c}_h^0 = 0$), we obtain
		\begin{equation} \label{eq:apriori_inter_11}
			\begin{split}
				\norm[0]{\mu_h^m}_{L^2(\Omega)}^2 + \kappa \tau \sum_{n=1}^m \norm{\delta_{\tau} 
					c_h^n}_{L^2(\Omega)}^2  \leq  & C \tau \sum_{n=0}^m 
				\del{\norm{\mu_h^n}_{L^2(\Omega)}^2  
					+ 
					\norm{\boldsymbol{\mu}_h^n}_{1,h}^2}  \\ &  + 2 \tau \sum_{n=1}^m 
				\norm[1]{\bm{\mathcal{J}}_h(\delta_\tau \boldsymbol{c}_h^{n-1})}_{1,h}^2 + 	
				\norm[0]{\mu_h^0}_{L^2(\Omega)}^2.
			\end{split}
		\end{equation}
		Using \eqref{eq:uniform_energy_bnd} and \eqref{eq:mu0_bnd}, this bound implies
		\begin{equation} \label{eq:apriori_inter_110}
			\norm[0]{\mu_h^m}_{L^2(\Omega)}^2 + \kappa \tau \sum_{n=1}^m \norm{\delta_{\tau} 
				c_h^n}_{L^2(\Omega)}^2  \leq   C \tau \sum_{n=0}^m \norm{\mu_h^n}_{L^2(\Omega)}^2  
			+ C.
		\end{equation}

		We can conclude by using a discrete Gr\"{o}nwall inequality, but this will yield a constant that 
		depends 
		exponentially 
		in time. 
		Instead we directly obtain a bound on $\Vert \mu_h^n\Vert_{L^2(\Omega)}$. 
		Testing \eqref{eq:CH_disc_b} 
		with $\boldsymbol{\phi}_h = \boldsymbol{\mu}_h^n$, we have
		\begin{equation}
			\norm{\mu_h^n}_{L^2(\Omega)}^2 =
			((c_h^n)^3 - c_h^{n-1}, \mu_h^n)_{\Omega} + 
			\kappa 
			a_{\mathcal{D}}(\boldsymbol{c}_h^n, \boldsymbol{\mu}_h^n),
		\end{equation}
		and thus using the Cauchy--Schwarz's inequality, the continuity of the bilinear form 
		$a_{\mathcal{D}}$, 
		Young's inequality, \eqref{eq:Lpboundch} and \eqref{eq:uniform_energy_bnd}, we have
		\begin{equation} \label{eq:mu_l2_sum_bnd_1}
			\norm{\mu_h^n}_{L^2(\Omega)}^2 \leq
			\norm{c_h^n}_{L^6(\Omega)}^6 + 
			\norm[0]{c_h^{n-1}}_{L^2(\Omega)}^2 
			+ \kappa \norm{\boldsymbol{c}_h^n}_{1,h}^2 +  \kappa 
			\norm{\boldsymbol{\mu}_h^n}_{1,h}^2  
			\leq C + \kappa \norm{\boldsymbol{\mu}_h^n}_{1,h}^2.
		\end{equation}
		and thus multiplying both sides of \eqref{eq:mu_l2_sum_bnd_1} by $\tau$, summing from $1\le n 
		\le N$, and using \eqref{eq:uniform_energy_bnd}, we find
		\begin{equation} \label{eq:mu_l2_sum_bnd_2}
			\tau \sum_{n=1}^N \norm{\mu_h^n}_{L^2(\Omega)}^2 \le C,
		\end{equation}
		with $C$ depending linearly on the final time $T$. Returning to \eqref{eq:apriori_inter_110}, we find 
		for 
		all $1 \le m \le N$,
		\begin{equation} \label{eq:mu_l2_bnd}
			\begin{split}
				\norm[0]{\mu_h^m}_{L^2(\Omega)}^2 + \kappa \tau \sum_{n=1}^m  \norm{\delta_{\tau} 
					c_h^n}_{L^2(\Omega)}^2  \le C,
			\end{split}
		\end{equation}
		with $C$ depending linearly on the final time $T$.
		Finally, we prove that
		\begin{equation}
			\max_{1 \le n \le N} \del{ \norm{c_h^n}_{L^{\infty}(\Omega)} + \norm{\boldsymbol{\Delta}_h 
					\boldsymbol{c}_h^n}_{0,h}^2 } \le C.
		\end{equation}
		For readibility, let $\boldsymbol{\chi}_h = \boldsymbol{\Delta}_h \boldsymbol{c}_h^n$.
		Testing the definition of the discrete Laplacian \eqref{eq:discrete_laplacian} with 
		$\boldsymbol{v}_h = \boldsymbol{\chi}_h  = (\chi_h, \hat{\chi}_h)$,
		and using \eqref{eq:CH_disc_b} and the definitions of $\Phi_+$ and 
		$\Phi_-$ 
		in 
		\eqref{eq:ginzburg_landau}, we find that
		%
		%
		%\begin{equation}
		%	\begin{split}
		%		\kappa \norm{\boldsymbol{\Delta}_h \boldsymbol{c}_h^n}_{0,h}^2  &= - \kappa 
		%		a_{\mathcal{D}}( 
		%		\boldsymbol{c}_h^n, \boldsymbol{\Delta}_h \boldsymbol{c}_h^n)  \\& = 	
		%		\del{ (c_h^n)^3 - c_h^{n-1}, \Delta_h 
		%			\boldsymbol{c}_h^n}_{\mathcal{E}_h} - 
		%		(\mu_h^n, \Delta_h \boldsymbol{c}_h^n)_{\mathcal{E}_h}.
		%	\end{split}
		%\end{equation}
		\begin{equation}
			\kappa \norm{\boldsymbol{\chi}_h}_{0,h}^2  = - \kappa 
			a_{\mathcal{D}}( 
			\boldsymbol{c}_h^n, \boldsymbol{\chi}_h)   = 	
			( (c_h^n)^3 - c_h^{n-1}, \chi_h )_\Omega
			- (\mu_h^n, \chi_h)_{\Omega}.
		\end{equation}
		The Cauchy--Schwarz's inequality and Young's inequality then yield
		\begin{equation} 
			\kappa \norm{\boldsymbol{\chi}_h}_{0,h}^2 \le  \frac{2}{\kappa} 
				\norm{c_h^n}_{L^6(\Omega)}^6 + \frac{2}{\kappa} 
				\norm[0]{c_h^{n-1}}_{L^2(\Omega)}^2  + 
			\frac{\kappa}{2} \norm{\boldsymbol{\chi}_h }_{0,h}^2 +\frac{1}{\kappa} 
			\norm{\mu_h^n}_{L^2(\Omega)}^2. 
		\end{equation}
		Applying \eqref{eq:Lpboundch} and \eqref{eq:mu_l2_bnd}, we obtain
		%
		%\begin{equation} \label{eq:disc_lapl_bnd_1}
		%\begin{split}
		%\frac{\kappa}{2} \norm{\boldsymbol{\chi}_h}_{0,h}^2 &\le \frac{1}{2} 
		%\norm{c_h^n}_{L^6(\Omega)}^6 + \frac{1}{2} 
		%\norm[0]{c_h^{n-1}}_{L^2(\Omega)}^2 +\frac{1}{\kappa} 
		%\norm{\mu_h}_{L^2(\Omega)}^2 \\
		%& \lesssim \frac{1}{2} 
		%\norm{\boldsymbol{c}_h^n}_{1,h}^6 + \frac{1}{2} 
		%\norm[0]{\boldsymbol{c}_h^{n-1}}_{1,h}^2 + 
		%\frac{1}{2} \norm{\overline{c}_0}_{L^6(\Omega)}^6 +	\frac{1}{2} 
		%\norm{\overline{c}_0}_{L^2(\Omega)}^2 +\frac{1}{2\kappa} 
		%\norm{\mu_h}_{L^2(\Omega)}^2
		%\end{split}
		%\end{equation}
		%%
		%and the right hand side of \eqref{eq:disc_lapl_bnd_1} is uniformly bounded for all $1 \le n \le N$ 
		%thanks to 
		%\eqref{eq:uniform_energy_bnd} and \cref{eq:mu_l2_bnd}; that is,
		%
		\begin{equation} \label{eq:disc_lapl_l2_bnd}
			\max_{1 \le n \le N} \norm{\boldsymbol{\Delta}_h \boldsymbol{c}_h^n}_{0,h} \le C.
		\end{equation}
		Finally, by the discrete Agmon inequality \Cref{cor:disc_agmon_Sh}, we have if the domain is 
		convex
		\begin{equation} \label{eq:linfty_bnd_1}
			\norm{c_h^n}_{L^{\infty}(\Omega)} \le \norm{c_h^n - \overline{c_0}}_{L^{\infty}(\Omega)}  + 
			\norm{\overline{c_0}}_{L^{\infty}(\Omega)} \lesssim \max_{1 \le n \le N} \del{ 
				\norm{\boldsymbol{c}_h^n}_{1,h}^{1/2} 
				\norm{\boldsymbol{\Delta}_h \boldsymbol{c}_h^n}_{0,h}^{1/2}} 
			+  \norm{ \overline{c_0}}_{L^{\infty}(\Omega)},
		\end{equation}
		and the right-hand side of \eqref{eq:linfty_bnd_1} is uniformly bounded thanks to 
		\eqref{eq:uniform_energy_bnd} and \eqref{eq:disc_lapl_l2_bnd}.
		\qed
	\end{proof}
	
	\subsection{Error analysis}
	The main goal of this section is to prove  the following convergence result. 
	\begin{theorem}[Error estimate]\label{theorem:error_estimate}
		Let  $2 \leq  s \leq k+1$.	 Assume that the  weak solution has the following regularity:
		\begin{align*}
			c,\mu \in L^{\infty}(0,T;H^s(\Omega)), \,\, \partial_t c \in  L^{2}(0,T;H^{s-1}(\Omega)), \,\, 
			\partial_{tt} c \in L^2(0;T,L^2(\Omega)).
		\end{align*}
		Further, assume the domain $\Omega \subset \mathbb{R}^d$, $d \in\cbr{2,3}$ is convex.
		There exists $\tau_0>0$ such that for any $\tau \leq \tau_0$, 
		the following error estimate holds.  For any $ 1 \leq m \leq N$, 
		\begin{equation}
			\|\bm{c}^n- \bm{c}_h^n \|^2_{1,h} + \tau \sum_{n=1}^{m} \|\bm{\mu}^n - \bm{\mu}_h^n 
			\|^2_{1,h} \lesssim \tau^2 + h^{2s-2}. 
		\end{equation}
	\end{theorem}
	The proof of this estimate requires several intermediate results; thus, it is presented in subsection 
	\ref{subsec:proof_err_estimate}. We begin the analysis by introducing useful operators and 
	interpolants. 
	\subsection{Interpolation and intermediate results}
	Recall the definitions of the $L^2$ projections \eqref{eq:l2_projections} and their approximation 
	properties~\Cref{prop:proj_estimates}.	For a given $ v \in H^s(\Omega)$,  set $\bm{v} = (v,v|_{\Gamma_h})$ and denote by $\bm{\pi}_h 
	v = ( \pi_h v, \hat{\pi}_h v)$, defined by \eqref{eq:l2_projections}.  From 
	~\Cref{prop:proj_estimates}, it  follows that 
	\begin{equation} \label{eq:approximation_property_L^2}
		\|\bm{v} - \bm{\pi}_h v\|_{1,h} \lesssim h^{s-1} |v|_{H^s(\Omega)}, \quad 2\leq s \leq k+1. 
	\end{equation}
	Further, we  consider the elliptic projection $\boldsymbol{\Pi}_h: H^2(\Omega)\rightarrow 
	\boldsymbol{S}_h$ defined as follows. Fix $w\in H^2(\Omega)$ and set $\boldsymbol{w} = (w, 
	w|_{\Gamma_h})$. The function $\boldsymbol{\Pi}_h w = (\Pi_h w, \hat{\Pi}_h w)$ is the unique 
	function in $\boldsymbol{S}_h$ satisfying:
	\begin{equation} \label{eq:broken_elliptic_projection}
		a_{\mathcal{D}}(\boldsymbol{\Pi}_h w,\boldsymbol{v}_h) = 
		a_{\mathcal{D}}(\boldsymbol{w},\boldsymbol{v}_h), \quad 
		\forall \boldsymbol{v}_h \in  \boldsymbol{S}_h,  \,\,\, \mathrm{and} \,\,\, 
		\int_\Omega \Pi_h w  =  \int_\Omega w.
	\end{equation}

	\begin{lemma}[Approximation properties of $\boldsymbol{\Pi}_h $]\label{lem:elliptic_proj_est} 
		The following approximation properties hold.  Fix $s$ such that $2 \le s \le k+1$ and let
		$w\in H^s(\Omega)$. Define $\boldsymbol{w} = (w,\hat{w}|_{\Gamma_h})$. 
		\begin{align}
			\|w - \Pi_h  w \|_{L^2(\Omega)} + 	h\norm{\boldsymbol{w} - \boldsymbol{\Pi}_h  w}_{1,h,\star} 
			&\lesssim h^{s} \norm{w}_{H^{s}(\Omega)}. 
			\label{eq:approximation_elliptic_projection}
		\end{align}
	\end{lemma}
	\begin{proof}
		Observe by the definition of the $L^2$-projection $\pi_h$ \eqref{eq:l2_projections}, it holds that
		\begin{equation}
			\int_{\Omega} \pi_h w \dif x = \int_{\Omega} w \dif x = \int_{\Omega} \Pi_h  w \dif x,
		\end{equation}
		so that $\pi_h w - \Pi_h  w \in M_h$. Consequently, by \eqref{eq:equivnorms},  the 
		coercivity and 
		continuity of 
		$a_{\mathcal{D}}(\cdot,\cdot)$ 
		\eqref{eq:aD_coercive} and \eqref{eq:aD_extended_continuous2}, we have
		\begin{align}
			\norm{ \boldsymbol{\pi}_h w - \boldsymbol{\Pi}_h  w}_{1,h,\star}^2 \lesssim	
			\norm{
				\boldsymbol{\pi}_h w - \boldsymbol{\Pi}_h  
				w}_{1,h}^2 &\lesssim a_{\mathcal{D}}(\boldsymbol{\pi}_h w - 
			\boldsymbol{\Pi}_h  
			w , \boldsymbol{\pi}_h w - \boldsymbol{\Pi}_h  
			w) \\&=  a_{\mathcal{D}}(\boldsymbol{\pi}_h w - 
			\boldsymbol{w} , \boldsymbol{\pi}_h w - \boldsymbol{\Pi}_h  
			w) \\& \lesssim \norm{\boldsymbol{\pi}_h w - 
				\boldsymbol{w}}_{1,h,\star}  \norm{\boldsymbol{\pi}_h w - 
				\boldsymbol{\Pi}_h  w}_{1,h,\star},
		\end{align}
		and therefore by the triangle inequality and the continuity of $a_{\mathcal{D}}(\cdot,\cdot)$ 
		\eqref{eq:aD_extended_continuous},
		\begin{equation}
			\norm{\boldsymbol{w} - \boldsymbol{\Pi}_h  w}_{1,h,\star} \lesssim \norm{\boldsymbol{w} - 
				\boldsymbol{\pi}_h 
				w}_{1,h,\star} + 
			\norm{\boldsymbol{\pi}_h w - \boldsymbol{\Pi}_h  w}_{1,h,\star}  \lesssim 
			\norm{\boldsymbol{\pi}_h w - 
				\boldsymbol{w}}_{1,h,\star}.
		\end{equation}
		The approximation properties of the projection $\boldsymbol{\pi}_h$ 
		\Cref{prop:proj_estimates} yield
		\begin{equation}\label{eq:elliptic_proj_est_1}
			\norm{\boldsymbol{w} - \boldsymbol{\Pi}_h  w}_{1,h,\star} \lesssim h^{s-1} 
			\norm{w}_{H^{s}(\Omega)}, 
			\quad \forall 2 \le s \le k+1.
		\end{equation}
		
		Since the domain $\Omega$ is convex, we can 
		prove 
		optimal $L^2$-estimates for $\Pi_h $ using a standard duality argument. To this end, consider 
		the 
		auxiliary Neumann problem: find $z \in H^2(\Omega) \cap L_0^2(\Omega)$ such that
		\begin{align}
			-\Delta z &= w - \Pi_h  w, &&  \hspace{-40mm} \text{ in }  \Omega, \\
			\nabla z \cdot n &= 0, &&  \hspace{-40mm} \text{ on } \partial \Omega.
		\end{align}
		We find that, on the one hand,
		\begin{equation}
			\norm{w - \Pi_h  w}_{L^2(\Omega)}^2 = -\int_{\Omega} \Delta z (w - \Pi_h  w) \dif x, 
		\end{equation}
		while on the other hand, the single-valuedness of $z$ across element interfaces yields, 
		\begin{equation}
			a_{\mathcal{D}}( \boldsymbol{w} - \boldsymbol{\Pi}_h  w, \boldsymbol{z}) =  \sum_{E \in 
				\mathcal{E}_h} \int_{E} \nabla 
			(w - \Pi_h  w)
			\cdot 
			\nabla 
			z
			\dif x
			- \sum_{E \in \mathcal{E}_h} \int_{\partial E}  
			\del[1]{ ( w -  \Pi_h  w) - (w- \hat{\Pi}_h  w)} \nabla z \cdot \mathbf{n}_E \dif s,
		\end{equation}
		where $\boldsymbol{z} = (z, z|_{\Gamma_h})$. Using the regularity of $w$ and $z$ and the 
		fact that $\nabla z \cdot n = 0$ on 
		$\partial 
		\Omega$, we find
		\begin{equation}
			\sum_{E \in \mathcal{E}_h} \int_{\partial E}  
			\del[1]{ w- \hat{\Pi}_h  w} \nabla z \cdot \mathbf{n}_E \dif s = 0,
		\end{equation}
		while an element-by-element integration by parts shows that
		\begin{equation}
			\sum_{E \in 
				\mathcal{E}_h} \int_{E} \nabla 
			(w - \Pi_h  w)
			\cdot 
			\nabla 
			z
			\dif x
			- \sum_{E \in \mathcal{E}_h} \int_{\partial E}  
			\del[1]{ ( w -  \Pi_h  w) } \nabla z \cdot \mathbf{n}_E \dif s = -\int_{\Omega} \Delta z (w - \Pi_h  w) \dif 
			x.
		\end{equation}
		Therefore, 
		\begin{equation}
			\norm{w - \Pi_h  w}_{L^2(\Omega)}^2 = a_{\mathcal{D}}( \boldsymbol{w} - 
			\boldsymbol{\Pi}_h  w, 
			\boldsymbol{z}).
		\end{equation}
		Now, observe that since $z \in L_0^2(\Omega)$, $\boldsymbol{\pi}_h z \in \boldsymbol{M}_h$. 
		Consequently, by the definition of the elliptic projection \eqref{eq:broken_elliptic_projection},
		\begin{equation}
			a_{\mathcal{D}}( \boldsymbol{w} - \boldsymbol{\Pi}_h  w, \boldsymbol{\pi}_h z) = 0,
		\end{equation}
		and therefore, by the continuity of $a_{\mathcal{D}}(\cdot,\cdot)$ 
		\eqref{eq:aD_extended_continuous2}, we have
		\begin{equation}
			\norm{w - \Pi_h  w}_{L^2(\Omega)}^2 = a_{\mathcal{D}}( \boldsymbol{w} - 
			\boldsymbol{\Pi}_h  w, 
			\boldsymbol{z} - \boldsymbol{\pi}_h z) \lesssim \norm{\boldsymbol{w} - \boldsymbol{\Pi}_h  
				w}_{1,h,\star} \norm{\boldsymbol{z} - \boldsymbol{\pi}_h z}_{1,h,\star},
		\end{equation}
		and the approximation properties of the projection $\boldsymbol{\pi}_h$ 
		\Cref{prop:proj_estimates} 
		yields
		\begin{equation}
			\norm{w - \Pi_h  w}_{L^2(\Omega)}^2  \lesssim h^2 \norm{z}_{H^2(\Omega)} 
			\norm{\boldsymbol{w} - \boldsymbol{\Pi}_h  
				w}_{1,h,\star}.
		\end{equation}
		The result now follows from a standard elliptic regularity argument and 
		\eqref{eq:elliptic_proj_est_1}. \qed
		%		Since the form $a_\mathcal{D}$ is coercive and continuous, standard arguments easily 
		%show 
		%that  
		%		\[ \|\bm{w} - \boldsymbol{\Pi} w\|_{1,h} \lesssim \inf_{\bm{v}_h \in \bm{S}_h} \|\bm{w} - 
		%\bm{v}_h 
		%		\|_{1,h} 
		%		\lesssim h^{s-2} \norm{w}_{H^{s+1}(\Omega)}, \forall \, 2 \le s \le k+1. \]
		%		In the above, we used \eqref{eq:approximation_property_L^2}. The $L^2$ estimate follows 
		%from 
		%a 
		%		standard duality argument since $a_\mathcal{D}$ is symmetric and the domain is convex. 
		%		\rami{maybe add a reference ?} \qed
	\end{proof}
	For notational brevity, we denote by $\bm{v}^n = \bm{v}(t_n)$ for a function $\bm{v} \in 
	L^1(0,T;H^2(\mesh) \times L^2(\Gamma_h))$ and $1 \le n \le N$. We define the errors
	\begin{equation}
		\boldsymbol{e}_{c}^n = (\boldsymbol{\Pi}_h  \boldsymbol{c})^n - \boldsymbol{c}_h^n = 
		(e_c^n,\hat{e}_c^n), \quad 	
		\boldsymbol{e}_{\mu}^n = (\boldsymbol{\pi}_h \mu)^n - \boldsymbol{\mu}_h^n = (e_\mu^n, 
		\hat{e}_\mu^n).
	\end{equation}
	To derive the error equations required for the a priori analysis, observe that the following consistency property holds. 
	\begin{lemma}[Consistency] \label{prop:consistency}
 	Let $(c,\mu)$ be the exact solution to \eqref{eq:CH_strong}. Define $\boldsymbol{c} = 
		(c,c|_{\Gamma_h})$
		and $\boldsymbol{\mu} = (\mu,\mu|_{\Gamma_h})$.	Under the same regularity assumptions as Theorem \ref{theorem:error_estimate}, we have
		\begin{subequations} \label{eq:CH_exact}
			\begin{align}
				((\partial_t c)^n, \chi_h)_{\Omega} + a_{\mathcal{D}}(\boldsymbol{\mu}^n, 
				\boldsymbol{\chi}_h) &= 0, && \forall \boldsymbol{\chi}_h \in \boldsymbol{S}_h,
				\label{eq:CH_exact_a}\\
				(\Phi'(c^n), \phi_h)_{\Omega} + 
				\kappa a_{\mathcal{D}}(\boldsymbol{c}^n, \boldsymbol{\phi}_h) &= (\mu^n, 
				\phi_h)_{\Omega}, && 
				\forall \boldsymbol{\phi}_h \in \boldsymbol{S}_h. \label{eq:CH_exact_b}
			\end{align}
		\end{subequations}
	\end{lemma}
	By the consistency 
	property \eqref{prop:consistency} and the definition of the numerical scheme \eqref{eq:CH_disc}, for 
	any 
		$ \boldsymbol{\chi}_h, \bm{\phi}_h  \in 
		\boldsymbol{S}_h$, we obtain the following 
	\begin{subequations}
		\begin{alignat}{2}
			((\partial_t c)^n - \delta_\tau  c_h^n , \chi_h)_{\Omega}& + a_{\mathcal{D}} 
			(\boldsymbol{\mu}^n - 
			\boldsymbol{\mu}_h^n, \boldsymbol{\chi}_h)  = 0, \\
			(\mu^n - \mu_h^n, \phi_h)_{\Omega} &= ((c^n)^3- 
			(c_h^n)^3,  \phi_h)_{\Omega} - (c^n-c_h^{n-1},  \phi_h)_{\Omega}  
			+ \kappa a_{\mathcal{D}}(\boldsymbol{c}^n - 
			\boldsymbol{c}_h^n, \bm{\phi}_h).
		\end{alignat}
	\end{subequations}
	By the definition of the elliptic projection \eqref{eq:broken_elliptic_projection} and the 
	orthogonal $L^2$-projection \eqref{eq:l2_projections}, we equivalently 
	have for any $ \boldsymbol{\chi}_h, \bm{\phi}_h  \in 
		\boldsymbol{S}_h$ 
	\begin{subequations}
		\begin{align}
			(\delta_\tau e_{c}^n , \chi_h)_{\Omega} + a_{\mathcal{D}} \label{eq:error_eq_a}
			(\boldsymbol{e}_{\mu}^n, \boldsymbol{\chi}_h) &= (\delta_\tau (\Pi_h c)^n - (\partial_t c)^n  , 
			\chi_h)_{\Omega} + a_{\mathcal{D}} 
			(\boldsymbol{\pi}_h \mu^n - \boldsymbol{\mu}^n, \boldsymbol{\chi}_h), \\ \label{eq:error_eq_b}
			(e_{\mu}^n, \phi_h)_{\Omega}& = ((c^n)^3- (c_h^n)^3,  \phi_h)_{\Omega} 
			- (c^n- c_h^{n-1},  \phi_h)_{\Omega} + \kappa a_{\mathcal{D}}(\boldsymbol{e}_{c}^n, 
			\boldsymbol{\phi}_h).
		\end{align}
	\end{subequations}
	In the proof of Theorem~\ref{theorem:error_estimate}, a major difficulty is in handling the first term 
	in the right-hand side of \eqref{eq:error_eq_b}, see term  $T_4$ in 
	\eqref{eq:termstobound}. The following result is critical in controlling that term. 
	
	\begin{lemma} \label{prop:Lipschitz_condition}
		Fix $1\leq n\leq N$ and set for brevity $Q = (c^n)^3 - (c_h^n)^3$. 
		Define $\boldsymbol{z}_h =  (\pi_h Q, \{ \pi_h Q\}|_{\Gamma_h})  \in \boldsymbol{S}_h$.
		Assume that $c^n\in H^s(\Omega)$ for $2 \leq s \leq k+1$.  We have
		\begin{equation}
			\norm{\boldsymbol{z}_h}_{1,h}^2 \lesssim h^{2s-2} 
			\norm{c^n}_{H^{s}(\Omega)}^2 + 
			\|\bm{e}^n_{c}\|_{1,h}^2.
		\end{equation}
	\end{lemma}
	\begin{proof}
		From the definition of the $\norm{\cdot}_{1,h}$-norm \eqref{eq:product_space_norm_h1} 
		and by the $H^1$-stability of the $L^2$-projection \eqref{eq:w1p_stab_proj}, we can write
		\begin{align}
			\norm{\boldsymbol{z}_h}_{1,h}^2 %&=  \norm{\nabla_h z_h}_{L^2(\Omega)} + \sum_{E \in 
			%\mathcal{E}_h} \frac{1}{h_E} \norm{z_h - \hat{z}_h}_{L^2(\partial E)}^2\\
			&= \norm[0]{\nabla_h \pi_h Q }_{L^2(\Omega)}^2 
			+ \sum_{E \in \mathcal{E}_h} \sum_{e \in \mathcal{F}_E} \frac{1}{h_E} 
			\norm{\pi_h Q \vert_{E} - \{\pi_h Q\} }_{L^2(e)}^2 \nonumber\\
			&\lesssim  \Vert \nabla_h Q\Vert_{L^2(\Omega)}^2  +
			\sum_{E \in \mathcal{E}_h} \sum_{e \in \mathcal{F}_E} \frac{1}{h_E} 
			\norm{\pi_h Q \vert_{E} - \{\pi_h Q\} }_{L^2(e)}^2. 
			\label{eq:T1T2}
		\end{align}		
		To bound the first term in the right-hand side,  observe that
		\begin{equation}
			Q = \del{c_h^n - c^n }\del{ (c^n)^2 + c^n c_h^n + (c_h^n)^2},
		\end{equation}
		so that on each element,
		\begin{equation}
			\nabla	Q %&= \del{ (c^n)^2 + c^n c_h^n + (c_h^n)^2} \nabla_h \del{c^n - 
			%c_h^n} \, +  \nabla_h\del{ (c^n)^2 + c^n c_h^n + (c_h^n)^2}\del{c^n - 
			%c_h^n} \\
			=  \del{ (c^n)^2 + c^n c_h^n + (c_h^n)^2} \nabla \del{c_h^n - c^n} 
			+ \del{ 2 c^n \nabla c^n+  (\nabla c^n) c_h^n + c^n (\nabla c_h^n) + 
				2 c_h^n\nabla c_h^n}\del{c_h^n - c^n}.
		\end{equation}
		Thus, with 
		%\begin{equation} \label{eq:Qn_elem_bnd}
		%\begin{split}
		%\norm{\nabla_h Q}_{L^2(\Omega)} &\lesssim \del{ \norm{c^n}_{L^\infty(\Omega)}^2 +  
		%\norm{c_h^n}_{L^\infty(\Omega)}^2} \norm[0]{\nabla_h \del{c^n - 
		%c_h^n}}_{L^2(\Omega)}
		%+ \norm[0]{2c^n \nabla c^n \del{c^n - 
		%c_h^n}}_{L^2(\Omega)} \\ & \quad \norm[0]{(\nabla c^n) c_h^n \del{c^n - 
		%c_h^n}}_{L^2(\Omega)}  + \norm[0]{ c^n (\nabla_h c_h^n) \del{c^n - 
		%c_h^n}}_{L^2(\Omega)} + \norm[0]{ 2c_h^n \nabla_h 
		%c_h^n \del{c^n - 
		%c_h^n}}_{L^2(\Omega)}.
		%\end{split}
		%\end{equation}
		%
		H\"{o}lder's inequality, we have
		%
		%\begin{align} \label{eq:Qn_elem_bnd_1}
		%\norm[0]{2c^n \nabla c^n \del{c^n - 
		%c_h^n}}_{L^2(\Omega)^d}  & \lesssim \norm{ c^n}_{L^{\infty}(\Omega)} \norm{\nabla 
		%c^n}_{L^3(\Omega)^d} \norm{c^n - c_h^n}_{L^6(\Omega)}, \\ 
		%\label{eq:Qn_elem_bnd_2}
		%\norm[0]{(\nabla c^n) c_h^n \del{c^n - 
		%c_h^n}}_{L^2(\Omega)^d}	& \lesssim \norm[0]{c_h^n}_{L^{\infty}(\Omega)}\norm{\nabla 
		%c^n}_{L^3(\Omega)^d} \norm{c^n - c_h^n}_{L^6(\Omega)}, \\ 
		%\label{eq:Qn_elem_bnd_3}
		%\norm[0]{ c^n (\nabla_h c_h^n) \del{c^n - 
		%c_h^n}}_{L^2(\Omega)^d} & \lesssim  \norm{ c^n}_{L^{\infty}(\Omega)} \norm{\nabla_h 
		%c_h^n}_{L^3(\Omega)^d} \norm{c^n - c_h^n}_{L^6(\Omega)}, \\ 
		%\label{eq:Qn_elem_bnd_4}
		%\norm[0]{ 2c_h^n \nabla_h 
		%c_h^n \del{c^n - 
		%c_h^n}}_{L^2(\Omega)^d} & \lesssim  \norm{ c_h^n}_{L^{\infty}(\Omega)} \norm{\nabla_h 
		%c_h^n}_{L^3(\Omega)^d} \norm{c^n - c_h^n}_{L^6(\Omega)}.
		%\end{align}
		%
		%Collecting \eqref{eq:Qn_elem_bnd}{eq:Qn_elem_bnd_4}, we have
		%
		\begin{equation}
			\begin{split}
				\norm{\nabla_h Q}_{L^2(\Omega)} &\lesssim \del{ \norm{c^n}_{L^\infty(\Omega)}^2 +  
					\norm{c_h^n}_{L^\infty(\Omega)}^2} \norm[0]{\nabla_h \del{c^n - 
						c_h^n}}_{L^2(\Omega)}  \\ 
				\quad &+ 
				\del{  \norm{ c^n}_{L^{\infty}(\Omega)}  +  \norm{ c_h^n}_{L^{\infty}(\Omega)}} \del{ 
					\norm{\nabla_h 
						c_h^n}_{L^3(\Omega)}  + \norm{\nabla 
						c^n}_{L^3(\Omega)} } \norm{c^n-c_h^n}_{L^6(\Omega)}.
			\end{split}
		\end{equation}
		Since $c^n\in H^2(\Omega)$, we have by Sobolev's embedding:
		\[
		\Vert \nabla c^n\Vert_{L^3(\Omega)} \lesssim \Vert c^n \Vert_{H^2(\Omega)} \lesssim 1.
		\]
		By the discrete Gagliardo--Nirenberg inequality \eqref{eq:disc_gagl_niren2}, 
		\eqref{eq:uniform_energy_bnd}   and \eqref{eq:uniform_apriori_bnds}, we have
		\begin{align}
			%\norm{\nabla c^n}_{L^3(\Omega)^d} & \lesssim |c^n|_{H^1(\Omega)}^{(6 + d)/12} 
			%\norm{c^n}_{H^2(\Omega)}^{(6-d)/12} \lesssim 1, 
			%\\
			\norm{\nabla_h c_h^n}_{L^3(\Omega)} & \lesssim 
			\norm{\boldsymbol{c}_h^n}_{1,h}^{(6+d)/12} 
			\norm{\boldsymbol{\Delta}_h \boldsymbol{c}_h^n}_{0,h}^{(6-d)/12} \lesssim 1.
		\end{align}
		Therefore, the $H^2$ regularity of $c^n$ and 
		\eqref{eq:uniform_apriori_bnds2} yield
		\begin{equation}
			\norm{\nabla_h Q}_{L^2(\Omega)}  \lesssim \norm[0]{\nabla_h \del{c^n - 
					c_h^n}}_{L^2(\Omega)} +   \norm{c^n-c_h^n}_{L^6(\Omega)}.
		\end{equation}
		Using \cref{eq:poincare}, we write
		\begin{equation}
			\norm{\nabla_h Q}_{L^2(\Omega)}  \lesssim \norm[0]{\nabla_h \del{c^n - 
					c_h^n}}_{L^2(\Omega)} 
			+ \Vert \boldsymbol{c}^n-\boldsymbol{c}_h^n\Vert_{1,h}.
		\end{equation}
		With the triangle inequality, we have
		\begin{equation}
			\norm{\nabla_h Q}_{L^2(\Omega)}  \lesssim \Vert \boldsymbol{e}_c^n \Vert_{1,h}
			+ \Vert \boldsymbol{c}^n - (\Pi_h \boldsymbol{c})^n\Vert_{1,h}.
		\end{equation}
		With  the approximation bound \eqref{eq:approximation_elliptic_projection}, we finally obtain 
		\begin{equation}
			\norm{\nabla_h Q}_{L^2(\Omega)}
			\lesssim \Vert \boldsymbol{e}_c^n\Vert_{1,h}
			+ h^{s-1} \Vert c^n \Vert_{H^s(\Omega)}.
		\end{equation}
		
		%%Using the definition of the norm $\norm[0]{\cdot}_{1,h}$ and the approximation properties of 
		%%the 
		%%projections $\boldsymbol{\pi}_h$ and 
		%%$\boldsymbol{\Pi}_h$ {\KK \Cref{prop:proj_estimates,lem:elliptic_proj_est}}, we find
		%%\begin{equation}
		%%\begin{split}
		%%\norm{\nabla_h Q^n}_{L^2(\Omega)^d} 
		%%& \lesssim h^{s-1} |c^n|_{H^{s}(\Omega)} + h^{s-1} 
		%%|c^n|_{W^{s-1,6}(\Omega)} + 
		%%\norm{\boldsymbol{\Pi}_h c^n- \boldsymbol{c}_h^n}_{1,h}.
		%%\end{split}
		%%\end{equation}
		%
		%%By the Sobolev embedding theorem (see e.g. \cite[Theorem 4.12]{adams2003sobolev}) 
		%%$H^{s}(\Omega) 
		%%\subset W^{s-1,6}(\Omega)$ with continuous embedding for $d \in \cbr{2,3}$ and all $k \ge 
		%%0$. 
		%%Therefore,
		%%%
		%%\begin{equation}\label{eq:lipschitz_element_bnd}
		%%\begin{split}
		%%\norm{\nabla_h Q^n}_{L^2(\Omega)^d} 
		%%& \lesssim h^{2s-2} \norm{c^n}_{H^{s}(\Omega)} +
		%%\norm{\boldsymbol{\Pi}_h c^n- \boldsymbol{c}_h^n}_{1,h}.
		%%\end{split}
		%%\end{equation}
		%
		We note that the terms on boundary faces in the right-hand side of \eqref{eq:T1T2} vanish. It 
		remains to bound
		the terms corresponding to interior faces. 	 
		%To bound the second term in the right-hand side of \eqref{eq:T1T2}, %we first note that for $e 
		%%\in \Gamma_h \setminus \Gamma_h^0$, $\del{z_h - 
		%\hat{z}_h}|_{e} = 0$. \Rd I do not think the previous equality is true\Bk.
		For $e \in \mathcal{F}_E \cap \Gamma_h^0$, we have by the triangle 
		inequality
		\begin{equation} \label{eq:lipschitz_face_bnd_1}
			\begin{split}
				& \norm{\pi_h Q \vert_{E} -  \av{ \pi_h Q} }_{L^2(e)} \\
				& = \frac{1}{2}\norm{\jump{\pi_h Q}}_{L^2(e)} \\
				&= \frac{1}{2}\norm{(\pi_h Q)|_{E_+} - (\pi_h Q)|_{E_-} }_{L^2(e)} \\
				& \le \frac{1}{2}\norm{Q|_{E_+} - (\pi_h Q)|_{E_+} }_{L^2(e)} + 
				\frac{1}{2}\norm{Q|_{E_+}  - 
					(\pi_h 
					Q)|_{E_-} }_{L^2(e)} \\
				& \le \frac{1}{2}\norm{Q|_{E_+} - (\pi_h Q)|_{E_+} }_{L^2(e)} + 
				\frac{1}{2}\norm{Q|_{E_+}  - 
					Q|_{E_-}  
				}_{L^2(e)} + \frac{1}{2}\norm{Q|_{E_-}  - (\pi_h 
					Q)|_{E_-} }_{L^2(e)}  \\
				& \lesssim h_{E_+}^{1/2} \norm{\nabla Q}_{L^2(E_+)} + h_{E_-}^{1/2} \norm{\nabla 
					Q}_{L^2(E_-)} + \norm{\jump{Q}}_{L^2(e)}.
			\end{split}
		\end{equation}
		In the last inequality, we used \eqref{eq:l2_proj_elem_face_est}. As $\jump{(c^n)^3} = 0$, it 
		suffices to estimate $\jump{(c_h^n)^3}$ to bound the third term on 
		the right-hand side of \eqref{eq:lipschitz_face_bnd_1}. To this end, note that
		\begin{equation}
			\begin{split}
				\jump{(c_h^n)^3} & = (c_h^n)^3|_{E_+} - (c_h^n)^3|_{E_-} \\
				&= ((c_h^n)|_{E_+} - (c_h^n)|_{E_-})   \left( (c_h^n)^2|_{E_+} + 
				(c_h^n)|_{E_+}(c_h^n)|_{E_-} + (c_h^n)^2|_{E_-}\right),
			\end{split}
		\end{equation}
		and therefore,
		\begin{equation}
			\norm[1]{\jump{(c_h^n)^3}}_{L^2(e)} \lesssim \norm{c_h^n}_{L^\infty(\Omega)}^2 
			\norm{\jump{c_h^n}}_{L^2(e)} \lesssim 	\norm{\jump{c_h^n}}_{L^2(e)},
		\end{equation}
		where we have used \eqref{eq:uniform_apriori_bnds2}. 
		Thus, we have
		\[
		\sum_{E \in \mathcal{E}_h} \sum_{e \in \mathcal{F}_E} \frac{1}{h_E} \norm{\pi_h Q \vert_{E} - 
			\{\pi_h Q\} }_{L^2(e)}^2
		\lesssim \Vert \nabla_h Q\Vert_{L^2(\Omega)}^2
		+ \sum_{e\in\Gamma_h^0} \frac{1}{h_e} \Vert [c_h^n]\Vert_{L^2(e)}^2.
		\]
		We have by \eqref{eq:1h_norm_bounded_DG} and the approximation properties 
		\eqref{eq:approximation_elliptic_projection}
		\begin{align*}
			\sum_{e\in\Gamma_h^0} \frac{1}{h_e} \Vert [c_h^n]\Vert_{L^2(e)}^2  &=
			\sum_{e\in\Gamma_h^0} \frac{1}{h_e} \Vert [c_h^n-c^n]\Vert_{L^2(e)}^2\\
			&\lesssim \Vert \boldsymbol{c}_h^n - \boldsymbol{c}^n\Vert_{1,h}^2\\
			&\lesssim \Vert \boldsymbol{e}_c^n \Vert_{1,h}^2 + \Vert \boldsymbol{\Pi}_h c^n - 
			\boldsymbol{c}^n\Vert_{1,h}^2\\
			&\lesssim \Vert \boldsymbol{e}_c^n \Vert_{1,h}^2 + h^{2s-2} \Vert c^n\Vert_{H^s(\Omega)}^2. 
		\end{align*}
		We then conclude by combining the bounds above. \qed
	\end{proof}
	\Bk
	
	\subsection{Proof of Theorem \ref{theorem:error_estimate}} \label{subsec:proof_err_estimate}
	\begin{proof}
		From \Cref{prop:mass_conservation} and  \eqref{eq:broken_elliptic_projection}, we see that 
		$\delta_\tau e_{c}^n \in M_h$. 
		Testing \eqref{eq:error_eq_a} with  $\boldsymbol{\chi}_h^n = 
		\boldsymbol{\mathcal{J}}_h(\delta_\tau 
		e_{c}^n) \in \boldsymbol{M}_h$, \eqref{eq:error_eq_b} with 
		$\boldsymbol{\phi}_h^n = \delta_\tau \boldsymbol{e}_{c}^n \in \boldsymbol{M}_h$ and 
		subtracting, we find 
		\begin{equation} \label{eq:bnding_ech_1}
			\begin{split}
				(\delta_\tau e_{c}^n, \mathcal{J}_h(\delta_\tau 
				e_{c}^n))_{\Omega} &+ 
				a_{\mathcal{D}} 
				(\boldsymbol{e}_{\mu}^n,  \boldsymbol{\mathcal{J}}_h(\delta_\tau 
				e_{c}^n)) +  \kappa a_{\mathcal{D}}
				(\boldsymbol{e}_{c}^n, 
				\delta_\tau \boldsymbol{e}_{c}^n ) \\ &= (\delta_\tau (\Pi_h c^n) - (\partial_t c)^n  , 
				\mathcal{J}_h(\delta_\tau e_{c}^n))_{\Omega} + a_{\mathcal{D}} 
				(\boldsymbol{\pi}_h \mu^n - \boldsymbol{\mu}^n, 	\boldsymbol{\mathcal{J}}_h(\delta_\tau 
				e_{c}^n)) \\
				& \quad - ((c^n)^3- (c_h^n)^3,  \delta_\tau e_{c}^n)_{\Omega} 
				+ (c^n- c_h^{n-1},  \delta_\tau e_{c}^n)_{\Omega} 
				+  (e_{\mu}^n, \delta_\tau e_{c}^n)_{\Omega}. 
			\end{split}
		\end{equation}
		Then,  it follows from \Cref{lem:alternate_green_op_properties} 
		and the symmetry of the bilinear 
		form 
		$a_{\mathcal{D}}$ that
		\begin{equation}
			a_{\mathcal{D}} 
			(\boldsymbol{e}_{\mu}^n,  \boldsymbol{\mathcal{J}}_h(\delta_\tau e_{c}^n)) %= 
			%a_{\mathcal{D}} 
			%(\boldsymbol{\mathcal{J}}_h(\delta_\tau 
			%			e_{c,h}^n),\boldsymbol{e}_{\mu}^n) 
			= (e_{\mu}^n, \delta_\tau e_{c}^n)_{\Omega}.
		\end{equation}
		Moreover, by the definition of the operator $\boldsymbol{\mathcal{J}}_h$, we have
		\begin{equation}
			a_{\mathcal{D}} (\boldsymbol{\mathcal{J}}_h	(\delta_\tau e_{c}^n), 
			\boldsymbol{\mathcal{J}}_h	
			(\delta_\tau e_{c}^n)) = (\delta_\tau e_{c}^n, \mathcal{J}_h(\delta_\tau e_{c}^n))_{\Omega}.
		\end{equation}
		Therefore, with the coercivity of $a_{\mathcal{D}}$ \eqref{eq:aD_coercive},  we obtain
		\begin{equation}  
			\label{eq:termstobound}
			\begin{split}
				&	C_\mathrm{coer} \norm[0]{\boldsymbol{\mathcal{J}}_h(\delta_\tau e_{c}^n)}_{1,h}^2  
				+  \kappa a_{\mathcal{D}}
				(\boldsymbol{e}_{c}^n, 
				\delta_\tau \boldsymbol{e}_{c}^n ) \\ & \leq	( \delta_\tau ( \Pi_h c - c )^n  , 
				\mathcal{J}_h(\delta_\tau 
				e_{c}^n))_{\Omega} 
				+  ( \delta_\tau c^n - (\partial_t c)^n , \mathcal{J}_h(\delta_\tau e_{c}^n))_{\Omega} \\
				& \quad  + a_{\mathcal{D}} 
				(\boldsymbol{\pi}_h \mu^n - \boldsymbol{\mu}^n,  \boldsymbol{\mathcal{J}}_h	(\delta_\tau 
				e_{c}^n))- 
				((c^n)^3- 
				(c_h^n)^3,  \delta_\tau e_{c}^n)_{\Omega} + (c^n- 
				c_h^{n-1},  \delta_\tau e_{c}^n)_{\Omega}  \\ & = T_1 + \cdots + T_5.
			\end{split}
		\end{equation}
		%		\begin{equation} 
		%			\begin{split}
		%				&	C_\mathrm{coer} \norm[0]{\boldsymbol{\mathcal{J}}_h	(\delta_\tau 
		%e_{c,h}^n)}_{1,h}^2  +  
		%\kappa 
		%				a_{\mathcal{D}}
		%				(\boldsymbol{e}_{c,h}^n, 
		%				\delta_\tau \boldsymbol{e}_{c,h}^n ) \\ &= (\delta_\tau (\Pi_h c^n) - (\partial_t \Pi_h 
		%c)^n  , 
		%				\mathcal{J}_h(\delta_\tau 
		%				e_{c,h}^n))_{\mathcal{E}_h} + 	( (\partial_t \Pi_h c - \partial _tc )^n  , 
		%				\mathcal{J}_h(\delta_\tau 
		%				e_{c,h}^n))_{\mathcal{E}_h} + a_{\mathcal{D}} 
		%				(\boldsymbol{\pi}_h \mu^n - \boldsymbol{\mu}^n,  \boldsymbol{\mathcal{J}}_h	
		%(\delta_\tau 
		%				e_{c,h}^n)) \\& \quad - 
		%				((c^n)^3- 
		%				(c_h^n)^3,  \delta_\tau e_{c,h}^n)_{\mathcal{E}_h} + (c^n- 
		%				c_h^{n-1},  \delta_\tau e_{c,h}^n)_{\mathcal{E}_h}  := T_1 + \cdots + T_5.
		%			\end{split}
		%		\end{equation}
		We  proceed to bound each $T_i$, $i=1,\dots,5$. 
		Fix $s$ such that $2\leq s \leq k+1$. By the Cauchy--Schwarz's
		inequality, a Taylor expansion,  the discrete Poincar\'{e} inequality \eqref{eq:poincare} 
		since $\bm{\mathcal J}_h(\delta_\tau e_{c,h}^n) \in \bm{M}_h$,  Young's 
		inequality, and \eqref{eq:approximation_elliptic_projection},  we have for any $\epsilon>0$
		\begin{equation}
			\begin{split}
				|T_1| &\le \norm{\delta_\tau ( \Pi_h c - c )^n  }_{L^2(\Omega)} 
				\norm[0]{\mathcal{J}_h	(\delta_\tau e_{c}^n)}_{L^2(\Omega)} \le C_P \norm{\delta_\tau ( 
					\Pi_h c - c )^n  }_{L^2(\Omega)} 	
				\norm[0]{\boldsymbol{\mathcal{J}}_h	(\delta_\tau e_{c}^n)}_{1,h} \\
				& \le \frac{C_P^2C}{2\epsilon}\tau^{-1} h^{2s-2} \int_{t^{n-1}}^{t^n} \|\partial_t c 
				\|_{H^{s-1}(\Omega)}^2 
				+ \frac{\epsilon}{2} \norm[0]{\boldsymbol{\mathcal{J}}_h	(\delta_\tau e_{c}^n)}_{1,h}^2.
			\end{split}
		\end{equation}
		We bound $T_2$ in a similar fashion:
		\begin{equation}
			\begin{split}
				|T_2| &\le \norm{\delta_\tau c^n - \partial_t c^n }_{L^2(\Omega)} 
				\norm[0]{\mathcal{J}_h	(\delta_\tau e_{c}^n)}_{L^2(\Omega)} \le C_P \norm{\delta_\tau 
					c^n - \partial_t c^n }_{L^2(\Omega)} 
				\norm[0]{\boldsymbol{\mathcal{J}}_h	(\delta_\tau e_{c}^n)}_{1,h} \\
				& \le \frac{C_P^2C }{2\epsilon} \tau\int_{t^{n-1}}^{t^n} \Vert \partial_{tt} 
				c\Vert_{L^2(\Omega)}^2 
				+ \frac{\epsilon}{2} \norm[0]{\boldsymbol{\mathcal{J}}_h	(\delta_\tau e_{c}^n)}_{1,h} ^2.
			\end{split}
		\end{equation}
		To bound $T_3$, we use the definition of the elliptic projection 
		\eqref{eq:broken_elliptic_projection} and the boundedness of the bilinear form $a_{\mathcal{D}}$
		\begin{equation}
			\begin{split}
				T_3 = a_{\mathcal{D}} 
				(\boldsymbol{\pi}_h \mu^n - \boldsymbol{\mu}^n,  \boldsymbol{\mathcal{J}}_h(\delta_\tau 
				e_{c}^n)) &= 	a_{\mathcal{D}} 
				(\boldsymbol{\pi}_h \mu^n - \boldsymbol{\Pi}_h \mu^n, 	  
				\boldsymbol{\mathcal{J}}_h(\delta_\tau 
				e_{c}^n)) \\
				& \le C_\mathrm{cont}\norm{\boldsymbol{\pi}_h \mu^n - \boldsymbol{\Pi}_h \mu^n}_{1,h} 
				\norm[0]{  
					\boldsymbol{\mathcal{J}}_h(\delta_\tau 
					e_{c}^n)}_{1,h}.
			\end{split}
		\end{equation}
		Therefore by the triangle inequality, Young's inequality, \eqref{eq:approximation_property_L^2}, 
		\eqref{eq:approximation_elliptic_projection}, we obtain
		\begin{equation}
			\begin{split}
				|T_3| &\le C_\mathrm{cont}(\norm{\boldsymbol{\pi}_h \mu^n - \bm{\mu}^n}_{1,h} + 
				\norm{\bm{\mu}^n - 
					\boldsymbol{\Pi}_h \mu^n}_{1,h} ) \norm[0]{  
					\boldsymbol{\mathcal{J}}_h(\delta_\tau 
					e_{c}^n)}_{1,h}\\
				& \le  \frac{C_\mathrm{cont}^2 C}{\epsilon}h^{2s-2} \Vert 
				\mu\Vert_{L^{\infty}(0,T;H^{s}(\Omega))}^2 
				+ 
				\frac{\epsilon}{2}  \norm[0]{  
					\boldsymbol{\mathcal{J}}_h(\delta_\tau 
					e_{c}^n)}_{1,h}^2. 
			\end{split} 
		\end{equation}
		To bound $T_4$, we first recall the definitions of $\boldsymbol{z}_h \in \boldsymbol{S}_h$ and 
		$Q$ given in 
		\Cref{prop:Lipschitz_condition} and note that by the definition of the $L^2$-projection $\pi_h$, 
		we 
		have
		\begin{equation}
			|T_4| = |(Q, \delta_\tau e_{c}^n)_{\Omega}| = |(z_h, \delta_\tau e_{c}^n)_{\Omega}|.
		\end{equation}
		By Lemma~\ref{lem:alternate_green_op_properties}, Lemma~\ref{prop:Lipschitz_condition}, the 
		boundedness \eqref{eq:aD_continuous} of the bilinear 
		form 
		$a_{\mathcal{D}}$, and Young's inequality, we have 
		\begin{equation}
			\begin{split}
				|T_4| &=| a_{\mathcal{D}}(\boldsymbol{\mathcal{J}}_h(\delta_\tau 
				e_{c}^n), \boldsymbol{z}_h ) | \\
				&\le C_\mathrm{cont} \norm{\boldsymbol{z}_h}_{1,h} \norm[0]{  
					\boldsymbol{\mathcal{J}}_h(\delta_\tau 
					e_{c}^n)}_{1,h} \\
				&\le \frac{C_\mathrm{cont}^2}{2\epsilon}\left(h^{2s-2} 
				\|c\|^2_{L^{\infty}(0,T;H^{s}(\Omega))} + 
				\|\bm{e}_{c}^n\|_{1,h}^2 \right) + \frac{\epsilon}{2}\norm[0]{  
					\boldsymbol{\mathcal{J}}_h(\delta_\tau 
					e_{c}^n)}_{1,h}^2.
			\end{split}
		\end{equation}
		Finally, to bound $T_5$, observe that by 
		\Cref{lem:alternate_green_op_properties}, Lemma \ref{prop:Lipschitz_condition}, triangle 
		inequality, \eqref{eq:1h_norm_bounded_DG}, \eqref{eq:approximation_elliptic_projection}, and 
		Young's inequality, we have
		\begin{equation}
			\begin{split}
				|T_5| &  \le   C \Vert c^{n}-c_h^{n-1} \Vert_{\mathrm{DG}} \, 
				\Vert\boldsymbol{\mathcal{J}}_h(\delta_\tau e_{c}^n) \Vert_{1,h} \\
				& \le C ( \Vert c^n-c^{n-1}\Vert_{H^1(\Omega)} +
				\Vert c^{n-1}-\Pi_h c^{n-1}\Vert_\mathrm{DG} + \Vert 
				\boldsymbol{e}_c^{n-1}\Vert_\mathrm{DG})
				\, \norm{\boldsymbol{\mathcal{J}}_h(\delta_\tau 
					e_{c}^n)}_{1,h}  \\ 
				& \le \frac{C}{\epsilon} \left(  \tau 
				\int_{t^{n-1}}^{t^n} \|\partial_t c\|_{H^1(\Omega)}^2 
				+ h^{2s-2} \|c\|^2_{L^{\infty}(0,T;H^{s}(\Omega))}  
				+ \| \bm{e}_{c}^{n-1} \|_{1,h}^2 \right)  + 
				\frac{\epsilon}{2} \norm[0]{\boldsymbol{\mathcal{J}}_h(\delta_\tau 
					e_{c}^n)}_{1,h}^2.
			\end{split}
		\end{equation}
		%In the above we used that $\|v\|_{\DG} \lesssim \|\bm{v}\|_{1,h}$ for any $\bm{v} \in 
		%\bm{S}_h$(see 
		%\eqref{eq:1h_norm_bounded_DG}) and that $\|v\|_{\DG} \lesssim \|v\|_{H^1(\Omega)}$ for any $v 
		%\in H^1(\Omega)$ since $[v] = 0$ on $\Gamma_h^0$ and the  semi-norm $\|\cdot\|_{\DG}$ 
		%includes only terms from  $\Gamma_h^0$, see \eqref{eq:def_dGnorm}. 
		%
		Collecting the bounds on $T_i$, $i=1\dots 5$, choosing $\epsilon = C_\mathrm{coer}/5$, and 
		rearranging, we find there are constants $C$ and $\tilde{C}$ independent of $h$, $\tau$, $c$, 
		$\mu$, and $\kappa$ such
		that 
		%
		%	\begin{equation}
		%		\begin{split}
		%		\frac{C_\mathrm{coer}}{2}&	\norm[0]{\boldsymbol{\mathcal{J}}_h	(\delta_\tau 
		%e_{c,h}^n)}_{1,h}^2 +  
		%\kappa \frac{1}{\tau} \left(
		%		a_{\mathcal{D}}
		%		(\boldsymbol{e}_{c,h}^n, 
		%	 \boldsymbol{e}_{c,h}^n )   -	a_{\mathcal{D}}
		%		(\boldsymbol{e}_{c,h}^{n-1}, 
		%	 \boldsymbol{e}_{c,h}^{n-1} ) + C_\mathrm{coer} \|\boldsymbol{e}_{c,h}^n - 
		%\boldsymbol{e}_{c,h}^{n-1}\|_{1,h}^2  \right) 
		%		\\ & \le \frac{5C_1^2}{C_\mathrm{coer}}\norm[0]{c^{n-1}- 
		%			c_h^{n-1}}_{\text{DG}}^2 + \frac{5C_1^2}{C_\mathrm{coer}} \norm[0]{c^{n}- 
		%			c^{n-1}}_{\text{DG}}^2 + \frac{5C_\mathrm{cont}^2}{2C_\mathrm{coer}} 
		%\norm{\boldsymbol{z}_h}_{1,h}^2 +  
		%			\frac{5C_\mathrm{cont}^2}{C_\mathrm{coer}}\norm{\boldsymbol{\mu}^n - 
		%			\boldsymbol{\Pi}_h 
		%			\mu^n}_{1,h}^2 \\ & \quad + 
		%\frac{5C_\mathrm{cont}^2}{C_\mathrm{coer}}\norm{\boldsymbol{\pi}_h \mu^n - 
		%			\boldsymbol{\mu}^n}_{1,h}^2 + \frac{5C_P^2}{2C_\mathrm{coer}} \norm{(\partial_t \Pi_h 
		%c - \partial 
		%_tc )^n
		%		}_{L^2(\Omega)}^2  +  \frac{5C_P^2}{2C_\mathrm{coer}} \norm{\delta_\tau (\Pi_h c^n) - 
		%(\partial_t \Pi_h 
		%c)^n 
		%	}_{L^2(\Omega)}^2 
		%		\end{split}
		%	\end{equation}
		\begin{equation}
			\begin{split}
				\frac{C_\mathrm{coer}}{2}&	\norm[0]{\boldsymbol{\mathcal{J}}_h	(\delta_\tau 
					e_{c}^n)}_{1,h}^2 +  
				\frac{\kappa}{ 2\tau } \left(
				a_{\mathcal{D}}
				(\boldsymbol{e}_{c}^n, 
				\boldsymbol{e}_{c}^n )   -	a_{\mathcal{D}}
				(\boldsymbol{e}_{c}^{n-1}, 
				\boldsymbol{e}_{c}^{n-1} ) + C_\mathrm{coer} \|\boldsymbol{e}_{c}^n - 
				\boldsymbol{e}_{c}^{n-1}\|_{1,h}^2  \right) 
				\\ & \le C \left( \tau^{-1} h^{2s-2} \int_{t^{n-1}}^{t^n} \|\partial_t c \|_{H^{s-1}(\Omega)}^2  + 
				\tau\int_{t^{n-1}}^{t^n}( \Vert \partial_{tt} c\Vert_{L^2(\Omega)}^2  + \|\partial_t 
				c\|^2_{H^1(\Omega)} )\right) \\ & 
				\quad + Ch^{2s-2}  \left(\|c\|^2_{L^{\infty}(0,T;H^{s}(\Omega))}  + 
				\|\mu\|^2_{L^{\infty}(0,T;H^{s}(\Omega))}  \right) + \tilde{C}\Bk \|\bm{e}_{c}^n\|^2_{1,h}
				+ \tilde{C}  \Vert \boldsymbol{e}_c^{n-1}\Vert_{1,h}^2. \Bk 
			\end{split}
		\end{equation}
		Multiplying by $2 \tau$, summing from $n=1$ to $n = m$, noting that 
		$\boldsymbol{e}_c^0 = \mathbf{0}$ by construction, 
		and applying a discrete Gr\"{o}nwall inequality yield for any $1 \leq m \leq N$, 
		\begin{equation}
			\|\bm{e}_{c}^m\|_{1,h}^2 
			+ \tau \sum_{n=1}^m \norm[0]{\boldsymbol{\mathcal{J}}_h	(\delta_\tau e_{c}^n)}_{1,h}^2  
			\lesssim \tau^2 + h^{2s-2},  \label{eq:bound_J_h}
		\end{equation}	
		under the assumption that $\tau$ is small enough, namely $\tau\leq (\kappa 
		C_\mathrm{coer})/(4\tilde{C} )$.
		Hence, the triangle inequality and Lemma~\ref{lem:elliptic_proj_est} yield the bound on $\| 
		\bm{c}^n - \bm{c}_h^n\|_{1,h}$. To obtain a 
		bound on $\bm{e}_{\mu}^n$, test \eqref{eq:error_eq_a} with $\bm{e}_{\mu}^n$. With the 
		coercivity property of $a_\mathcal{D}$, we obtain 
		\begin{align*}
			C_\mathrm{coer} \|\bm{e}_{\mu}^n\|^2_{1,h} \leq - (\delta_\tau e_{c}^n , 
			e_{\mu}^n)_{\Omega} +  
			(\delta_\tau (\Pi_h c^n) - (\partial_t c)^n , e_{\mu}^n)_{\Omega} + a_{\mathcal{D}} 
			(\boldsymbol{\pi}_h \mu^n - \boldsymbol{\mu}^n, \bm{e}_{\mu}^n)   = W_1 +W_2 +W_3.  
		\end{align*}
		The term $W_1$ is bounded by \eqref{eq:alternate_green_op_bnd}, 
		\eqref{eq:1h_norm_bounded_DG}, and Young's inequality:
		\begin{equation}
			|W_1| \leq C \| \boldsymbol{\mathcal{J}}_h(\delta_\tau 
			e_{c}^n)\|_{1,h} \|\bm{e}_{\mu}^n\|_{1,h} \leq \frac{C_\mathrm{coer}}{4}  
			\|\bm{e}_{\mu}^n\|_{1,h}^2 + C 
			\| \boldsymbol{\mathcal{J}}_h(\delta_\tau 
			e_{c}^n)\|_{1,h}^2. 
		\end{equation}
		Choosing $\bm{\chi}_h = (1,1) $ in \eqref{eq:error_eq_a}, we observe that 
		\begin{equation}
			(\delta_\tau (\Pi_h c^n) - (\partial_t c)^n , 1 )_{\Omega} = (\delta_{\tau} 
			e_{c}^n,1)_{\Omega} = 0, 
		\end{equation}
		since $\delta_\tau e_{c}^n \in M_h$. 
		Thus, we bound $W_2$ with the Cauchy-Schwarz's and Poincar\'{e}'s inequalities, the approximation 
		properties of $\Pi_h$ and
		the usual Taylor expansions (similar to the bounds of $T_1$ and $T_2$) as follows: 
		\begin{align}
			|W_2| & = | (\delta_\tau (\Pi_h c^n) - (\partial_t c)^n , e_{\mu}^n - 
			\overline{e_{\mu}^n})_{\Omega} | \leq C_P \|\delta_\tau (\Pi_h c^n) - (\partial_t 
			c)^n\|_{L^2(\Omega)}
			\, \|\bm{e}_{\mu}^n\|_{1,h} \nonumber \\ 
			& \leq \frac{C_\mathrm{coer}}{4}  \|\bm{e}_{\mu}^n\|_{1,h}^2  +  C\left(\tau^{-1} h^{2s-2} 
			\int_{t^{n-1}}^{t^n} 
			\|\partial_t c \|_{H^{s-1}(\Omega)}^2  +  \tau\int_{t^{n-1}}^{t^n} \|\partial_{tt} 
			c\|_{L^2(\Omega)}^2 \right).
		\end{align}
		Similar to $T_3$, we derive that 
		\begin{align*}
			|W_3| = |a_{\mathcal{D}} 
			(\boldsymbol{\pi}_h \mu^n - \boldsymbol{\Pi}_h \mu^n, \bm{e}_{\mu}^n) | \leq 
			\frac{C_\mathrm{coer}}{4}  
			\|\bm{e}_{\mu}^n\|_{1,h}^2  +  Ch^{2s-2} \|\mu\|^2_{L^{\infty}(0,T;H^{s}(\Omega))}. 
		\end{align*}
		Combining the above bounds, multiplying by $\tau$, summing the resulting inequality from $n=1$ 
		to $n=m$, and using \eqref{eq:bound_J_h} and  the triangle inequality yields the error bound for 
		the chemical potential $\mu$. 
		\qed
	\end{proof}
	
	\section{Numerical experiment}
%	
%	\subsection{Transient manufactured solution}
%	
%	
%	
%	\subsection{Merging droplet}
	
	To illustrate the rates of convergence experimentally, we consider a scenario taken from 
	\cite{liu2019numerical} involving the merging or separation of two droplets of one fluid surrounded 
	by another.  We take as the initial condition for the order parameter $c_0$:
	\begin{equation*}
		c_0(x,y) = \begin{cases}
			1,& \quad \text{if} \quad (x,y) \in \sbr{\frac{1}{8},\frac{1}{2}}^2 \cup 
			\sbr{\frac{1}{2},\frac{7}{8}}^2 \\
			-1,& \quad \text{otherwise}
		\end{cases}
	\end{equation*}
	on the two dimensional domain $\Omega = (0,1)^2$. The region occupied by the droplets is 
	indicated by $c=1$, while the region occupied by the surrounding fluid is indicated by $c=-1$. This 
	example has been implemented using Netgen/NGSolve \cite{Schoberl-1997-netgen, 
		Schoberl-2014-ngsolve}.
	We apply the HDG scheme 
	\eqref{eq:CH_disc} with $k=1$ for a sequence of mesh and time step sizes $h_j = 1/2^j$ and $\tau_j 
	= 0.1/2^{2j}$ for $j =3,4,5,6$ and interface parameters $\kappa = 1/2^8, 1/2^{10}, 
	1/2^{12}$ until the end time $T=0.1$. In the absence of an analytical solution, we instead compute 
	the 
	$L^2$ error $\norm[0]{c_{h_{j}} - c_{h_6}}_{L^2(\Omega)}$ at the final time $T=0.1$ for $j=3,4,5$. 
	The order of 
	convergence is estimated via
	\begin{equation}
		\text{rate} = \log_2 \del{ \frac{\norm[0]{c_{h_{j-1}} - c_{h_6}}_{L^2(\Omega)}}{\norm[0]{c_{h_{j}} 
					- 
					c_{h_6}}_{L^2(\Omega)}} },
	\end{equation}
	for $j=3,4,5$. Here, for notational brevity, we have suppressed the time index. The approximate 
	solutions are plotted in \Cref{fig:droplet} and the approximate errors and rates of convergence are 
	listed in \Cref{table:droplet}. We observe merging of the two droplets for $\kappa = 1/2^8,  1/2^{10}$ 
	and separation of the two droplets for $\kappa = 1/2^{12}$. In each case, the order of 
	convergence in the $L^2$-norm appears to approach two. As we have taken $\tau_j = O(h_j^2)$, 
	we observe first order convergence in time and second order convergence in space as predicted by 
	Theorem~\ref{theorem:error_estimate}. 
	
	\begin{table}[h]
		\centering
		\begin{tabular}{ccccccccc} \toprule
						 & &  & 
			\multicolumn{2}{c}{$\kappa = 1/2^8$} & \multicolumn{2}{c}{$\kappa = 1/2^{10}$}  & 
			\multicolumn{2}{c}{$\kappa = 1/2^{12}$ }  \\ \cmidrule(l){4-5} \cmidrule(l){6-7} \cmidrule(l){8-9}
			{$j$} & {$h_j$} & {$\tau_j$} & 
			{Error} & {Rate} & {Error} & {Rate} & 
			{Error} & {Rate}  \\ \midrule
			3  &  { $1/2^3$ } & { $0.1/2^{6}$ }& 5.652{$\times 
				10^{-2}$}  & {---} & 1.422{$\times 
				10^{-1}$}& {---}   & 1.441{$\times 
				10^{-1}$}&{---} \\
			4  &  { $1/2^4$ } & { $0.1/2^{8}$ }& 1.206{$\times 
				10^{-2}$}  & 2.229 & 3.078{$\times 
				10^{-2}$} & 2.208  & 7.822{$\times 
				10^{-2}$}& 0.881  \\
			5  &  { $1/2^5$ }  & { $0.1/2^{10}$ }& 2.403{$\times 
				10^{-3}$}  & 2.327 &6.340{$\times 
				10^{-3}$} & 2.279   & 1.939{$\times 
				10^{-2}$}  & 2.012   \\
			\bottomrule
		\end{tabular}
		\caption{$L^2$ error $\norm[0]{c_{h_{j}} - c_{h_6}}_{L^2(\Omega)}$ between the discrete 
			approximations for the order parameter at the {$j^{\text{th}}$} refinement ($j=3,4, 5$) and the 
			solution on the fine grid $(j=6)$ for various fixed interface parameters $\kappa$ computed at 
			the final time $T= 
			0.1$ and the corresponding estimated rates of convergence.}
		\label{table:droplet}
	\end{table}

	 		\begin{figure}
	 				\centering
	 		\begin{minipage}{1.8in}	
	 			\centering
	 			\hspace{15mm}
	 			$\kappa = 1/2^8$ 
	 		\end{minipage}
	 		\begin{minipage}{1.8in}	
	 			\centering
	 			\hspace{15mm}
	 		$\kappa = 1/2^{10}$ 
	 		\end{minipage}
	 	\begin{minipage}{1.8in}	
	 		\centering
	 		\hspace{15mm}
	 		$\kappa = 1/2^{12}$ 
	 	\end{minipage}
	
	 		\centering
	 				$j=3$
	 		\begin{minipage}{1.8in}
	 			\centering
	 			\includegraphics[width=\textwidth,trim={5cm 5cm 10cm 
	 			5cm},clip]{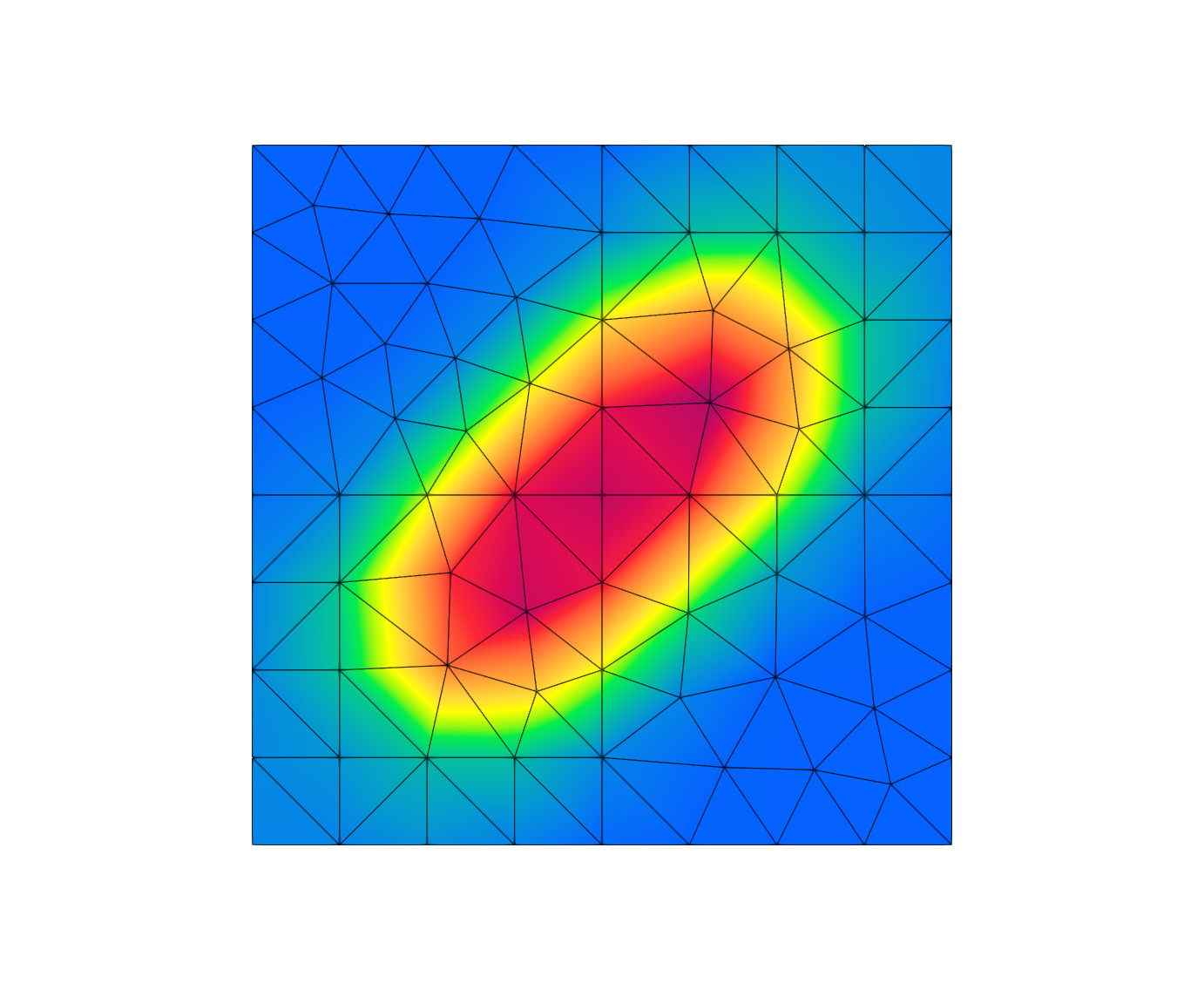}% Tune this to the image height.
	 		\end{minipage}
	 		\begin{minipage}{1.8in}
	 			\centering
	 			\includegraphics[width=\textwidth,trim={5cm 5cm 10cm 
	 				5cm},clip]{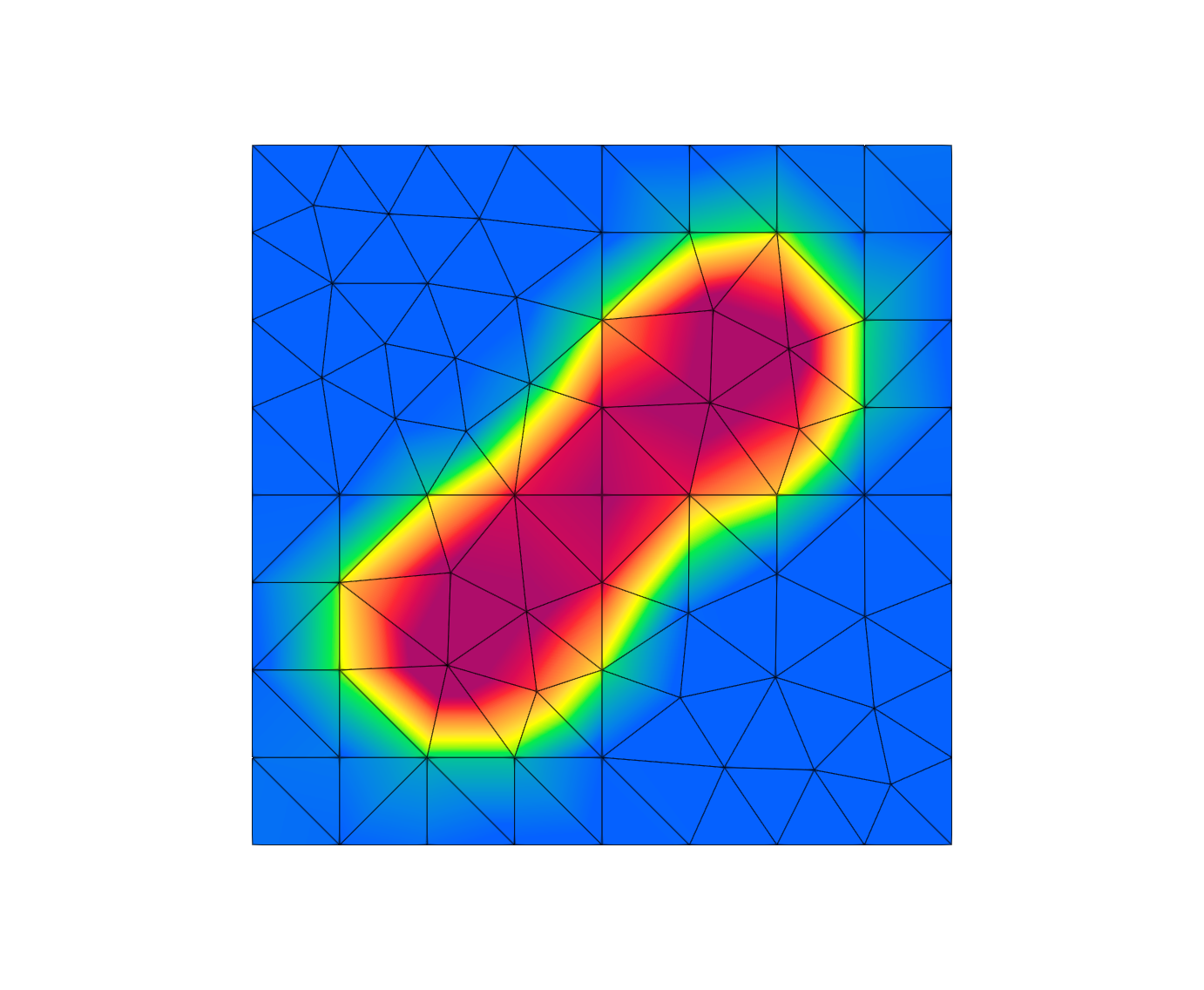}
	 		\end{minipage}
	 		\begin{minipage}{1.8in}
	 			\centering
	 			\includegraphics[width=\textwidth,trim={5cm 5cm 10cm 
	 				5cm},clip]{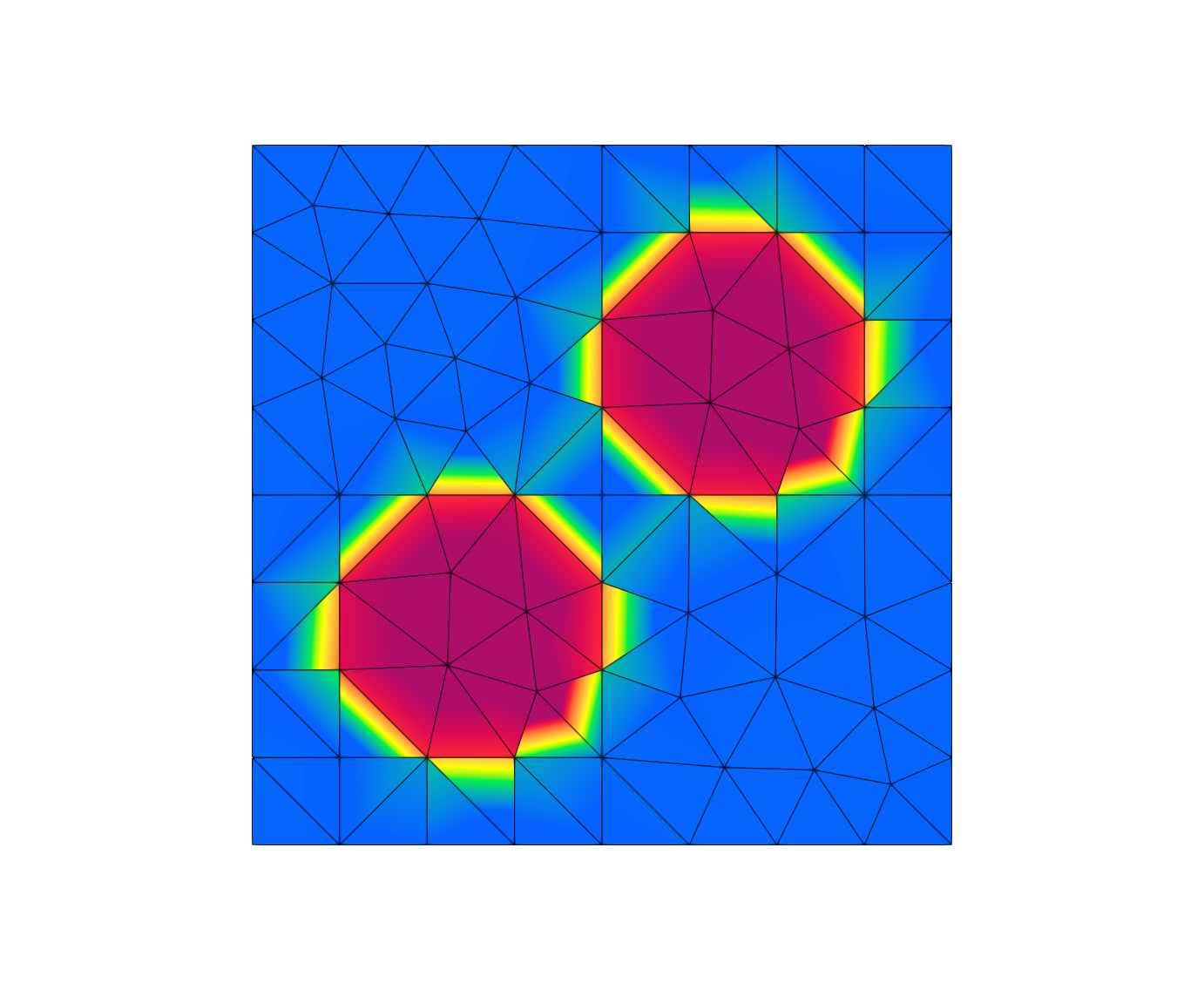}
	 		\end{minipage}
	 	\\
	 			\centering
	 							$j=4$
	 	\begin{minipage}{1.8in}
	 		\centering
	 		\includegraphics[width=\textwidth,trim={5cm 5cm 10cm 
	 			5cm},clip]{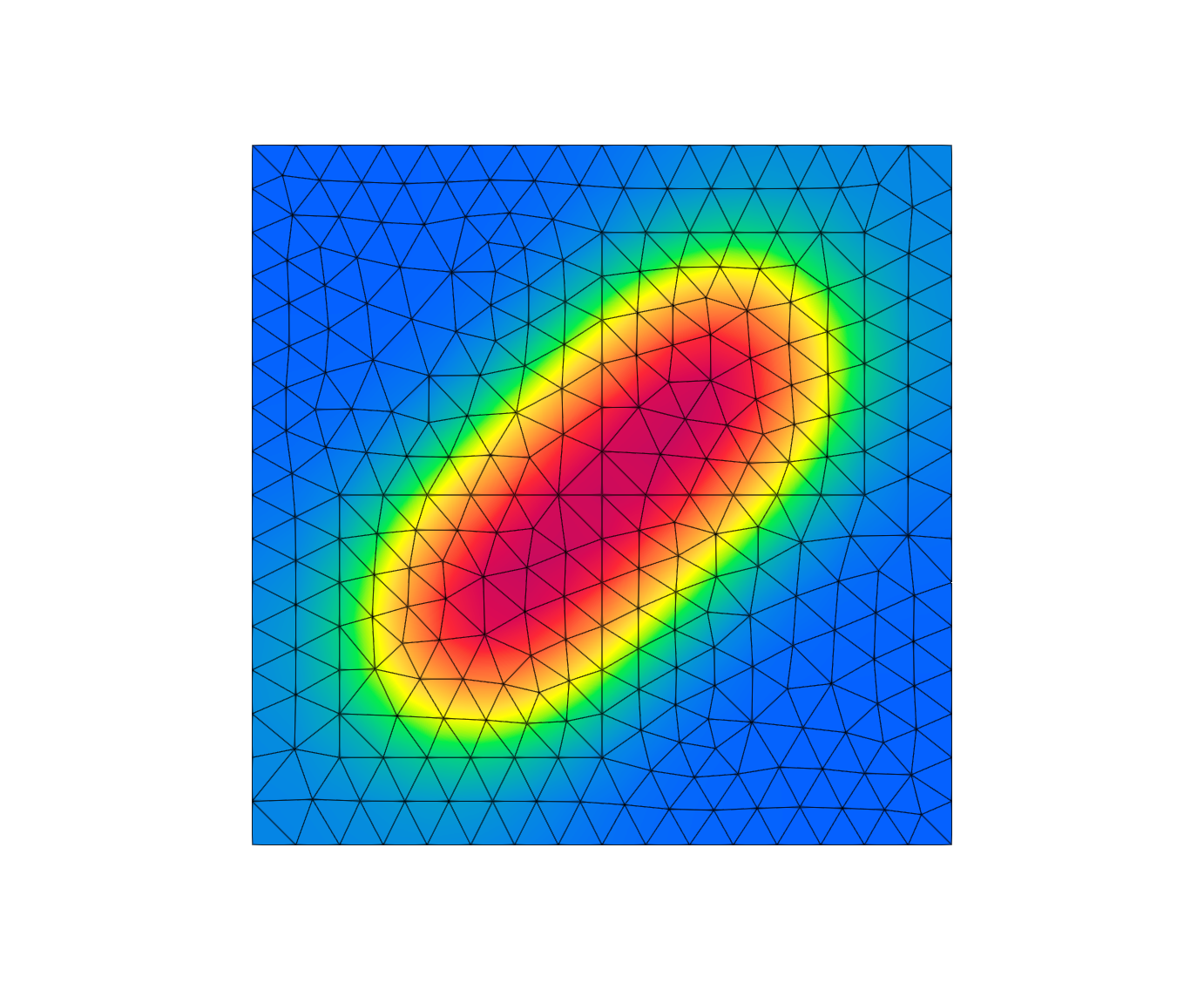}
	 	\end{minipage}
	 	\begin{minipage}{1.8in}
	 		\centering
	 		\includegraphics[width=\textwidth,trim={5cm 5cm 10cm 
	 			5cm},clip]{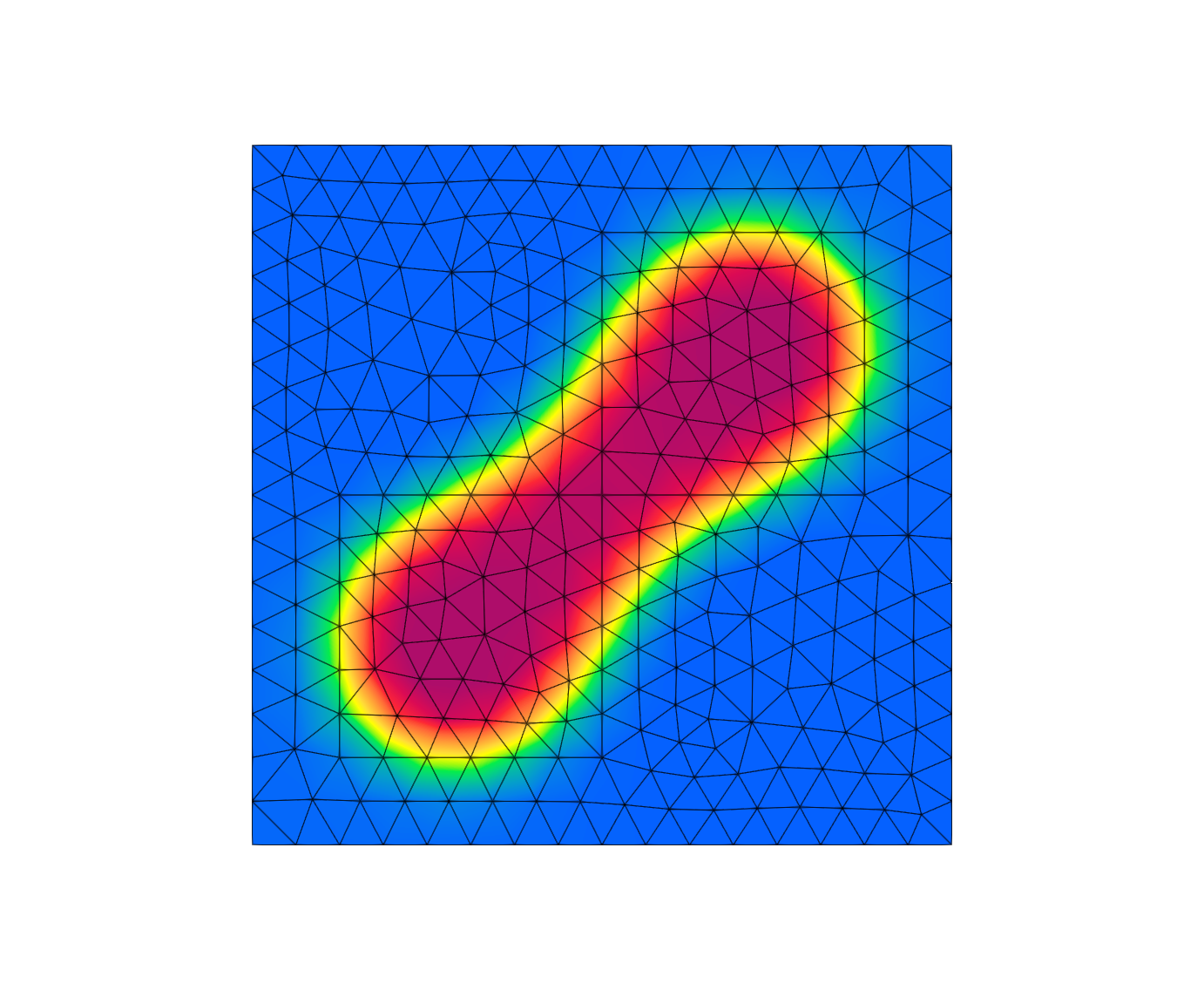}
	 	\end{minipage}
	 	\begin{minipage}{1.8in}
	 		\centering
	 		\includegraphics[width=\textwidth,trim={5cm 5cm 10cm 
	 			5cm},clip]{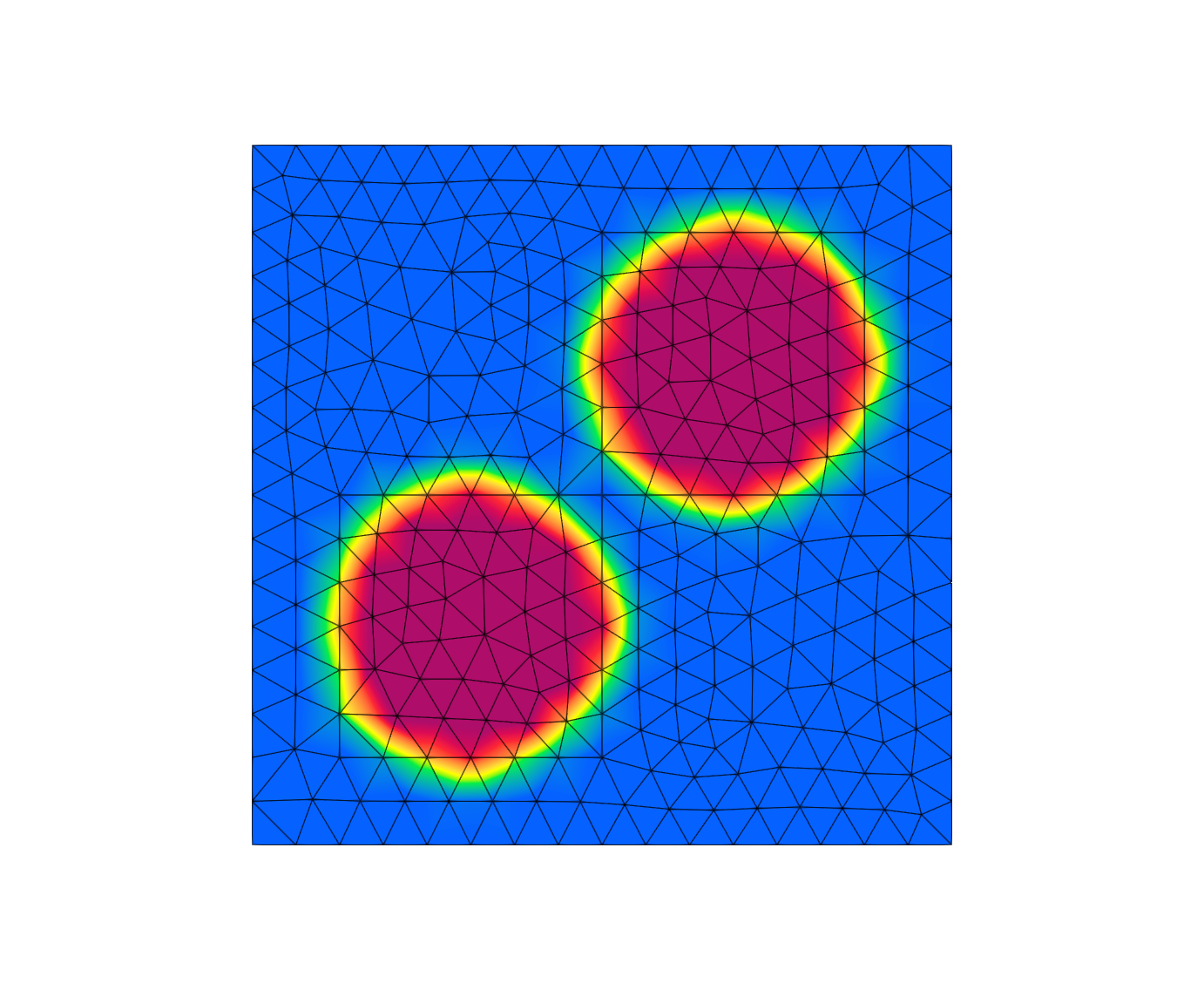}
	 	\end{minipage}
	 \\
	 		\centering
	 						$j=5$
	 \begin{minipage}{1.8in}
	 	\centering
	 	\includegraphics[width=\textwidth,trim={5cm 5cm 10cm 
	 		5cm},clip]{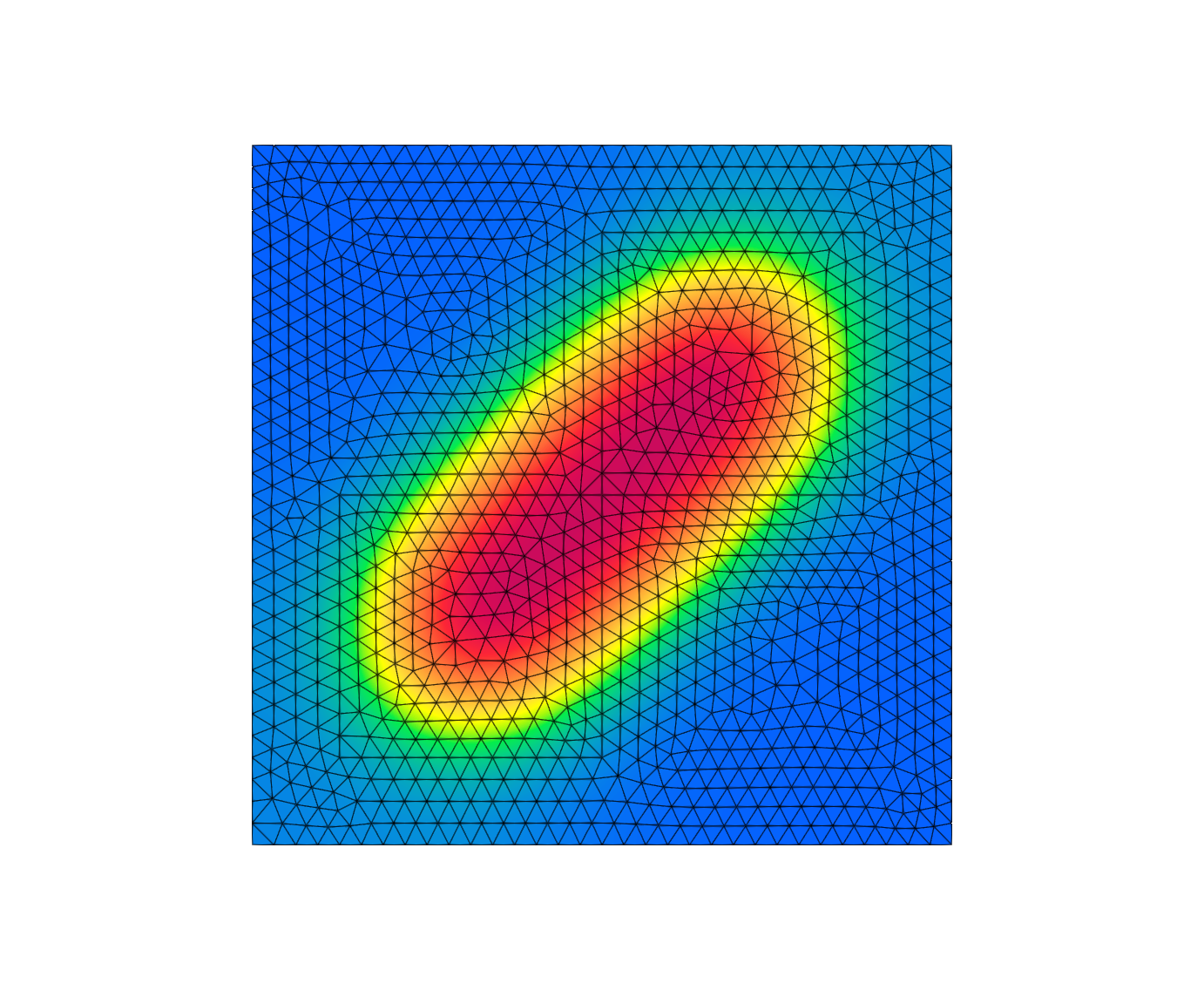}
	 \end{minipage}
	 \begin{minipage}{1.8in}
	 	\centering
	 	\includegraphics[width=\textwidth,trim={5cm 5cm 10cm 
	 		5cm},clip]{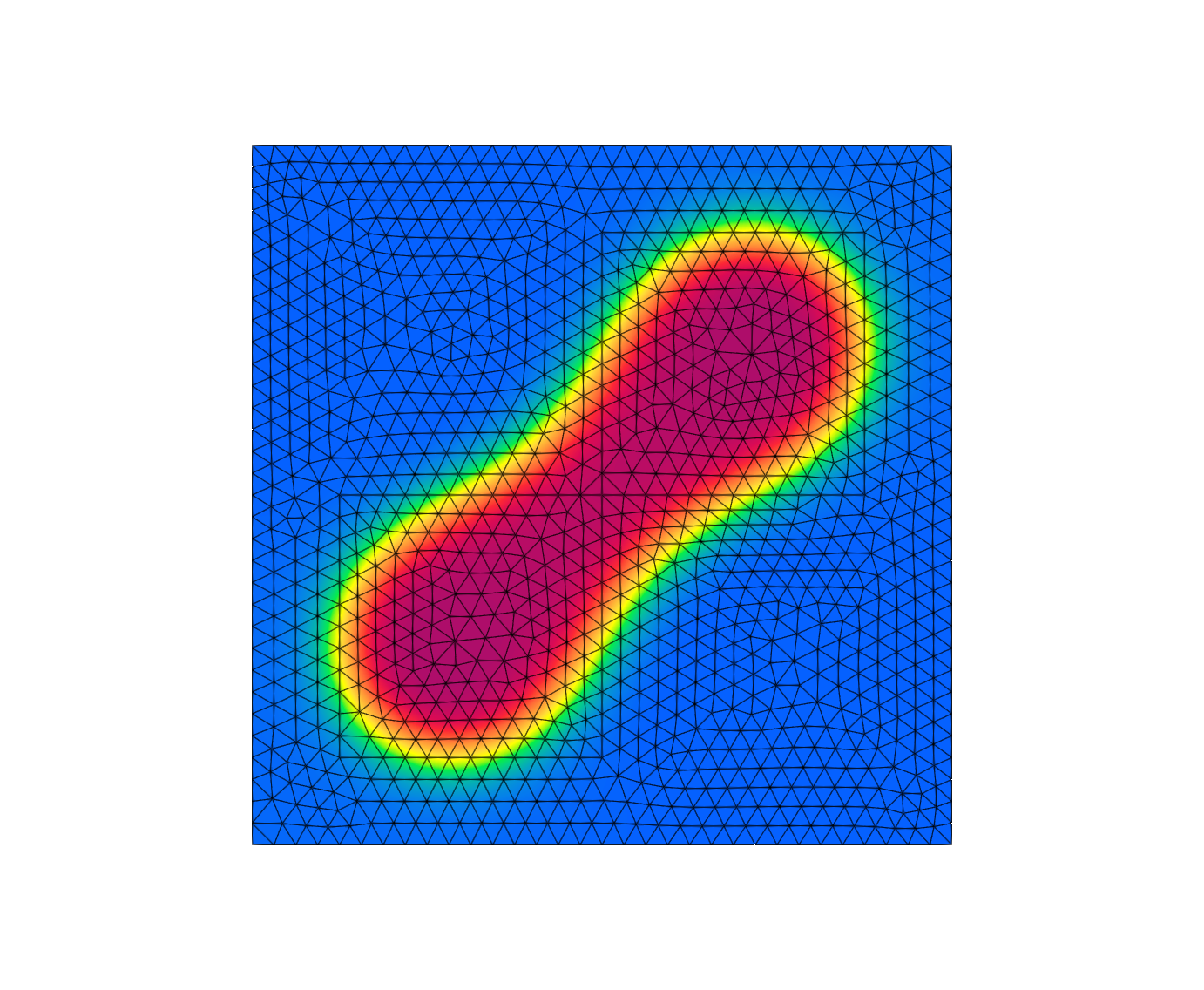}
	 \end{minipage}
	 \begin{minipage}{1.8in}
	 	\centering
	 	\includegraphics[width=\textwidth,trim={5cm 5cm 10cm 
	 		5cm},clip]{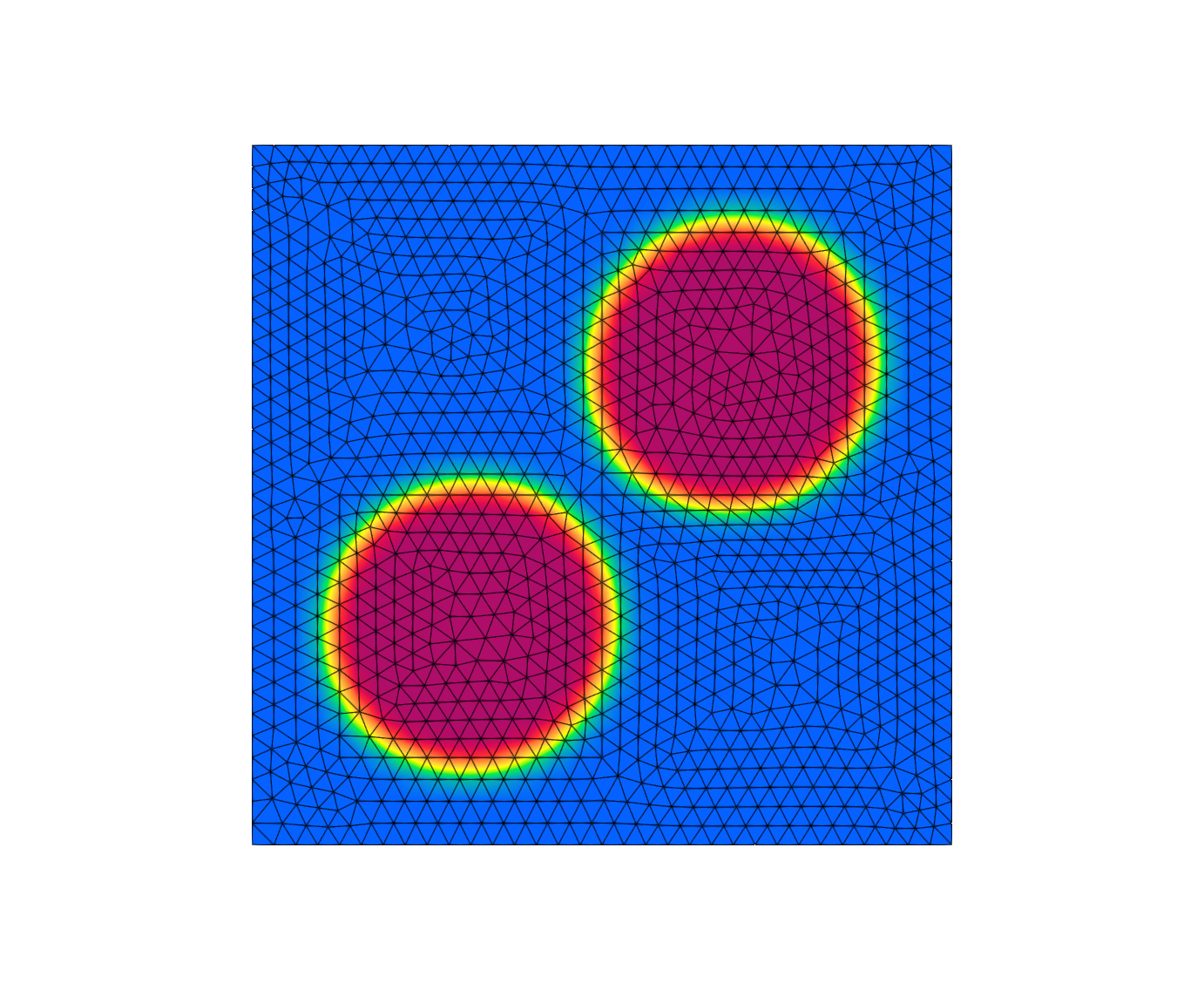}
	 \end{minipage}
	 \\
	 \centering
	 				$j=6$
	 \begin{minipage}{1.8in}
	 	\centering
	 	\includegraphics[width=\textwidth,trim={5cm 5cm 10cm 
	 		5cm},clip]{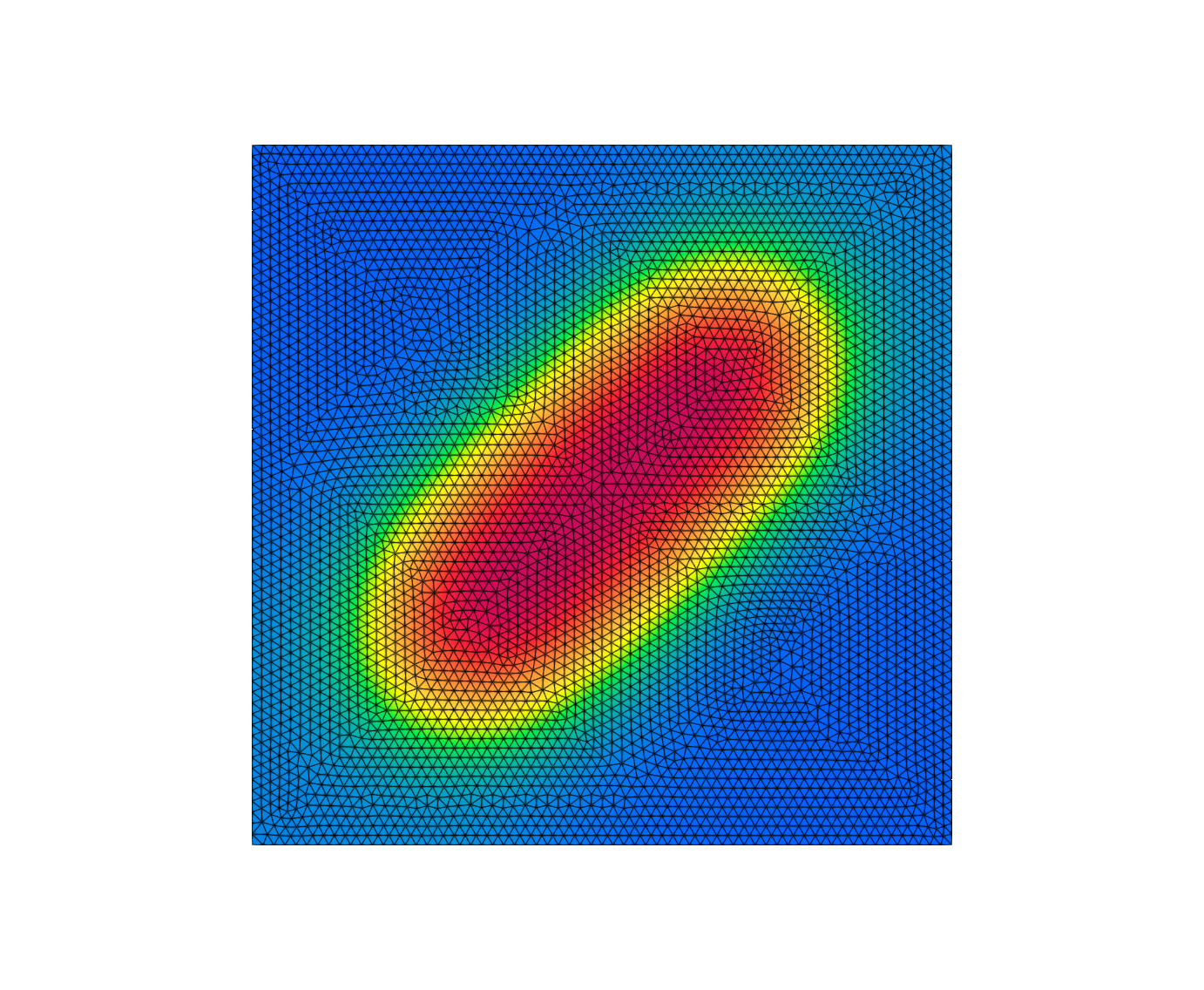}
	 \end{minipage}
	 \begin{minipage}{1.8in}
	 	\centering
	 	\includegraphics[width=\textwidth,trim={5cm 5cm 10cm 
	 		5cm},clip]{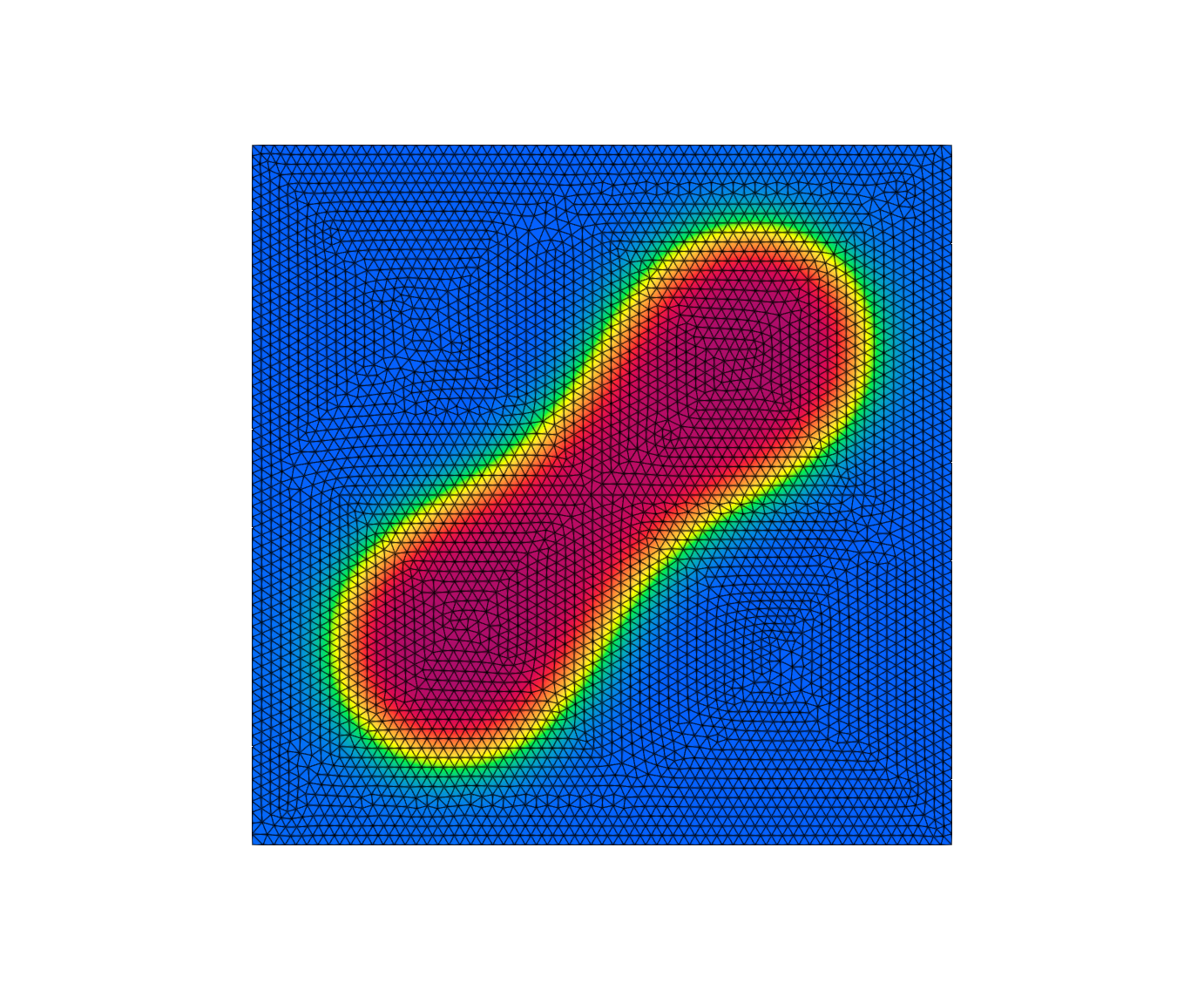}
	 \end{minipage}
	 \begin{minipage}{1.8in}
	 	\centering
	 	\includegraphics[width=\textwidth,trim={5cm 5cm 10cm 
	 		5cm},clip]{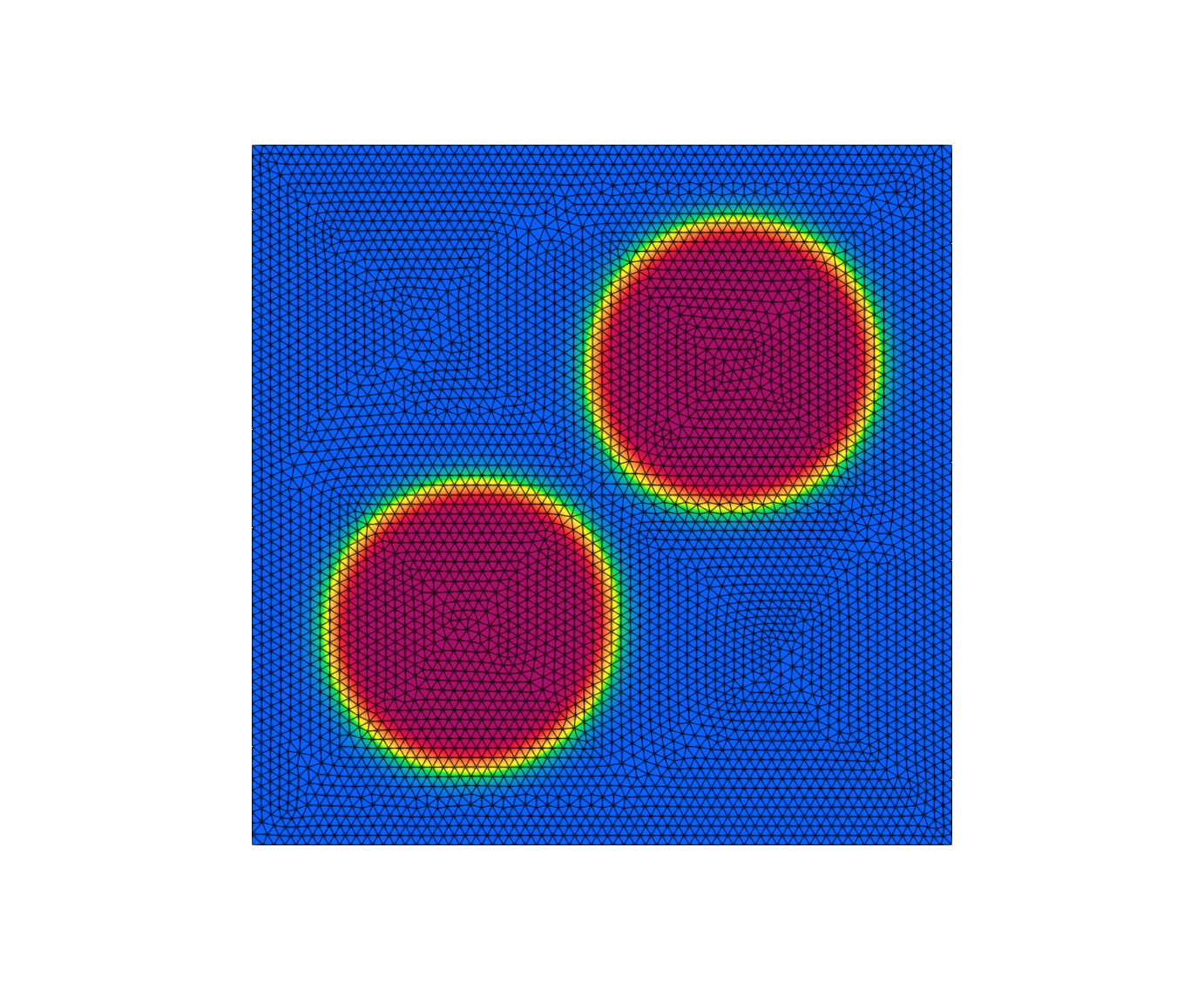}
	 \end{minipage}
	 \caption{ Merging or separating droplets for various fixed time steps and mesh sizes $(\tau_j,h_j)$ 
	 and fixed interface parameters $\kappa$ at the final time $T= 
	 0.1$. See \Cref{table:droplet} for estimated rates of 
	 convergence.}
	 \label{fig:droplet}
	 	\end{figure}

	\section{Conclusion}
	In this paper, we have analyzed a hybridized IPDG method for solving the mixed Cahn--Hilliard system 
	combined with a convex-concave splitting of the chemical energy density and a first order implicit 
	Euler method. We proved the unconditional unique solvability of the nonlinear algebraic system arising 
	from our discretization using techniques from the theory of monotone operators.  We showed the 
	unconditional stability of the scheme for any $\mathcal{C}^2$ potential function, and established the $L^\infty$ 
	stability of the order parameter for the Ginzburg--Landau potential on convex domains. Next, we 
	derived optimal a priori error estimates in space and time for the Ginzburg--Landau potential in a 
	mesh-dependent $H^1$-like norm. Finally, we observed the expected rate of convergence from our 
	theoretical results through a numerical experiment.

	\appendix
	\section{Proof details for \Cref{lem:green_estimates}} \label{appendix:proof_G_prop}
	
	\begin{proof}
		We begin by proving \eqref{eq:green_op_error_1}.  
		To simplify notation,  define $\boldsymbol{\xi} =  \boldsymbol{G}w_h - \boldsymbol{\pi}_h G(w_h)$, 
		$\boldsymbol{\eta}_h = \boldsymbol{\pi}_h G(w_h) - \boldsymbol{G}_h \boldsymbol{w}_h$ 
		so that  $\boldsymbol{G} w_h - \boldsymbol{G}_h \boldsymbol{w}_h = \boldsymbol{\xi} + 
		\boldsymbol{\eta}_h$.  We have from \eqref{eq:galerkin_orth_green} that
		\begin{equation}
			a_{\mathcal{D}}( \boldsymbol{\eta}_h, \boldsymbol{v}_h) 
			= - j_{0}(\bm{w}_h, \bm{v}_h) - 
			a_{\mathcal{D}}( \boldsymbol{\xi},  \boldsymbol{v}_h), \quad \forall \bm{v}_h \in \bm{M}_h. 
		\end{equation}
		Choosing $\boldsymbol{v}_h = \boldsymbol{\eta}_h \in \bm{M}_h$ and using the coercivity 
		\eqref{eq:aD_coercive} and continuity \eqref{eq:aD_continuous} of 
		the bilinear form 
		$a_{\mathcal{D}}$  and the bound on $j_{0}$ \eqref{eq:bound_j0_1h}, we have
		\begin{equation}
			\norm{\boldsymbol{\eta}_h}_{1,h}^2 \lesssim \del{h\norm{\boldsymbol{w}_h}_{0,h}+ 
				\norm{\boldsymbol{\xi}_h}_{1,h,\star} }\norm{\boldsymbol{\eta}_h}_{1,h}.
		\end{equation}
		Using \eqref{eq:approximation_property_L^2}, Lemma~\ref{prop:proj_estimates} and elliptic regularity,  we have
		\begin{equation}
			\norm{\boldsymbol{\xi}_h}_{1,h,\star} \lesssim h \norm{G 
				w_h}_{H^2(\Omega)} \lesssim h \norm{w_h}_{L^2(\Omega)}, \label{eq:B5}
		\end{equation}
		and therefore,
		\begin{equation}
			\norm{\boldsymbol{\eta}_h}_{1,h} \lesssim h\norm{\boldsymbol{w}_h}_{0,h}.\label{eq:B6}
		\end{equation}
		%		%
		This shows \eqref{eq:green_op_error_1}. Next, we show \eqref{eq:green_op_error_2} using a 
		standard duality argument. Consider the 
		following boundary value problem: find $z \in H^2(\Omega) \cap L_0^2(\Omega)$ such that
		\begin{equation} \label{eq:dual_BVP}
			\begin{split}
				-\Delta z =   \pi_h G(w_h) - G_h \boldsymbol{w}_h,  \quad \text{ in } \Omega,  \quad 
				\nabla z \cdot n = 0,  \quad \text{ on } \partial \Omega.
			\end{split}
		\end{equation}
		Since $\Omega$ is convex, $z \in H^2(\Omega)$, we have the following identity for $\boldsymbol{z} = (z,z|_{\Gamma_h})$.
		\begin{equation}
			a_{\mathcal{D}}(\boldsymbol{z}, \boldsymbol{G} w_h - \boldsymbol{G}_h 
			\boldsymbol{w}_h) = \int_{\Omega} (-\Delta z) \del{ G(w_h) - G_h \boldsymbol{w}_h} \dif x.
		\end{equation}
		As $z$ solves \eqref{eq:dual_BVP}, $(\pi_h G(w_h) - G_h \bm{w}_h, G(w_h))_{\Omega} = (\pi_h G(w_h) 
		- G_h \bm{w}_h,\pi_h Gw_h)_{\Omega}$, and the bilinear form $a_{\mathcal{D}}$ is symmetric, 
		we have
		that 	%
		\begin{equation}
			\norm{G_h \boldsymbol{w}_h - \pi_h G w_h}_{L^2(\Omega)}^2 = 	
			a_{\mathcal{D}}(\boldsymbol{G} w_h - \boldsymbol{G}_h 
			\boldsymbol{w}_h, \boldsymbol{z} ).
		\end{equation}
		Let $\boldsymbol{\pi}_h z = (\pi_h z, \hat{\pi}_h z) \in \boldsymbol{M}_h$. By 
		\eqref{eq:galerkin_orth_green}, 
		\begin{equation} \label{eq:green_l2_bnd_1}
			\norm{G_h \boldsymbol{w}_h - \pi_h G w_h}_{L^2(\Omega)}^2 = 	
			a_{\mathcal{D}}(\boldsymbol{G} w_h - \boldsymbol{G}_h 
			\boldsymbol{w}_h, \boldsymbol{z} - \boldsymbol{\pi}_h z)  - j_0(\bm{w}_h, \bm{\pi}_h z). %- 
			%\sum_{E \in 
			%\mathcal{E}_h} h_E \int_{\partial E} (w_h - \hat{w}_h)( \pi_h z - \hat{\pi}_h z) \dif s.
		\end{equation}
		Observe that the first term on the right-hand side of \eqref{eq:green_l2_bnd_1} is bounded by 
		\eqref{eq:aD_extended_continuous2},  \eqref{eq:B5}, \eqref{eq:B6}, \eqref{eq:equivnorms}, \eqref{eq:approximation_property_L^2},  \eqref{eq:l2_proj_elem_face_est} and elliptic regularity:
		\begin{multline}
			|a_{\mathcal{D}}(\boldsymbol{G} w_h - \boldsymbol{G}_h 
			\boldsymbol{w}_h, \boldsymbol{z} - \boldsymbol{\pi}_h z)| \lesssim \norm{ \boldsymbol{G} 
				w_h - 
				\boldsymbol{G}_h 
				\boldsymbol{w}_h}_{1,h} \norm{\boldsymbol{z} - \boldsymbol{\pi}_h z}_{1,h,\star} 
			\lesssim h \norm{\boldsymbol{w}_h}_{0,h} \norm{\boldsymbol{z} - \boldsymbol{\pi}_h 
				z}_{1,h,\star}  \\
			\lesssim h^2 \norm{\boldsymbol{w}_h}_{0,h} \norm{z}_{H^2(\Omega)} 
			\lesssim h^2 \norm{\boldsymbol{w}_h}_{0,h} 	\norm{G_h \boldsymbol{w}_h - \pi_h G(w_h)}_{L^2(\Omega)}.
		\end{multline}
		% \begin{equation}
		% 	\begin{split}
		% 		|a_{\mathcal{D}}(\boldsymbol{G} w_h - \boldsymbol{G}_h 
		% 		\boldsymbol{w}_h, \boldsymbol{z} - \boldsymbol{\pi}_h z)| &\lesssim \norm{ \boldsymbol{G} 
		% 			w_h - 
		% 			\boldsymbol{G}_h 
		% 			\boldsymbol{w}_h}_{1,h} \norm{\boldsymbol{z} - \boldsymbol{\pi}_h z}_{1,h,\star} \\
		% 		& \lesssim h \norm{\boldsymbol{w}_h}_{0,h} \norm{\boldsymbol{z} - \boldsymbol{\pi}_h 
		% 			z}_{1,h,\star}  \\
		% 		& \lesssim h^2 \norm{\boldsymbol{w}_h}_{0,h} \norm{z}_{H^2(\Omega)} \\
		% 		& \lesssim h^2 \norm{\boldsymbol{w}_h}_{0,h} 	\norm{G_h \boldsymbol{w}_h - \pi_h G 
		% 			w_h}_{L^2(\Omega)},
		% 	\end{split}
		% \end{equation}
		% %
		%where we have used the continuity of the bilinear form $a_{\mathcal{D}}$ in the first line, 
		%\eqref{eq:green_op_error_1} in the second line, standard approximation properties of $\pi_h$ 
		%and 
		%$\hat{\pi}_h$ in the third line, and elliptic regularity in the fourth line. 
		
		To bound the second term on the right-hand side of \eqref{eq:green_l2_bnd_1}, we use 
		\eqref{eq:bound_j0_1h},  \eqref{eq:approximation_property_L^2} and
		elliptic regularity:
		\begin{multline}
			j_0(\bm{w}_h, \bm{\pi}_h \bm{z})  = j_0(\bm{w}_h, \bm{\pi}_h \bm{z} - \bm{z})
			\lesssim h \norm{\boldsymbol{w}_h}_{0,h} \norm{\boldsymbol{z} - \boldsymbol{\pi}_h 
				z}_{1,h} \\
			\lesssim h^2 \norm{\boldsymbol{w}_h}_{0,h} \norm{z}_{H^2(\Omega)} 
			\lesssim h^2 \norm{\boldsymbol{w}_h}_{0,h} 	\norm{G_h \boldsymbol{w}_h - \pi_h G(w_h)}_{L^2(\Omega)}.
		\end{multline}
		% \begin{equation}
		% 	\begin{split}
		% 		j_0(\bm{w}_h, \bm{\pi}_h \bm{z}) 
		% 		& \lesssim h \norm{\boldsymbol{w}_h}_{0,h} \norm{\boldsymbol{z} - \boldsymbol{\pi}_h 
		% 			z}_{1,h} \\
		% 		& \lesssim h^2 \norm{\boldsymbol{w}_h}_{0,h} \norm{z}_{H^2(\Omega)} \\
		% 		& \lesssim h^2 \norm{\boldsymbol{w}_h}_{0,h} 	\norm{G_h \boldsymbol{w}_h - \pi_h G 
		% 			w_h}_{L^2(\Omega)}.
		% 	\end{split}
		% \end{equation}
		%
		Therefore,
		\begin{equation}
			\norm{G_h \boldsymbol{w}_h - \pi_h G(w_h)}_{L^2(\Omega)} \lesssim h^2 
			\norm{\boldsymbol{w}_h}_{0,h}.
		\end{equation}
		% %
		% Finally, the approximation properties of $\pi_h$ and elliptic regularity ensure that
		% %
		% \begin{equation}
		% 	\norm{G w_h - \pi_h G w_h} \lesssim h^2 \norm{G w_h}_{H^2(\Omega)} \lesssim h^2 
		% 	\norm{\boldsymbol{w}_h}_{0,h}.
		% \end{equation}
		% %
		% We conclude using the triangle inequality.
		\qed
	\end{proof}
	
	\section{Proof of \eqref{eq:alternate_green_op_bnd}}\label{sec:proof_alternate_green_op_bnd} 
	\begin{proof}
		To prove \eqref{eq:alternate_green_op_bnd}, fix $w_h\in M_h$ and $v\in H^1(\mathcal{E}_h)$. 
		Define define $\boldsymbol{z}_h =  (\pi_h v, \{\pi_h v\}|_{\Gamma_h}) \in \boldsymbol{S}_h$.
		% such that $z_h = \pi_h 
		%		v \in S_h$ and $\hat{z}_h|_{e} =  \av{\pi_h v}_e$, for all $e \in \Gamma_h^0$ and 
		%		$\hat{z}_h|_{e} = (\pi_h v|_{E_e})|_{e}$ for all $e \in \Gamma_h \setminus \Gamma_h^0$, 
		%		where 
		%		$E_e$ 
		%		is the unique element in $\mathcal{E}_h$ such that $e \in \mathcal{F}_{E_e}$.
		%
		By the definition of the orthogonal 
		$L^2$-projection $\pi_h : H^1(\mathcal{E}_h) \to S_h$, the definition of the operator 
		$\boldsymbol{\mathcal{J}}_h$, and the boundedness of the bilinear form $a_{\mathcal{D}}$ we 
		have
		\begin{equation}
			\begin{split}
				|(w_h, v)_{\Omega}| &= |(w_h, \pi_h v)_{\Omega}| \\
				& = |a_{\mathcal{D}}(\boldsymbol{\mathcal{J}}_{h}(w_h), \boldsymbol{z}_h)| 
				\\ & \lesssim \norm[0]{ \boldsymbol{\mathcal{J}}_{h}(w_h)}_{1,h} 
				\norm{\boldsymbol{z}_h}_{1,h}.
			\end{split}
		\end{equation}
		Observe that
		\begin{equation}
			\begin{split}
				\norm{\boldsymbol{z}_h}_{1,h}^2 %&= \norm{\nabla_h \pi_h v}_{L^2(\Omega)}^2 + \sum_{E 
				%\in 
				%\mathcal{E}_h} \frac{1}{h_E} \norm{ z_h - \hat{z}_h}_{L^2(\partial E)}^2 \\
				&= \norm{\nabla_h \pi_h v}_{L^2(\Omega)}^2 + \sum_{E \in \mathcal{E}_h} \sum_{e \in 
					\mathcal{F}_E} \frac{1}{h_E} \norm{ \pi_h v - \{\pi_h v\} }_{L^2(e)}^2 \\
				&= \norm{\nabla_h \pi_h v}_{L^2(\Omega)}^2 + \sum_{E \in \mathcal{E}_h} \sum_{e \in 
					\mathcal{F}_E \cap \Gamma_h^0}  \frac{1}{4h_E}  \norm{ \jump{\pi_h 
						v}}_{L^2(e)}^2 \\
				& = \norm{\nabla_h \pi_h v}_{L^2(\Omega)}^2 +  \sum_{e \in \Gamma_h^0} \sum_{E \in 
					\mathcal{E}_e } \frac{1}{4h_E} \norm{ \jump{\pi_h v}_e}_{L^2(e)}^2,
			\end{split}
		\end{equation}
		where $\mathcal{E}_e$ is the set of the two neighboring elements of $e$, denoted by $E_{e,1}$ 
		and $E_{e,2}$.
		% \cbr{ E \in \mathcal{E}_h \; | \; e \subset \partial E}$. As the mesh is 
		%		assumed conforming, $E_{e} = \cbr{E_{e,1}, E_{e,2}}$ 
		%with $e = E_{e,1} \cap E_{e,2} $ and thus 
		%$\text{card}(\mathcal{E}_e) = 2$ for all $e \in \Gamma_h^0$. 
		Therefore,
		\begin{equation}
			\begin{split}
				\norm{\boldsymbol{z}_h}_{1,h}^2 
				& = \norm{\nabla_h \pi_h v}_{L^2(\Omega)}^2 +  \sum_{e \in \Gamma_h^0} 
				\del[3]{ \frac{1}{4h_{E_{e,1}}} + \frac{1}{4h_{E_{e,2}}} }\norm{ \jump{\pi_h v} }_{L^2(e)}^2 \\
				& \lesssim \norm{\nabla_h \pi_h v}_{L^2(\Omega)}^2 +  \sum_{e \in \Gamma_h^0} 
				\frac{1}{h_e} 
				\norm{ 
					\jump{\pi_h v} }_{L^2(e)}^2.
			\end{split}
		\end{equation}
		Thus, by the stability of the $L^2$-projection in the DG norm we find  
		\begin{equation} \label{eq:stab_l2_semi}
			\norm{\boldsymbol{z}_h}_{1,h}^2 \lesssim \norm{\pi_h v}_{\text{DG}}^2 \lesssim 
			\norm{v}_{\text{DG}}^2.
		\end{equation}
		The result follows.
		\qed
	\end{proof}

	\bibliographystyle{elsarticle-num-names}
	\bibliography{references_arxiv}

\begin{thebibliography}{39}
\expandafter\ifx\csname natexlab\endcsname\relax\def\natexlab#1{#1}\fi
\providecommand{\url}[1]{\texttt{#1}}
\providecommand{\href}[2]{#2}
\providecommand{\path}[1]{#1}
\providecommand{\DOIprefix}{doi:}
\providecommand{\ArXivprefix}{arXiv:}
\providecommand{\URLprefix}{URL: }
\providecommand{\Pubmedprefix}{pmid:}
\providecommand{\doi}[1]{\href{http://dx.doi.org/#1}{\path{#1}}}
\providecommand{\Pubmed}[1]{\href{pmid:#1}{\path{#1}}}
\providecommand{\bibinfo}[2]{#2}
\ifx\xfnm\relax \def\xfnm[#1]{\unskip,\space#1}\fi
%Type = Article
\bibitem[{{C}ahn and {H}illiard(1958)}]{cahn1958free}
\bibinfo{author}{J.~W. {C}ahn}, \bibinfo{author}{J.~E. {H}illiard},
\newblock \bibinfo{title}{Free energy of a nonuniform system. {I}.
  {I}nterfacial free energy},
\newblock \bibinfo{journal}{The Journal of Chemical Physics}
  \bibinfo{volume}{28} (\bibinfo{year}{1958}) \bibinfo{pages}{258--267}.
%Type = Article
\bibitem[{Medina et~al.(2022)Medina, Toledo, Igreja, and
  Rocha}]{medina2022stabilized}
\bibinfo{author}{E.~Y. Medina}, \bibinfo{author}{E.~M. Toledo},
  \bibinfo{author}{I.~Igreja}, \bibinfo{author}{B.~M. Rocha},
\newblock \bibinfo{title}{A stabilized hybrid discontinuous {G}alerkin method
  for the {C}ahn--{H}illiard equation},
\newblock \bibinfo{journal}{Journal of Computational and Applied Mathematics}
  \bibinfo{volume}{406} (\bibinfo{year}{2022}) \bibinfo{pages}{114025}.
%Type = Article
\bibitem[{Agosti et~al.(2017)Agosti, Antonietti, Ciarletta, Grasselli, and
  Verani}]{agosti2017cahn}
\bibinfo{author}{A.~Agosti}, \bibinfo{author}{P.~F. Antonietti},
  \bibinfo{author}{P.~Ciarletta}, \bibinfo{author}{M.~Grasselli},
  \bibinfo{author}{M.~Verani},
\newblock \bibinfo{title}{A {C}ahn-{H}illiard--type equation with application
  to tumor growth dynamics},
\newblock \bibinfo{journal}{Mathematical Methods in the Applied Sciences}
  \bibinfo{volume}{40} (\bibinfo{year}{2017}) \bibinfo{pages}{7598--7626}.
%Type = Article
\bibitem[{Fu(2020)}]{fu2020CHNS}
\bibinfo{author}{G.~Fu},
\newblock \bibinfo{title}{A divergence-free {HDG} scheme for the
  {C}ahn-{H}illiard phase-field model for two-phase incompressible flow},
\newblock \bibinfo{journal}{Journal of Computational Physics}
  \bibinfo{volume}{419} (\bibinfo{year}{2020}) \bibinfo{pages}{109671}.
  \DOIprefix\doi{https://doi.org/10.1016/j.jcp.2020.109671}.
%Type = Article
\bibitem[{Liu et~al.(2020)Liu, Frank, Thiele, Alpak, Berg, Chapman, and
  Riviere}]{liu2020efficient}
\bibinfo{author}{C.~Liu}, \bibinfo{author}{F.~Frank},
  \bibinfo{author}{C.~Thiele}, \bibinfo{author}{F.~O. Alpak},
  \bibinfo{author}{S.~Berg}, \bibinfo{author}{W.~Chapman},
  \bibinfo{author}{B.~Riviere},
\newblock \bibinfo{title}{An efficient numerical algorithm for solving
  viscosity contrast {C}ahn--{H}illiard--{N}avier--{S}tokes system in porous
  media},
\newblock \bibinfo{journal}{Journal of Computational Physics}
  \bibinfo{volume}{400} (\bibinfo{year}{2020}) \bibinfo{pages}{108948}.
%Type = Article
\bibitem[{Liu and Riviere(2020)}]{liu2020priori}
\bibinfo{author}{C.~Liu}, \bibinfo{author}{B.~Riviere},
\newblock \bibinfo{title}{A priori error analysis of a discontinuous {G}alerkin
  method for {{C}ahn--{H}illiard--Navier--Stokes} equations},
\newblock \bibinfo{journal}{CSIAM Trans. Appl. Math} \bibinfo{volume}{1}
  (\bibinfo{year}{2020}) \bibinfo{pages}{104--141}.
%Type = Article
\bibitem[{Elliott and French(1989)}]{elliottfrench1989cahnmorley}
\bibinfo{author}{C.~M. Elliott}, \bibinfo{author}{D.~A. French},
\newblock \bibinfo{title}{A nonconforming finite-element method for the
  two-dimensional {C}ahn-{H}illiard equation},
\newblock \bibinfo{journal}{SIAM Journal on Numerical Analysis}
  \bibinfo{volume}{26} (\bibinfo{year}{1989}) \bibinfo{pages}{884--903}.
  \DOIprefix\doi{10.2307/2157884}.
%Type = Article
\bibitem[{Feng and Karakashian(2007)}]{feng2007fully}
\bibinfo{author}{X.~Feng}, \bibinfo{author}{O.~Karakashian},
\newblock \bibinfo{title}{Fully discrete dynamic mesh discontinuous {G}alerkin
  methods for the {C}ahn-{H}illiard equation of phase transition},
\newblock \bibinfo{journal}{Mathematics of Computation} \bibinfo{volume}{76}
  (\bibinfo{year}{2007}) \bibinfo{pages}{1093--1117}.
%Type = Article
\bibitem[{Wells et~al.(2006)Wells, Kuhl, and
  Garikipati}]{wells2006discontinuous}
\bibinfo{author}{G.~N. Wells}, \bibinfo{author}{E.~Kuhl},
  \bibinfo{author}{K.~Garikipati},
\newblock \bibinfo{title}{A discontinuous {G}alerkin method for the
  {C}ahn--{H}illiard equation},
\newblock \bibinfo{journal}{Journal of Computational Physics}
  \bibinfo{volume}{218} (\bibinfo{year}{2006}) \bibinfo{pages}{860--877}.
%Type = Article
\bibitem[{Aristotelous et~al.(2015)Aristotelous, Karakashian, and
  Wise}]{aristotelous2015adaptive}
\bibinfo{author}{A.~C. Aristotelous}, \bibinfo{author}{O.~A. Karakashian},
  \bibinfo{author}{S.~M. Wise},
\newblock \bibinfo{title}{Adaptive, second-order in time, primitive-variable
  discontinuous {G}alerkin schemes for a {C}ahn--{H}illiard equation with a
  mass source},
\newblock \bibinfo{journal}{IMA Journal of Numerical Analysis}
  \bibinfo{volume}{35} (\bibinfo{year}{2015}) \bibinfo{pages}{1167--1198}.
%Type = Article
\bibitem[{Elliott et~al.(1989)Elliott, French, and
  Milner}]{elliottfrench1989cahn}
\bibinfo{author}{C.~M. Elliott}, \bibinfo{author}{D.~A. French},
  \bibinfo{author}{F.~A. Milner},
\newblock \bibinfo{title}{A second order splitting method for the
  {C}ahn--{H}illiard equation},
\newblock \bibinfo{journal}{Numerische Mathematik} \bibinfo{volume}{54}
  (\bibinfo{year}{1989}) \bibinfo{pages}{575--590}.
%Type = Article
\bibitem[{Xia et~al.(2007)Xia, Xu, and Shu}]{shu2007ldgcahn}
\bibinfo{author}{Y.~Xia}, \bibinfo{author}{Y.~Xu}, \bibinfo{author}{C.-W. Shu},
\newblock \bibinfo{title}{Local discontinuous {G}alerkin methods for the
  {C}ahn--{H}illiard type equations},
\newblock \bibinfo{journal}{Journal of Computational Physics}
  \bibinfo{volume}{227} (\bibinfo{year}{2007}) \bibinfo{pages}{472--491}.
  \DOIprefix\doi{https://doi.org/10.1016/j.jcp.2007.08.001}.
%Type = Article
\bibitem[{Song and Shu(2017)}]{shu2017ldgcahn}
\bibinfo{author}{H.~Song}, \bibinfo{author}{C.-W. Shu},
\newblock \bibinfo{title}{Unconditional energy stability analysis of a second
  order implicit-explicit local discontinuous {G}alerkin method for the
  {C}ahn--{H}illiard equation},
\newblock \bibinfo{journal}{Journal of Scientific Computing}
  \bibinfo{volume}{73} (\bibinfo{year}{2017}) \bibinfo{pages}{1178--1203}.
  \DOIprefix\doi{https://doi.org/10.1007/s10915-017-0497-5}.
%Type = Article
\bibitem[{Yan and Xu(2021)}]{yan2021ldgcahn}
\bibinfo{author}{F.~Yan}, \bibinfo{author}{Y.~Xu},
\newblock \bibinfo{title}{Error analysis of an unconditionally energy stable
  local discontinuous {G}alerkin scheme for the {C}ahn–{H}illiard equation
  with concentration-dependent mobility},
\newblock \bibinfo{journal}{Comput. Methods Appl. Math.} \bibinfo{volume}{21}
  (\bibinfo{year}{2021}). \DOIprefix\doi{10.1515/cmam-2020-0066}.
%Type = Article
\bibitem[{Kay et~al.(2009)Kay, Styles, and S{\"u}li}]{kay2009discontinuous}
\bibinfo{author}{D.~Kay}, \bibinfo{author}{V.~Styles},
  \bibinfo{author}{E.~S{\"u}li},
\newblock \bibinfo{title}{Discontinuous {G}alerkin finite element approximation
  of the {C}ahn--{H}illiard equation with convection},
\newblock \bibinfo{journal}{SIAM Journal on Numerical Analysis}
  \bibinfo{volume}{47} (\bibinfo{year}{2009}) \bibinfo{pages}{2660--2685}.
%Type = Article
\bibitem[{Aristotelous et~al.(2013)Aristotelous, Karakashian, and
  Wise}]{aristotelous2013mixed}
\bibinfo{author}{A.~C. Aristotelous}, \bibinfo{author}{O.~Karakashian},
  \bibinfo{author}{S.~M. Wise},
\newblock \bibinfo{title}{A mixed discontinuous {G}alerkin, convex splitting
  scheme for a modified {C}ahn-{H}illiard equation and an efficient nonlinear
  multigrid solver},
\newblock \bibinfo{journal}{Discrete \& Continuous Dynamical Systems-B}
  \bibinfo{volume}{18} (\bibinfo{year}{2013}) \bibinfo{pages}{2211}.
%Type = Article
\bibitem[{Liu et~al.(2019)Liu, Frank, and Riviere}]{liu2019numerical}
\bibinfo{author}{C.~Liu}, \bibinfo{author}{F.~Frank}, \bibinfo{author}{B.~M.
  Riviere},
\newblock \bibinfo{title}{Numerical error analysis for nonsymmetric interior
  penalty discontinuous {G}alerkin method of {C}ahn--{H}illiard equation},
\newblock \bibinfo{journal}{Numerical Methods for Partial Differential
  Equations} \bibinfo{volume}{35} (\bibinfo{year}{2019})
  \bibinfo{pages}{1509--1537}.
%Type = Article
\bibitem[{Cockburn et~al.(2009)Cockburn, Gopalakrishnan, and
  Lazarov}]{Cockburn:2009}
\bibinfo{author}{B.~Cockburn}, \bibinfo{author}{J.~Gopalakrishnan},
  \bibinfo{author}{R.~Lazarov},
\newblock \bibinfo{title}{Unified hybridization of discontinuous {G}alerkin,
  mixed, and continuous {G}alerkin methods for second order elliptic problems},
\newblock \bibinfo{journal}{SIAM Journal on Numerical Analysis}
  \bibinfo{volume}{47} (\bibinfo{year}{2009}) \bibinfo{pages}{1319--1365}.
  \DOIprefix\doi{10.1137/070706616}.
%Type = Article
\bibitem[{Chen et~al.(2023)Chen, Han, Singler, and Zhang}]{chen2023LHDG}
\bibinfo{author}{G.~Chen}, \bibinfo{author}{D.~Han}, \bibinfo{author}{J.~R.
  Singler}, \bibinfo{author}{Y.~Zhang},
\newblock \bibinfo{title}{On the superconvergence of a hybridizable
  discontinuous {G}alerkin method for the {C}ahn–{H}illiard equation},
\newblock \bibinfo{journal}{SIAM Journal on Numerical Analysis}
  \bibinfo{volume}{61} (\bibinfo{year}{2023}) \bibinfo{pages}{83--109}.
  \DOIprefix\doi{10.1137/21M1437780}.
%Type = Article
\bibitem[{{Di Pietro} and Ern(2015)}]{DiPietro:2015}
\bibinfo{author}{D.~A. {Di Pietro}}, \bibinfo{author}{A.~Ern},
\newblock \bibinfo{title}{A hybrid high-order locking-free method for linear
  elasticity on general meshes},
\newblock \bibinfo{journal}{Computer Methods in Applied Mechanics and
  Engineering} \bibinfo{volume}{283} (\bibinfo{year}{2015})
  \bibinfo{pages}{1--21}. \DOIprefix\doi{10.1016/j.cma.2014.09.009}.
%Type = Article
\bibitem[{Chave et~al.(2016)Chave, Di~Pietro, Marche, and
  Pigeonneau}]{Chave:2016}
\bibinfo{author}{F.~Chave}, \bibinfo{author}{D.~A. Di~Pietro},
  \bibinfo{author}{F.~Marche}, \bibinfo{author}{F.~Pigeonneau},
\newblock \bibinfo{title}{A hybrid high-order method for the
  {{C}ahn--{H}illiard} problem in mixed form},
\newblock \bibinfo{journal}{SIAM Journal on Numerical Analysis}
  \bibinfo{volume}{54} (\bibinfo{year}{2016}) \bibinfo{pages}{1873--1898}.
  \DOIprefix\doi{10.1137/15M1041055}.
%Type = Inproceedings
\bibitem[{Chave et~al.(2017)Chave, Di~Pietro, and Marche}]{chave2017HHOconv}
\bibinfo{author}{F.~Chave}, \bibinfo{author}{D.~A. Di~Pietro},
  \bibinfo{author}{F.~Marche},
\newblock \bibinfo{title}{A hybrid high-order method for the convective
  cahn--hilliard problem in mixed  form},
\newblock in: \bibinfo{booktitle}{Finite Volumes for Complex Applications VIII
  - Hyperbolic, Elliptic and Parabolic Problems}, \bibinfo{publisher}{Springer
  International Publishing}, \bibinfo{year}{2017}, pp.
  \bibinfo{pages}{517--525}. \DOIprefix\doi{10.1007/978-3-319-57394-6_54}.
%Type = Article
\bibitem[{Cockburn et~al.(2016)Cockburn, Pietro, and Ern}]{Cockburn:2016}
\bibinfo{author}{B.~Cockburn}, \bibinfo{author}{D.~A.~D. Pietro},
  \bibinfo{author}{A.~Ern},
\newblock \bibinfo{title}{Bridging the hybrid high-order and hybridizable
  discontinuous {G}alerkin methods},
\newblock \bibinfo{journal}{ESAIM: M2AN} \bibinfo{volume}{50}
  (\bibinfo{year}{2016}) \bibinfo{pages}{635--650}.
  \DOIprefix\doi{10.1051/m2an/2015051}.
%Type = Article
\bibitem[{Barrett et~al.(1999)Barrett, Blowey, and Garcke}]{barrett1999finite}
\bibinfo{author}{J.~W. Barrett}, \bibinfo{author}{J.~F. Blowey},
  \bibinfo{author}{H.~Garcke},
\newblock \bibinfo{title}{Finite element approximation of the
  {C}ahn--{H}illiard equation with degenerate mobility},
\newblock \bibinfo{journal}{SIAM Journal on Numerical Analysis}
  \bibinfo{volume}{37} (\bibinfo{year}{1999}) \bibinfo{pages}{286--318}.
%Type = Article
\bibitem[{{Di Pietro} and Droniou(2017)}]{DiPietro:2017}
\bibinfo{author}{D.~A. {Di Pietro}}, \bibinfo{author}{J.~Droniou},
\newblock \bibinfo{title}{A hybrid high-order method for {L}eray--{L}ions
  elliptic equations on general meshes},
\newblock \bibinfo{journal}{Math. Comp.} \bibinfo{volume}{86}
  (\bibinfo{year}{2017}) \bibinfo{pages}{2159--2191}.
  \DOIprefix\doi{10.1090/mcom/3180}.
%Type = Book
\bibitem[{Di~Pietro and Ern(2012)}]{Pietro:book}
\bibinfo{author}{D.~A. Di~Pietro}, \bibinfo{author}{A.~Ern},
  \bibinfo{title}{Mathematical {A}spects of {D}iscontinuous {G}alerkin
  {M}ethods}, volume~\bibinfo{volume}{69} of
  \textit{\bibinfo{series}{Math\'ematiques et Applications}},
  \bibinfo{publisher}{Springer--Verlag Berlin Heidelberg},
  \bibinfo{year}{2012}.
%Type = Article
\bibitem[{Lasis and S\"uli(2003)}]{Lasis1}
\bibinfo{author}{A.~Lasis}, \bibinfo{author}{E.~S\"uli},
\newblock \bibinfo{title}{Poincar\'{e}-type inequalities for broken {S}obolev
  spaces},
\newblock \bibinfo{journal}{Internal Report 03/10, Oxford University Computing
  Laboratory}  (\bibinfo{year}{2003}).
%Type = Article
\bibitem[{Lasis and S\"uli(2007)}]{Lasis2}
\bibinfo{author}{A.~Lasis}, \bibinfo{author}{E.~S\"uli},
\newblock \bibinfo{title}{hp-{V}ersion discontinuous {G}alerkin finite element
  method for semilinear parabolic problems},
\newblock \bibinfo{journal}{SIAM Journal on Numerical Analysis}
  \bibinfo{volume}{45} (\bibinfo{year}{2007}) \bibinfo{pages}{1544--1569}.
%Type = Book
\bibitem[{Brenner and Scott(2008)}]{Brenner:book}
\bibinfo{author}{S.~C. Brenner}, \bibinfo{author}{L.~R. Scott},
  \bibinfo{title}{The {M}athematical {T}heory of {F}inite {E}lement {M}ethods},
  volume~\bibinfo{volume}{15} of \textit{\bibinfo{series}{Texts in {A}pplied
  {M}athematics}}, \bibinfo{publisher}{Springer}, \bibinfo{year}{2008}.
%Type = Article
\bibitem[{Fabien et~al.(2020)Fabien, Kneply, and Riviere}]{Fabien:2020}
\bibinfo{author}{M.~Fabien}, \bibinfo{author}{M.~Kneply},
  \bibinfo{author}{B.~Riviere},
\newblock \bibinfo{title}{Families of interior penalty hybridizable
  discontinuous {G}alerkin methods for second order elliptic problems},
\newblock \bibinfo{journal}{Journal of Numerical Mathematics}
  \bibinfo{volume}{28} (\bibinfo{year}{2020}) \bibinfo{pages}{161--174}.
  \DOIprefix\doi{10.1515/jnma-2019-0027}.
%Type = Article
\bibitem[{Rhebergen and Wells(2017)}]{Rhebergen:2017}
\bibinfo{author}{S.~Rhebergen}, \bibinfo{author}{G.~N. Wells},
\newblock \bibinfo{title}{Analysis of a hybridized/interface stabilized finite
  element method for the {S}tokes equations},
\newblock \bibinfo{journal}{SIAM Journal on Numerical Analysis}
  \bibinfo{volume}{55} (\bibinfo{year}{2017}) \bibinfo{pages}{1982--2003}.
  \DOIprefix\doi{10.1137/16M1083839}.
%Type = Article
\bibitem[{Kay et~al.(2009)Kay, Styles, and S\"{u}li}]{Kay:2009}
\bibinfo{author}{D.~Kay}, \bibinfo{author}{V.~Styles},
  \bibinfo{author}{E.~S\"{u}li},
\newblock \bibinfo{title}{Discontinuous {G}alerkin finite element approximation
  of the {C}ahn--{H}illiard equation with convection},
\newblock \bibinfo{journal}{SIAM Journal on Numerical Analysis}
  \bibinfo{volume}{47} (\bibinfo{year}{2009}) \bibinfo{pages}{2660--2685}.
  \DOIprefix\doi{10.1137/080726768}.
%Type = Book
\bibitem[{Constantin and Foias(1988)}]{Constantin:book}
\bibinfo{author}{P.~Constantin}, \bibinfo{author}{C.~Foias},
  \bibinfo{title}{{N}avier--{S}tokes Equations}, Chicago Lectures in
  Mathematics, \bibinfo{publisher}{The University of Chicago Press},
  \bibinfo{year}{1988}.
%Type = Article
\bibitem[{Gagliardo(1959)}]{Gagliardo1959}
\bibinfo{author}{E.~Gagliardo},
\newblock \bibinfo{title}{Ulteriori propriet`a di alcune classi di funzioni in
  piu` variabili},
\newblock \bibinfo{journal}{Ricerche Mat.} \bibinfo{volume}{8}
  (\bibinfo{year}{1959}) \bibinfo{pages}{24--51}.
%Type = Article
\bibitem[{Li and Zhang(2022)}]{LiZhang2022}
\bibinfo{author}{C.~Li}, \bibinfo{author}{K.~Zhang},
\newblock \bibinfo{title}{A note on the {G}agliardo-{N}irenberg inequality in a
  bounded domain},
\newblock \bibinfo{journal}{Communications on Pure and Applied Analysis}
  \bibinfo{volume}{21} (\bibinfo{year}{2022}) \bibinfo{pages}{4013--4017}.
%Type = Article
\bibitem[{Eyre(1998)}]{eyre1998}
\bibinfo{author}{D.~J. Eyre},
\newblock \bibinfo{title}{Unconditionally gradient stable time marching the
  {C}ahn-{H}illiard equation},
\newblock \bibinfo{journal}{MRS Online Proceedings Library (OPL)}
  \bibinfo{volume}{529} (\bibinfo{year}{1998}) \bibinfo{pages}{39}.
  \DOIprefix\doi{10.1557/PROC-529-39}.
%Type = Book
\bibitem[{Ciarlet(2013)}]{ciarlet2013linear}
\bibinfo{author}{P.~G. Ciarlet}, \bibinfo{title}{Linear and Nonlinear
  Functional Analysis with Applications}, volume \bibinfo{volume}{130},
  \bibinfo{publisher}{Siam}, \bibinfo{year}{2013}.
%Type = Article
\bibitem[{Sch{\"o}berl(1997)}]{Schoberl-1997-netgen}
\bibinfo{author}{J.~Sch{\"o}berl},
\newblock \bibinfo{title}{{NETGEN} - {A}n advancing front 2{D}/3{D}-mesh
  generator based on abstract rules},
\newblock \bibinfo{journal}{Computing and Visualization in Science}
  \bibinfo{volume}{1} (\bibinfo{year}{1997}) \bibinfo{pages}{41--52}.
  \DOIprefix\doi{10.1007/s007910050004}.
%Type = Techreport
\bibitem[{Sch{\"o}berl(2014)}]{Schoberl-2014-ngsolve}
\bibinfo{author}{J.~Sch{\"o}berl}, \bibinfo{title}{C++11 {I}mplementation of
  {F}inite {E}lements in {NGS}olve}, \bibinfo{type}{{ASC} {R}eport 30/2014},
  {I}nstitute for {A}nalysis and {S}cientific {C}omputing, {V}ienna
  {U}niversity of {T}echnology, \bibinfo{year}{2014}.

\end{thebibliography}
	%------------------------------------------------------------------------------
\end{document}